\documentclass[11pt,leqno]{article}

\usepackage{amsthm,amsfonts,amssymb,amsmath,color}

\numberwithin{equation}{section}

\renewcommand\d{\partial}
\renewcommand\a{\alpha}
\renewcommand\b{\beta}
\renewcommand\o{\omega}
\newcommand\s{\sigma}
\renewcommand\t{\tau}
\newcommand\R{\mathbb R}\newcommand\N{\mathbb N}\newcommand\Z{\mathbb Z}
\newcommand\C{\mathbb C}

\def\g{\gamma}
\def\de{\delta}
\def\G{\Gamma}
\def\t{\tau}
\def\O{\Omega}
\def\th{\theta}
\def\k{\kappa}

\def\l{\lambda}

\def\epsilon{\varepsilon}
\def\e{\varepsilon}

\newcommand\br{\begin{rem}}
\newcommand\er{\end{rem}}
\newcommand\bp{\begin{pmatrix}}
\newcommand\ep{\end{pmatrix}}
\newcommand\be{\begin{equation}}
\newcommand\ee{\end{equation}}
\newcommand\ba{\begin{equation}\begin{aligned}}
\newcommand\ea{\end{aligned}\end{equation}}
\newcommand\ds{\displaystyle}
\newcommand\nn{\nonumber}

\newcommand{\TT}{{\mathbb T}}

\newcommand{\Id}{{\rm Id }}

\newcommand{\tr}{{\rm tr }}
\newcommand{\supp}{{\rm supp }}

\newcommand{\D}{{\mathcal D}}

\newcommand{\diag}{\text{\rm diag}}

\newcommand{\rank}{\text{\rm rank}}

\newtheorem{defi}{Definition}[section]
\newtheorem{theo}[defi]{Theorem}
\newtheorem{prop}[defi]{Proposition}
\newtheorem{lem}[defi]{Lemma}
\newtheorem{cor}[defi]{Corollary}
\newtheorem{rem}[defi]{Remark}
\newtheorem{ass}[defi]{Assumption}

\def\op{{\rm op} }
\def\ope{{\rm op}_\e^\psi }

\numberwithin{equation}{section}

\begin{document}

\title{Higher-order resonances and instability of high-frequency WKB solutions}

\author{Yong Lu\footnote{Mathematical Institute, Faculty of Mathematics and Physics, Charles University, Sokolovsk\'a 83, 186 75 Praha, Czech Republic, {\tt luyong@karlin.mff.cuni.cz}}}

\date{}

\maketitle

\begin{abstract}

This paper focuses on the destabilizing role of resonances in high-frequency WKB solutions. Specifically, we study higher-order resonances  associated with higher-order harmonics generated by nonlinearities. We give examples of systems and solutions for which such resonances generate instantaneous instabilities, even though the equations linearized around the leading WKB terms are initially stable, meaning in particular that the key destabilizing terms are not present in the data.

\end{abstract}








\setcounter{tocdepth}{2}
 \tableofcontents

\section{Introduction}

We study highly-oscillating solutions to hyperbolic systems based on Maxwell's equations. Considerable progress has recently been made in this line of research, especially following the works of Joly, M\'etivier and Rauch in the nineties (see for instance \cite{DJMR, JMR1, JMR2}, and \cite{Dumas1} for an overview and further references). The underlying physical problems deal with light-matter interactions.

 The specific systems under study here have the form
\be \label{1} \d_t U + \frac{1}{\e} A_0 U + \sum_{1 \leq j \leq d} A_j \d_{x_j} U = \frac{1}{\sqrt \e} B(U,U),\ee
and the data have the form
\be \label{2}
 U(0,x) = \Re e \, \big( a(x) e^{i k \cdot x/\e} \big) + \sqrt\e \varphi^\e(x).
 \ee
The small parameter $\e > 0$ is the wavelength of light. The initial wavenumber $k$ is a given vector in $\R^d.$ The constant matrix $A_0$ is non-zero and skew-symmetric. Its presence in the equations implies that the dispersion relation is non-homogeneous. This is a typical feature of systems describing light-matter interactions. The matrices $A_j$ are constant and symmetric. Explicit examples of such operators are given in \eqref{KG} below. We denote for any $\xi\in\C^d$:
$$
A(\xi):=\sum_{1 \leq j \leq d} A_j \xi_j.
$$

The data \eqref{2} are oscillating, with frequencies $O(1/\e)$ that are naturally of the same order of magnitude as the characteristic frequencies of the hyperbolic operator in \eqref{1}; this means that the light may propagate in the medium. The amplitude of the initial oscillations is $O(1)$ with respect to $\e.$

In \eqref{1}, the map $B$ is bilinear $\R^n \times \R^n \to \R^n.$ Its specific form derives from a phenomenological description of nonlinear interactions \cite{NM}, which may also include higher-order (order three, four, ...) interactions.

A key point here is the large prefactor $1/\sqrt \e$ in front of the nonlinearity. Since we will consider solutions of amplitude $O(1)$ with respect to $\e,$ this means that nonlinear effects play a role in the propagation of the initial oscillations in short time $O(\sqrt \e).$ In other words, the propagation is {\it not} weakly nonlinear in time $O(1).$

It has indeed been observed that the weakly nonlinear regime fails to describe nonlinear effects in time $O(1)$ for a number of physical systems: Maxwell-Bloch (by Joly, M\'etivier, and Rauch \cite{JMR2}), Maxwell-Euler (by Texier in \cite{em3}), and Maxwell-Landau-Lifshitz (by the author in \cite{MLL}).

For systems and data of the form \eqref{1}-\eqref{2}, a systematic study of resonances and stability of WKB solutions was given by Texier and the author in \cite{em4}. In particular, the article \cite{em4} contains a detailed account of how resonances may destabilize precise WKB solutions.

By WKB solutions we mean truncated power series in $\e$ which approximately solve \eqref{1}. Each term in the series is a trigonometric polynomial in $\th:=(k \cdot x - \o t)/\e,$ where $\o$ is an appropriate characteristic temporal frequency, in the sense that
\be \label{char} {\rm det}\,(-i\o+A(ik)+A_0)=0.\ee
 That is, a WKB solution is $U^a$ such that
 \be\label{suite}U^a = \sum_{n=0}^{2K_a} \e^{n/2} {\bf U}_{n}, \qquad {\bf U}_{n} = \sum_{p \in {\cal H}_n} e^{i p \th} U_{n,p}, \qquad K_a \in \Z_+, \,\, {\cal H}_n \subset \Z,\ee
 where {\it amplitudes} $U_{n,p}(t,x)$  are {\it not} highly-oscillating, and
 \be \label{wkb-new} \left\{ \begin{aligned} &\d_t U^a + \frac{1}{\e} A_0 U^a + \sum_{1 \leq j \leq d} A_j \d_{x_j} U^a = \frac{1}{\sqrt \e} B(U^a,U^a) + \e^{K_a} R^\e,\\
 & U^a(0,x)=U(0,x)+\e^K \psi^\e(x),\end{aligned}\right.
 \ee
 where $|R^\e|_{L^\infty}+|\psi^\e|_{L^\infty}$ is bounded uniformly in $\e.$ Parameters $K_a$ and $K$ describe the level of precision of the WKB solution $U^a.$

If $\o$ is a characteristic temporal frequency, satisfying \eqref{char}, then symmetries in the equations typically imply that some $p \o,$ with $p \in \Z,$ $p \neq 1,$ are characteristic as well, meaning
$${\rm det}\,(-ip\o+A(ip k)+A_0)=0.$$
We denote ${\cal H}_0:=\{p \in \Z, \,\, {\rm det}\,(-ip\o+A(ipk)+A_0)=0 \}$. Since $A_0 \neq 0,$ this set is typically finite. In \eqref{suite}, the leading term ${\bf U}_0$ appears as a sum over ${\mathcal H}_0.$ The coefficients $U_{0,p}$ are said to be {\it harmonics} of the leading coefficient ${\bf U}_{0}$ in the WKB solution.  We suppose ${\rm ker}\,(-ip\o+A(ipk)+A_0)$ is of dimension one with a generator $e_p$. Then $U_{0,p}$ satisfies the  so-called polarization condition for any $p\in {\cal H}_0$:
\be\label{intr-pola}
U_{0,p}(t,x) =g(t,x)\,e_p,\quad \mbox{for some scalar function $g(t,x)$}.
\ee

Resonances are frequencies that satisfy some functional relations involving the eigenvalues of the hyperbolic operator in \eqref{1} and the initial wavenumber $k,$ of the form
 \be \label{res} \l_j(\xi + p k) = \l_{j'}(\xi) + p \o, \quad \mbox{for some $\xi \in \R^d$},\ee
 where $\l_j$ and $\l_{j'}$ solve the dispersion relation.

The main result of \cite{em4} gives structural relations, involving the hyperbolic operator and the nonlinearity $B,$ which imply that arbitrarily small initial perturbations may be instantaneously amplified. Small perturbations means $K$ large in \eqref{wkb-new}. That is, no matter how large $K$ and $K_a$ are in \eqref{wkb-new}, the WKB solution $U^a$  may deviate instantaneously from the exact solution to \eqref{1}, under conditions put forward in \cite{em4}. In applications, WKB solutions are commonly used to simulate the interactions, since they satisfy model equations which are considerably cheaper to simulate than the original system \eqref{1} based on Maxwell. The result of \cite{em4} shows precisely how in some instances the WKB computations may fail to accurately describe the interactions, hence should {\it not} be used for simulations.

In \cite{em4}, it is assumed that there is no higher-order harmonics, such as $2 (\o,k),$ $3 (\o,k),$ etc., in the leading terms of WKB solutions. These higher-order harmonics are created by the nonlinearity. We relax this assumption here, and specifically focus on the destabilizing role played by resonances associated with higher-order harmonics. These resonances are related to \eqref{res} with $p = 2, 3,...$

Instead of building a complete theory, as in \cite{em4}, here we focus on one example, explicitly given in \eqref{KG} below, comprising coupled Klein-Gordon operators. Such operators were shown to derive from Euler-Maxwell in \cite{em1, em2,em4}; they were also the focus of article \cite{Germain-KG}. The nonlinearity in \eqref{KG} is tailored for the instability phenomenon that we want to observe; there is no doubt that the phenomenon is not limited to this specific form, but could occur for a variety of operators and bilinear terms of the form \eqref{1}.

Our point is that, for physical systems based on Maxwell's equations, {\it resonances associated with higher-order harmonics may destabilize precise WKB solutions}. We give a precise description of the destabilization process. This is a purely nonlinear mechanism: the higher-order harmonics are generated by the semilinear source terms; in particular, they are
not present in the equation at $t = 0.$ In our example, the equations linearized around the initial oscillations are indeed stable.

There is an analogy with recent work of Lerner, Morimoto and Xu \cite{LMX}, and later Lerner, Nguyen and Texier \cite{LNT}. These articles, \cite{LMX} and \cite{LNT}, study the phenomenon of loss of hyperbolicity, for which a model equation is $(\d_t + i t \d_x) u = 0:$ hyperbolicity holds at $t = 0,$ but strong instabilities occur for $t >0.$ Similarly, the initial linearized equations are stable here. Higher-order harmonics are $O(t),$ and generate instabilities. There is a form of degeneracy analogous to the one in $\d_t + i t  \d_x.$ As a result, the instability is slow to develop: it occurs in time $O(\e^{1/4}|\ln \e|^{1/2}),$ as opposed to time $O(\e^{1/2} |\ln \e|)$ in \cite{em4}.

One example of the interest of geometric optics for other research areas is the important similarity of concepts and tools with the construction of global solutions of small data for nonlinear dispersive equations. For instance, the transparency conditions are analogous to the null conditions introduced by Klainerman in \cite{kl}; the reduction to normal form allowed by the transparency property is analogous to the analysis of Shatah in \cite{shatah}. We refer to \cite{Lannes-Bourbaki} a survey for the connection between null conditions and transparency conditions.  Hence, there probably holds a strong link between the instability in nonlinear geometric optics and the possible non global existence results for nonlinear dispersive equations. However, it seems that more structures of the nonlinearities can be used in the study of nonlinear dispersive equations to obtain global existence of solutions with small initial data. For instance, Germain, Masmoudi and Shata in \cite{GMS1, GMS2, GMS3} introduced the \emph{space-time resonance} method to handle the situations where the normal form method cannot be used. It is not clear whether we can employ the arguments in \cite{em4} and in this paper to obtain some non global existence results for nonlinear dispersive equations, but it will be very interesting to look into it.

\subsection{Klein-Gordon systems}
For notational simplicity, we restrict to the following, one-dimensional system:
\begin{equation} \label{KG}
    \left\{ \begin{aligned}
    \d_t u + \left(\begin{array}{ccc} 0 & \d_x & 0 \\ \d_x & 0 & \e^{-1}\a_0\o_0 \\ 0 & - \e^{-1} \a_0 \o_0 & 0 \end{array}\right)  u & = \left(\begin{array}{c} 0 \\ \e^{-1/2} (u_3 + v_3) v_3 \\ 0 \end{array}\right), \\
 \d_t v + \left(\begin{array}{ccc} 0 & \theta_0 \d_x & 0 \\ \theta_0 \d_x & 0 & \e^{-1}  \o_0 \\ 0 & -\ \e^{-1}  \o_0 & 0 \end{array}\right)  v
 & = \left(\begin{array}{c} 0 \\ \e^{-1/2} (- u_2^2 + v_2^2) \\ 0 \end{array}\right),
     \end{aligned}
     \right.
    \end{equation}
   where $x \in \R,$ $U:=(u,v)\in\R^{3}\times \R^3$, and $\e$ is a small parameter. The constant $\theta_0$ is assumed to satisfy $0 < \theta_0 < 1.$ This means that the Klein-Godon operators have different velocities (as in the
operators derived from Euler-Maxwell, which feature one velocity equal
to the speed of light, and one velocity equal to the ratio of the
electronic thermal velocity to the speed of light). We assume $\o_0 > 0.$
We also assume that the masses are distinct; this means $\a_0 \neq 1$. We
assume $5/2< \a_0 < 3$ for two technical reasons. The first is to make sure that higher-order harmonics exist in the leading term of WKB solution; precisely, that is to guarantee that equation \eqref{o-k} has solution $(\o,k)$. The second reason introducing this condition on $\a_0$ is to reduce the number of resonances as in Section 5.3 of \cite{em4}. This allows us to focus on the resonances associated with higher-order harmonics.

We consider  highly  oscillating initial datum of the form:
\begin{equation} \label{id-KG}
  u(0,x) = \sqrt\e u_r(\e,x), \quad v(0,x) = v^0(x) e^{i k x/\e} + \left( v^0(x) e^{i k x/\e}\right)^*+\sqrt\e
  v_r(\e,x),
   \end{equation}
where for any fixed $\e>0$, we suppose sufficient Sobolev regularity for
$v^0,~u_r$ and $v_r$. The initial spatial wave number $k\in \R$ will be chosen such that
\eqref{und-phase} is satisfied. This initial datum will be chosen in \eqref{id-KG-pert} and in
Section \ref{sec:lower} in such a way that the slow instability
develops. The notation $a^*$ denotes the complex conjugate of $a$.

The system \eqref{KG} is symmetric hyperbolic, with bilinear source
term. For any fixed $\e>0$, the
local existence, uniqueness and regularity in smooth Sobolev spaces $H^s$ with $s>d/2$
are classical; however, the large nonlinear source term of order $O(1/\sqrt\e)$ causes the classical existence time to be  $O(\sqrt\e)$ (see for instance Chapter 7 of \cite{M}). We  study the system in high frequency
limit $\e\to 0$.

\subsection{Structure of the paper}

This paper is organized as follows. We mention previous results in \cite{JMR2} and \cite{em4} in Section \ref{sec:background}; we
compare these with our results in Section \ref{sec:compare}. In Section \ref{sec:results}, we give some
notations and state our main results. Section 3 and 4 comprise the
proofs. In Section 3, we construct a WKB solution, and in Section 4 we
show its instability. In Appendix A we recall some concepts about pseudodifferential operators.  The most technical parts of the instability proof are given in Appendix B and Appendix C.


\subsection{Background}\label{sec:background}

In this section,  we describe briefly the  main results in \cite{JMR2} and \cite{em4}.

\subsubsection{WKB solution, weak and strong transparency, stability}

We rewrite the system \eqref{1} as
\begin{equation}\label{mb}
\d_t U + \frac{1}{\e} A_0 U + \sum_{1 \leq j \leq d} A_j \d_{x_j} U -\frac{1}{\sqrt\e}B(U,U)=0.
\end{equation}
 We assume that the spectral decomposition
$$
A(\xi)+A_0/i=\sum_{j=1}^J \l_j(\xi)\Pi_j(\xi)
$$
is smooth, meaning that  the eigenvalues $\l_j$ and  the
eigenprojectors $\Pi_j$ are smooth. The notations $\l_j$ and $\Pi_j$ are used temporarily in Introduction.

We study the stability of WKB approximate solutions of form \eqref{suite}, in particular with amplitude $O(1)$.  We plug \eqref{suite} into \eqref{mb}, then the left-hand side
 has the form $ \sum_{n\geq -2}\e^{n/2} \Phi_n(t,x,\th) $ with
$$
\Phi_{n}:=\big(\d_t+A(\d_x)\big){\bf U}_n+\big(-\o\d_\th
+A(k\d_\th)+A_0\big)
 {\bf U}_{n+2}-\sum_{n_1+n_2=n+1}B({\bf U}_{n_1},{\bf U}_{n_2}).
$$
Solving \eqref{mb} amounts to solve $\Phi_n=0$ for all $n\in\Z,~n\geq -2$.
 This is generally not possible because
there are infinity of $n$. At least, we can solve \eqref{mb} approximately by solving $\Phi_n=0$ up to some nonnegative order:
if we solve $\Phi_n=0$ up
to some order $N\geq 0$, the WKB solution $U^a$ solves \eqref{mb}
with a remainder of order $O(\e^{(N+1)/2})$ which goes to zero.  This is the typical way to construct a WKB
solution (see Section \ref{sec:wkb} for more details).
\begin{defi} We say $U^a$ in \eqref{suite} to be a
WKB solution to \eqref{mb} of order $N$ provided
$\Phi_0=\Phi_1=\cdots=\Phi_{N}=0.$
\end{defi}
Remark that the leading term ${\bf U}_0=\sum_{p\in{\mathcal H}_0}e^{ip\th}U_{0,p}$ plays a special role in the WKB expansion.
Indeed, a WKB solution $U^a$ is approximated by its leading term
$$
|U^a-{\bf U}_0(t,x,(kx-\o t)/\e)|=O(\sqrt\e)\to 0.
$$
Considering the initial datum \eqref{2}, the initial values of $U_{0,p}$ are chosen as
$$
U_{0,1}(0,x)=a(x),\quad U_{0,-1}(0,x)=\big(a(x)\big)^*, \quad
U_{0,p}(0,x)=0,~\mbox{for}~p\not\in \{-1,1\}.
$$
In the leading term,  $U_{0,1}$ and $U_{0,-1}$ are said to be \emph{fundamental harmonics (or phases)}, while $U_{0,p}$ with $|p|>1$ are said to be \emph{higher-order harmonics (or phases)}.

We would like to show that the
high-order harmonics can destabilize the WKB solution, in spite
of their initial values being null. To achieve this, we need the higher-order harmonics to be non-null when $t>0$. The
higher-order harmonics are generated by the nonlinearity
$B$ and the fundamental harmonics. In solving $\Phi_n=0$, there arises
an equation of the form
$$
\big(\d_t+\Pi(3\tilde\b)A(\d_x)\Pi(3\tilde\b)\big)
U_{0,3}=B(U_{0,1},U_{1,2})+\cdots,
$$
where
$
\tilde\b:=(\o,k),$ $\Pi(\t,\xi)$  denotes the orthogonal projector onto $\ker
\big(-i\t +A(i\xi)+A_0\big)$ and $U_{1,2}$ solves
$$
\big(-2\o+A(2k)+A_0\big)U_{1,2}=B(U_{0,1},U_{0,1}).
$$
In our context, $\big(-2\o+A(2k)+A_0\big)$ is an invertible matrix, then $U_{1,2}$ can be written as a bilinear form of $U_{0,1}$. This
gives
\be\label{nonzero-u03}
\big(\d_t+\Pi(3\tilde\b)A(\d_x)\Pi(3\tilde\b)\big)
U_{0,3}=Q(U_{0,1})+\cdots,
\ee
where $Q$ is cubic in $U_{0,1}$. Then a non-null solution
$U_{0,3}(t,\cdot)$ is expected for $t>0$.

In \cite{JMR2}, Joly, M\'etivier and Rauch  introduced the
\emph{weak transparency} condition:\\
{\bf Weak transparency} For any $p,p_1\in \Z$ and any $U,~V\in \C^n$, one has
\begin{equation}\label{wtrans0}
\big|\Pi(p_1\tilde\b)B\big(\Pi((p_1-p)\tilde\b)U,\Pi(p\tilde\b)V\big)\big|=0.
\end{equation}
Under this weak transparency assumption, we can show the existence of WKB solution of any order in time
$O(1)$ (see Proposition 6.19 in \cite{em4}).

Concerning the stability of the WKB solution, the following \emph{strong
transparency} condition was introduced in \cite{JMR2}:\\
{\bf Strong transparency} There exists a constant $C$ such that for any $p\in \Z$,  $1\leq j,j'\leq J $, $\xi\in \R^d$ and  $U,~V\in \C^n$, one has
\begin{equation}\label{sttrans0}
\big|\Pi_j(\xi+p k)B\big(\Pi(p\tilde\b)U,\Pi_{j'} (\xi)V\big)\big|\leq C|\l_j(\xi+p k)-\l_{j'}(\xi)-p\o|\cdot|U|\cdot|V|.
\end{equation}
Remark  that the strong transparency condition  is  strictly  stronger than
the weak transparency condition.

Typically, strong transparency implies stability, via normal form
reductions (see \cite{JMR2,MLL,em4}). For \eqref{KG}-\eqref{id-KG} we show that the weak transparency is
satisfied and a WKB solution can be constructed.  However,  the strong
transparency is not satisfied. This implies instabilities of the WKB solution.

\subsubsection{Absence of strong transparency and instability}
In \cite{em4}, Texier and the author consider systems of the form
\eqref{1} for which the weak transparency condition  is satisfied while the strong transparency
condition  is not. This indicates that,
approximate solutions can be constructed through WKB expansion, but
the normal form reduction method cannot be applied to show the stability of
such WKB solutions.

The absence of strong transparency means that there
exists $(j,j',p)$ such that \eqref{sttrans0} is not satisfied.
Denote $J_0$  the set containing all such index $(j,j',p)$ and $R_{jj',p}$ the $(j,j',p)$-resonant set defined as
$$R_{jj',p}:=\{\xi\in\R^d,\,\l_j(\xi+p k)=p \o+\l_{j'}(\xi)\}.$$
If $R_{jj',p}$ is empty, by the regularity of $\l_j$ and $\Pi_j$, $j=1,\cdots, J$,
condition \eqref{sttrans0} is satisfied for the index
$(j,j',p)$. Then for any $(j,j',p)\in J_0$, $R_{jj',p}$ is not
empty, and the following quantity is well defined:
\be\label{def:Gamma} \Gamma:=\sup_{(j,j',p)\in J_0}|g_p(0,x_{p})|^2\sup_{\xi\in
R_{jj',p}} \tr\Big(\Pi_j(\xi+p k)B(\vec
e_p)\Pi_{j'}(\xi)B(\vec e_{-p})\Pi_j(\xi+p k)\Big), \ee where $g_p$ come from the polarization condition \eqref{intr-pola}
and $x_p$ is the point where  $|g_p(0,\cdot)|$ admits its maximum.

In \cite{em4}, it is shown that the stability  of  the WKB solution
is determined by the sign of $\Gamma$:\\
If  $\Gamma<0$, the perturbation  system is symmetrizable and the WKB solution is stable.\\
If $\Gamma>0$, it is shown that the WKB solution is unstable. The instability analysis consists first in reducing, via
normal form reductions, the problem to the study of {\it interaction
systems} of the form
\be\label{intr-trun-eq}
\d_t u+\frac{1}{\sqrt\e}\op_\e (M_0)u=f,
\ee
where $f$ is small of order $O(\e^{\kappa})$ for some $\kappa>0$ and contains in particular nonlinear terms, and $\op_\e(M_0)$ is the semiclassical pseudo-differential operator associated with a matrix-valued symbol $M_0$  which is of order zero, essentially independent of $t$ and has the form
$$
M_0 := \bp i\l_j-ip\o &-\sqrt\e b_{jj'} \\-\sqrt\e
b_{j'j}&i\l_{j'}\ep$$ with $$ b_{jj'}=\Pi_{j}(\xi+p
k)B(U_{0,p}(0,x))\Pi_{j'}(\xi),\quad
b_{j'j}=\Pi_{j'}(\xi)B(U_{0,-p}(0,x))\Pi_{j}(\xi+p k)
$$
the \emph{interaction coefficients} associated with resonance $\l_j(\xi+p k)=p\o+\l_{j'}(\xi)$.

The analysis of the interaction systems relies on a Duhamel
representation formula introduced by Texier in \cite{dua}. The analysis of
\cite{dua} shows that a solution operator for the interaction systems can be
constructed as a pseudo-differential operator with leading symbol the
symbolic flow of \eqref{intr-trun-eq}, defined by
 \be\label{intr:flow1} \d_t  S_0(\t;t)+\frac{1}{\sqrt\e}
M_0 S_0(\t;t) =0,\qquad S_0(\t;\t)=\Id. \ee

 In fact, the index in \eqref{def:Gamma} is the maximal real part of ${\rm sp}\, (M/\sqrt\e)$.
 The ordinary differential equation \eqref{intr:flow1} is autonomous and has solution
\be\label{intr:def-s0}
S_0(\t;t)=\exp\big(-M_0(t-\t)/\sqrt\e\big).\ee
From this explicit expression, a good upper bound for $S_0$ can be obtained. By {\it good} upper bound
we mean an upper rate of growth that is arbitrarily close to a lower
rate of growth. Indeed, $\Gamma > 0$ implies that some eigenvalue of $M_0$
has positive real part. Then $S_0$ is exponentially growing in time. Via
the Duhamel Lemma of Appendix C, an instability result for the WKB
solution ensues.

\subsection{Higher-order resonances and instability}\label{sec:compare}
In this section, we compare our results to the results of \cite{JMR2}  and \cite{em4}. In particular, we point out the main new difficulties compared to \cite{em4}.

\subsubsection{Transparency and loss of hyperbolicity}

In \cite{JMR2} and \cite{em4}, there exist interaction coefficients which
are non-transparent (meaning that the strong transparency condition is not satisfied) both initially and for positive time. This implies a loss of hyperbolicity around a resonance, both initially and for positive time. A simple example model is the non-degenerate Cauchy-Riemann equation
$$
\d_t u+\frac{i \d_x}{\sqrt\e} u=0,\quad u(0)=u_0,
$$
of which the solution is
$$\displaystyle{\hat u(t,\xi)=\exp\Big(\frac{t\xi}{\sqrt\e}\Big)\hat
u_0(\xi)}.$$
Then the instability is recorded in time $O(\sqrt\e
|\ln \e|)$ for frequencies $O(1)$.

 The present situation ia analogous to the  degenerate Cauchy-Riemann equation
$$
\d_t u+\frac{it \d_x}{\sqrt\e} u=0,\quad u(0)=u_0.
$$
When $t=0$, the equation is hyperbolic. When $t>0$, the hyperbolicity is lost.  The solution is
$$\displaystyle{\hat u(t,\xi)=\exp\Big(\frac{t^2\xi}{2\sqrt\e}\Big)\hat u_0(\xi)}.$$
The instability develops in time $O(\e^{1/4}|\ln\e|^{1/2})$ for frequencies $O(1)$.

\subsubsection{Stability index}

For system \eqref{KG}-\eqref{id-KG}, we will show that for any non-transparent index $(j,j',p)\in J_0$, there holds
$U_{0,p}(0,\cdot)= 0$.  Then the stability index $\Gamma$ defined in \eqref{def:Gamma} is zero. This case is not covered by the analysis of \cite{em4}. Here instability relies on condition \be\label{def:gamma1}
\sup_{(j,j',p)\in J_0}\sup_{\xi\in R_{jj',p}} |\d_t
g_p(0,y_{p})|^2\tr\Big(\Pi_j(\xi+p k)B(\vec
e_p)\Pi_{j'}(\xi)B(\vec e_{-p})\Pi_j(\xi+p k)\Big)>0.\ee
Recall that $g_p$ is the function introduced in the polarization condition \eqref{intr-pola}. The stability index $\Gamma_1$ is defined as the square root of the left-hand side of \eqref{def:gamma1} where $y_p$ is the point at which $|\d_t
g_p(0,\cdot)|$ admits its maximum. In our analysis, we have $p\in \{-3,3\}$ corresponding to  higher-order harmonics. Parameter $\Gamma_1$ can be explicitly calculated, in terms of the system and the datum.
 Indeed, in the WKB expansion, we find that $g_{3}$ satisfies a transport equation with a cubic source term in $g_{1}$ and $g_{-1}$ (see \eqref{nonzero-u03}). Then we can calculate $\d_t g_3(0,\cdot)$ through the equation and initial data $g_{\pm}(0,\cdot)$ which can be obtained from \eqref{id-KG}.

\subsubsection{Bounds for the symbolic flow}\label{sec:intr-symflow}

The analysis here relies partly  on the Duhamel representation formula introduced in \cite{dua}. Contrary to \cite{em4},
here we need to consider the upper bound for a symbolic flow which
is solution to a non-autonomous system of the following form (see
\eqref{eq:tildeS0}): \be\label{intr:sys-flow-2} \d_t \widetilde
S_0(\t;t)+\frac{1}{\e^{3/4}} M_0(t) \widetilde S_0(\t;t) =0,\qquad
\widetilde S_0(\t;\t)=\Id. \ee

We remark that in Section \ref{sec:intr-symflow},  the time $t$ is rescaled by $\e^{1/4}$ so that the instability is now expected in time $O(|\ln\e|^{1/2})$.

The  matrix $M_0(t)$ is of the form
(the index $(j,j',p)$ is chosen accordingly):
$$
M_0(t):=\bp i\l_j-ip\o &-\e^{3/4}t \,\tilde b_{jj'}
\\-\e^{3/4}t \,\tilde b_{j'j}&i\l_{j'}\ep,
$$
 where
$$ \tilde b_{jj'}=\Pi_{j}(\xi+p k)B(\d_tU_{0,p}(0,x))\Pi_{j'}(\xi),\quad
\tilde b_{j'j}=\Pi_{j'}(\xi)B(\d_tU_{0,-p}(0,x))\Pi_{j}(\xi+p k).
$$
The two
blocks $t \,\tilde b_{jj'}$ and $t\,\tilde b_{j'j}$ come from the
interaction coefficients
$$ \Pi_{j}(\xi+p k)B(U_{0,p}(\e^{1/4}t,x))\Pi_{j'}(\xi),\quad
\Pi_{j'}(\xi)B(U_{0,-p}(\e^{1/4}t,x))\Pi_{j}(\xi+p k)
$$
and the Taylor formula with respect to $t$:
$$U_{0,p}(\e^{1/4}t)=U_{0,p}(0)+\e^{1/4}t\d_t U_{0,p}(0)+O(\e^{1/2}t^2),$$
where $U_{0,p}(0,\cdot)= 0$ and the remainder $O(\e^{1/2}t^2)$ contributes a term of order $O(\e^{1/4})$ and is neglected (see \eqref{eq:widetildeS}-\eqref{eq:tildeS0}).

The goal is to obtain a good upper bound for $\widetilde S_0$ of the form \be\label{intro-upper}
|\widetilde S_0(\t;t)|\leq C \exp\big((t^2-\t^2)\g^+/2\big). \ee  Here
$\g^+$ is such that $\g^+ t$ is the maximum of the real parts of the
eigenvalues of $-M_0(t)/\e^{3/4}$. By direct calculation, the
eigenvalues of $M_0(t)$ are (see also \eqref{eig-m}):
\be\label{intro-ev}
\begin{aligned} &i\l_{j}-ip\o,\quad i\l_{j'},\\
&\nu_{\pm}:=\frac{i}{2} \big(\l_{j}-p\o + \l_{j'} \big) \pm
\frac{1}{2}
  \big( 4 \e^{3/2}t^2 \tr\,(\tilde b_{12} \tilde b_{21}) - (\l_{j}-p\o - \l_{j'})^2
  \big)^{1/2}.
 \end{aligned}
\ee This implies that, at resonances $\l_{j}-p\o=\l_{j'}$, the real
parts of the eigenvalues of $-M_0(t)/\e^{3/4}$ admit their maximum
$t\, \sqrt{\tr\,(\tilde b_{12} \tilde b_{21})}$ provided $\tr\,(\tilde
b_{12} \tilde b_{21})>0$. The positivity of $\tr\,(\tilde b_{12}
\tilde b_{21})$  is guaranteed by the positivity of  $\Gamma_1$. Precisely, $\g^+$ is the maximum of $\sqrt{\tr\,(\tilde b_{12}
\tilde b_{21})}$ over a sufficient small neighbourhood of the resonant sets due to the localization (see Section \ref{s-f loc}).

We say upper bound \eqref{intro-upper} is good because it has almost the same growth
rate as the lower bound that we can obtain, which is $\exp\big((t^2-\t^2)\g^-/2\big)$ with $\g^-$  sufficiently close to $\g^+$.

 In addition to the difficulties already present in \cite{em4}: fast
oscillations $O(\e^{-3/4})$ and small distance $O(\e^{3/4})$ between
resonances and crossing points of the eigen-modes of $M_0(t)$, the issue here is that the equation
\eqref{intr:sys-flow-2} is non-autonomous, implying that we do not have
an explicit formula like \eqref{intr:def-s0} for the
solution $\widetilde S_0$.  In particular, the argument in \cite{em4}  to show the uniform bound for $S_0$ cannot be
applied.

To show upper bound \eqref{intro-upper}, the idea is to
diagonalize $M_0(t)$, wherever possible:
$$
M_0(t)=\sum_{j} \gamma_j(t) \Theta_j(t), \quad \mbox{$\g_j$ are
eigenvalues, $\Theta_j $ are  eigenprojectors}.
$$
Remark that notations $\g_j$ and $\Theta_j $ are used temporarily in Introduction.
Applying $\Theta_j (t)$ onto \eqref{intr:sys-flow-2} gives
$$
\d_t \big(\Theta_j (t)\widetilde S_0\big)+\frac{1}{\e^{3/4}}\g_j(t)
\big(\Theta_j (t)\widetilde S_0\big) =\big(\d_t\Theta_j (t)\big)\widetilde
S_0,\quad \big(\Theta_j \widetilde S_0\big)(\t;\t)=\Theta_j (\t).
$$
Then we can write an explicit formula:
$$
\begin{aligned}
\Theta_j (t)\widetilde
S_0(\t;t)=&\exp\left(-\e^{-3/4}\int_\t^t\g_j(t')dt'\right)\Theta_j (\t)\\&+\int_\t^t
\exp\left(-\e^{-3/4}\int_{t'}^t\g_j(s)ds\right)
\big(\d_t\Theta_j (t')\big)\widetilde S_0(\t;t')dt'.
\end{aligned}
$$
It is shown that  the maximum of the real parts of
$\e^{-3/4}\g_j$ for all $j$ is $\g^+ t$.  Then
$\e^{-3/4}\big(\g_j(t)+\g_j(t)^*\big)/2\leq \g^+ t$ and
\be\label{intro-slow-ps}
\begin{aligned}
&|\Theta_j (t)\widetilde S_0(\t;t)|\leq
\exp\big((t^2-\t^2)\g^+/2\big)|\Theta_j (\t)| \\&~~~~~~~~+\int_\t^t
\exp\big((t^2-t'^2)\g^+/2\big) |\big(\d_t\Theta_j (t')\big)\widetilde
S_0(\t;t')|dt'. \end{aligned}\ee In the context of
\eqref{KG}-\eqref{id-KG}, we find $\d_t \Theta_j (t)\leq C(1+1/t) $. Then
for $t$ large, $\d_t \Theta_j (t) $ is bounded. The problem is that for $t$ near $0$, $\d_t \Theta_j (t)$ is unbounded and is of order $1/t$, which implies that the integral on the right-hand side of \eqref{intro-slow-ps} is not well defined when taking $\t=0$.  As in the proof of Lemma \ref{lem:ar-res}, we overcome this difficulty by introducing the
following change of variable for some small $c_0$:
$$
\widetilde S_1(\t;t):=\widetilde S_0(\t;t+c_0).
$$
Then for $\widetilde S_1$, we can obtain a similar inequality as
\eqref{intro-slow-ps}, in which the term $\d_t \Theta_j (t')$ is replaced by $\d_t \Theta_j (t'+c_0)$ which becomes
uniformly bounded with bound $C/c_0$. For small time in $[0,c_0]$, it is enough to use a rough estimate (Lemma \ref{rough-est}) for $\widetilde S_0$. As we remarked right after \eqref{intr:sys-flow-2}, the instability is expected in time $O(|\ln\e|^{1/2})$, so $[0,c_0]$ is indeed a small time interval.

After considering $\Theta_j  \widetilde S_0$ for all $j$, by the identity $\sum_{j}\Theta_j =\Id$ we finally obtain
$$
|\widetilde S_0(\t;t)|\leq C \exp\big((t^2-\t^2)\g^+/2\big)
\exp(Ct),\quad \mbox{for some constant $C>0$}.
$$
We remark that, for the time we consider of order $O(|\ln\e|^{1/2})$,
the term $\exp(Ct)$ is negligible compared to the main growth term
$\exp\big(t^2\g^+/2\big)$.

 For the case where $M_0(t)$ is not
diagonalizable, we show  there exists an invertible matrix $P$
which is independent of $t$, and  $|P|+|P^{-1}| \leq C(c_0)$ such that $\big|P
M_0(t)P^{-1}+\big(PM_0(t)P^{-1}\big)^*\big|/2\leq c_0\,t\e^{3/4}$ with
$c_0$ small and $C(c_0)$ a constant bounded for $c_0$ away from zero (we can choose $c_0$ as small as we want; here we only need
to fix $c_0$ such that $c_0\leq \g^+$). Then for the new unknown
$\widetilde S_0^{(1)}:=P\widetilde S_0$ which satisfies
$$
\d_t \widetilde S_0^{(1)}+\frac{1}{\e^{3/4}} \big(P
M_0(t)P^{-1}\big) \widetilde S_0^{(1)} =0,\qquad \widetilde
S_0^{(1)}(\t;\t)=P,
$$
we have $$|\widetilde S_0^{(1)}|\leq |P|\exp\big(c_0
(t^2-\t^2)/2\big).$$
 This implies
$$|\widetilde S_0(\t;t)|\leq |P||P^{-1}|\exp\big(c_0(t^2-\t^2)/2\big)\leq C
\exp\big((t^2-\t^2)\g^+/2\big).$$
Rigorous arguments are given in Appendix B.

\section{Description of the results}\label{sec:results}

In this paper, we focus on one spatial dimension $d=1$. However, we still use $d$ on
some occasions, when it is useful to stress the dependence on the dimension. If there is no specific definition, $C$ denotes a
constant independent of $(x,\xi,t,\t,\e,c_0)$ where $c_0$ is a small constant introduced in Section B.2.2 and fixed in Section B.2.7. However the value of $C$ could change from line to line.

\subsection{Notations}

  We introduce the notations
   \be\label{notation-lm} L(\o_0, \d) := \d_t + \left(\begin{array}{ccc} 0 & \d_x & 0 \\ \d_x & 0 & \a_0\o_0 \\ 0 & - \a_0\o_0 & 0 \end{array}\right),
   \quad  M( \o_0, \d) := \d_t + \left(\begin{array}{ccc} 0 & \theta_0 \d_x & 0 \\ \theta_0 \d_x & 0 & \o_0 \\ 0 & -\o_0 & 0
   \end{array}\right),\ee
where $\d=(\d_t,\d_x)$. Then the system \eqref{KG} of coupled
Klein-Gordon systems can be written as
\begin{equation} \label{coupled-KG}
    \left\{ \begin{aligned}
     L(\frac{ \o_0}{\e},\d) u & = \frac{1}{\sqrt \e} F(u + v,v), \\
     M(\frac{ \o_0}{\e}, \d) v & = \frac{1}{\sqrt \e} (G(u,u) + H(v,v)),
     \end{aligned}
     \right.
    \end{equation}
  where $F, G, H: \R^3 \times \R^3 \to \R^3$ are symmetric bilinear forms defined for any $u=(u_1,u_2,u_3)$ and $v=(v_1,v_2,v_3)$ as
\be\label{def-F-G-H} F(u,v):=(0,u_3v_3,0),\quad
G(u,v):=(0,-u_2v_2,0),\quad H(u,v):=(0,u_2 v_2,0). \ee

The characteristic varieties  are the sets of
time-space frequencies that define plane-wave solutions of $L$ and $M$:
$$
   \begin{aligned} \mbox{Char}\, L&  := \{ (\t,\xi) \in \R \times \R, \mbox{det}\, L(\o_0, -i \t, i \xi) = 0\}, \\
    \mbox{Char}\, M & := \{ (\t,\xi) \in \R \times \R, \mbox{det}\, M(\o_0, -i\t, i \xi) = 0\}.
  \end{aligned}
$$
 They both admit global smooth parameterizations, by $\{0, \pm \l\}$ and $\{0, \pm \mu\}$ respectively, where
 \begin{equation} \label{kg-modes}
  \l(\xi) := \sqrt{\a_0^2\o_0^2 + |\xi|^2}, \qquad \mu (\xi):= \sqrt{\o_0^2 + \theta_0^2 |\xi|^2}.
 \end{equation}
For any $ (\t, \xi)\in\R\times\R,$ we denote by $P(\t,\xi)$ and  $Q(\t,\xi)$ the
projectors onto the kernel of $L( \o_0, -i\t,i\xi)$ and $M( \o_0,
-i\t, i\xi),$ respectively. Then we have the following smooth
spectral decompositions: \be\label{spe-dec}\begin{aligned}
&L(\o_0,0,i\xi):=i\l(\xi)P_+(\xi)-i\l(\xi)P_-(\xi)+0\cdot P_0(\xi),\\
&M(\o_0,0,i\xi):=i\mu(\xi)Q_+(\xi)-i\mu(\xi)Q_-(\xi)+0\cdot
Q_0(\xi),\end{aligned}\ee where the eigenvalues are given by
\eqref{kg-modes} and the eigenprojectors are
\be\begin{aligned}
&P_+(\xi):=P\big(\l(\xi),\xi\big),\quad P_-(\xi):=P\big(-\l(\xi),\xi\big),\quad P_0(\xi):=P\big(0,\xi\big), \\
&Q_+(\xi):=Q\big(\mu(\xi),\xi\big),\quad
Q_-(\xi):=Q\big(-\mu(\xi),\xi\big),\quad
Q_0(\xi):=Q\big(0,\xi\big).\end{aligned} \nn\ee

Given a characteristic phase $ (\t, \xi) \in \mbox{Char}\,L,$
given $p \in \Z,$ the phase $(p\t, p\xi)$ belongs to $\mbox{Char}\, L$ if
and only if $p \in \{-1, 0, 1\}.$ The same is true of $\mbox{Char}\,
M.$ Also, by choice of $\theta_0$ and $\a_0$, the intersection
$\mbox{Char}\, L \, \cap \, \mbox{Char}\, M$ is equal to $\{ (0,
\xi),~\xi\in\R\}.$

\subsection{Statement of the results}

For initial datum \eqref{id-KG}, we choose  $k\neq 0$ such that for
some $\o\neq 0$, there holds \be\label{und-phase} \tilde\b=(\o,k)\in
\mbox{Char}\, M, \qquad 3\tilde\b=(3\o,3k)\in \mbox{Char}\, L .\ee By
\eqref{kg-modes}, equation \eqref{und-phase} amounts to
\be\label{o-k}
k^2=\big(1-\frac{\a_0^2}{9}\big)(1-\th_0^2)^{-1}\o_0^2,\quad
\o^2=k^2+\frac{\a_0^2}{9}\o_0^2.\ee We choose initial amplitude
$v^0$ satisfying the polarization condition:
\be\label{intro-pola}v^0\in \ker
  M(\o_0,-i\o,ik).\ee We suppose the regularity $v^0\in H^s$ with
  $s$ sufficient large as in Remark \ref{choi-s}.

  \medskip

In Section \ref{sec:wkb}, we  show that the weak transparency
condition is satisfied, then we construct an approximate solution
by WKB expansion:
\begin{prop}\label{prop:wkb}
Under the choice of $(\o,k)$ as in \eqref{und-phase}-\eqref{o-k},
the polarization condition \eqref{intro-pola} and the regularity
assumption $v^0\in H^s$ with $s$ large,  for any $K_a$,  there
exists $(u^a,v^a)$ that solves
\begin{equation} \label{sys:app}
    \left\{ \begin{aligned}
     L(\frac{ \o_0}{\e},\d) u^a & = \frac{1}{\sqrt \e} F(u^a + v^a,v^a)+\e^{K_a}r_1^\e, \\
     M(\frac{ \o_0}{\e}, \d) v^a & = \frac{1}{\sqrt \e} (G(u^a,u^a) +
     H(v^a,v^a))+\e^{K_a}r_2^\e,\\
 u^a(0,x) &= \sqrt\e
 u^\e_r(0,x),\\  v^a(0,x) &= \Re e\left(v^0(x) e^{i k x/\e}\right) + \sqrt\e
 v^\e_r(0,x),
     \end{aligned}
     \right.
    \end{equation}
in some time interval $[0,\tilde T]$ with
$\tilde T>0$ independent of $\e$, and for $j=1,2$,
    $$r_j^\e(t,x)=R_j(t,x,\frac{k
x-\o t}{\e}),\quad R_j(t,x,\th)\in L^\infty
\big([0,\tilde T]_t,H^1(\TT_{\th},H^{s-K_a}(\R_x))\big).$$
 Moreover,
$(u^a,v^a)$ has the following expansion:
\be\begin{aligned}&u^a=\Re e\left( u_{03}(t,x) e^{3i (k x-\o
t)/\e}\right) +\sqrt\e u_r^\e,\quad  v^a=\Re e\left(v_{01}(t,x) e^{i (k x-\o   t)/\e}\right) +\sqrt\e v_r^\e,
 \end{aligned}\nn\ee
where the leading amplitudes have  initial data $$u_{03}(0,x)=0,\quad
v_{01}(0,x)=v^0(x).$$  The correctors are of the form
$$\big(u^\e_r(t,x),v^\e_r(t,x)\big)=\left(U^a_r\big(t,x,\frac{k
x-\o t}{\e}\big),V^a_r\big(t,x,\frac{k x-\o t}{\e}\big)\right) $$ with
$$(U^a_r,V^a_r)(t,x,\th)\in L^\infty
\big([0,\tilde T]_t,H^1(\TT_{\th},H^{s-K_a}(\R_x))\big).$$
\end{prop}
\begin{rem}\label{choi-s}
Precisely, we choose the Sobolev regularity index $s>d/2+K_a+(d+2)+(q_0+3)/4$ according
to the $H^{s-1}$ estimate \eqref{est-g} of $\d_t g$, the need for
the smallness of  the right-hand side of \eqref{cho-s}, and the estimate for $\d_x^\a S_q, ~|\a|\leq d+1,~q\leq q_0$ where $S_q$ defined as in \eqref{sq} and $q_0$ is the order for the expansion in constructing solution operator in Appendix C (see \eqref{def:zeta}).
\end{rem}
In Section \ref{sec:slow-ins},  we show the WKB solution $(u^a,v^a)$ obtained in Proposition \ref{prop:wkb} is  unstable. We consider  initial data of the form
  \begin{equation} \label{id-KG-pert}
  u(0) = \sqrt\e u^\e_r(0)+\e^K \phi_1^\e, \quad v(0) = \Re e\left(v^0(x) e^{i k  x/\e} \right)+\sqrt\e
 v^\e_r(0)+ \e^K \phi_2^\e
   \end{equation}
corresponding to small perturbations of the WKB solution of Proposition \ref{prop:wkb}.
\begin{theo} \label{theorem1}
There exists $\e_0>0$ such that for any $0<\e<\e_0$, for some
initial perturbations satisfying $\ds{\sup_{0<\e<\e_0}
|(\phi_1^\e,\phi_2^\e)|_{L^1\cap L^\infty}<\infty},$  the solution $(u,v)$
to \eqref{KG} issued form the initial datum \eqref{id-KG-pert} is
unstable, in the following two senses:
\begin{itemize}

 \item  for any $K_a +1/4 > K > d/2+1/4,$ there exists a unique solution $u \in C^0([0,T_0 \e^{1/4} |\ln \e|^{1/2}],H^{s_0})$
  for some $s_0>d/2$ and all $\e$-independent $T_0<T_0^*$ where
  \be\label{def:Tstar}T_0^*:=\sqrt{2(K-d/2-1/4)/\Gamma_1}\ee with $\Gamma_1$  precisely given in
\eqref{def:gamma1-0}. Moreover,  for any $\kappa_0 > d/2+1/4,$
  there holds for $T_0$ close to $T_0^*$:
   \be\label{est:th3}\sup_{ \begin{smallmatrix} 0 < \e < \e_0 \\ 0 \leq t \leq T_0  \e^{1/4} |\ln \e|^{1/2} \end{smallmatrix}} \e^{-\kappa_0} | (u - u^a,v-v^a)(t)|_{L^2} =\infty;\ee

 \item for any $K_a + 1/4 > K > 1/4$ and solution $u\in L^\infty([0, T_1  \e^{1/4} |\ln \e|^{1/2}] \times \R)$
for any $\e$-independent $T_1<T_1^*$ where \be\label{time-ins}T_1^*:=\sqrt{2(K-1/4)/\Gamma_1},\ee   for any $\kappa_1 > 1/4$ there holds for $T_1$ close to $T_1^*$:
\begin{equation} \label{bd:insta-theo}
\sup_{ \begin{smallmatrix} 0 < \e < \e_0 \\ 0 \leq t \leq T_0  \e^{1/4} |\ln \e|^{1/2} \end{smallmatrix}}\e^{-\kappa_1} \big|(u-u^a,v-v^a)(t) \big|_{L^2 \cap
L^\infty}=\infty.
\end{equation}
\end{itemize}

\end{theo}

\medskip

  We make two remarks about Theorem \ref{theorem1}:
\begin{itemize}

  \item  first remark: We can think of $K_a$ being equal to $K$ and large. Then
the initial perturbation is very small, and the WKB solution almost solves the
system \eqref{KG} of coupled KG equations. The point we make in Theorem \ref{theorem1} is that
even though the WKB solution is very close to solving the exact system, it is
somehow not close to the exact solution. Parameter $\kappa_0$ measure how 'far' the WKB
solution strays from the exact solution: the distance between $(u,v)$ and $(u_a,v_a)$ goes from
$\e^K$, very small ($K$ arbitrarily large), to much larger than $\e^{\k_0}$ (with $\k_0$
fixed, depending on the dimension), in short time $O(\e^{1/4} |\ln \e|^{1/2})$.

  \item second remark: The second result \eqref{bd:insta-theo} is obtained
   by \emph{assuming} the existence of  solution $u$ on $[0, T_1 \e^{1/4} |\ln \e|^{1/2}]$ with any $T_1< T_1^*$. The existence time we obtain here is $[0, T_0 \e^{1/4} |\ln \e|^{1/2}]$ with $T_0<T_0^*$ and clearly $T_0^*<T_1^*$. Hence an open problem is to show the existence in longer time.

\end{itemize}


\section{Proof of Proposition \ref{prop:wkb}}\label{sec:wkb}

The aim of  this section is to construct an approximate solution
through WKB expansion. At the same time, this gives a proof for
Proposition \ref{prop:wkb}.

\subsection{WKB expansion}\label{sec:wkb-exp}
We describe approximate solutions to \eqref{coupled-KG} in the forms
of WKB expansions for \emph{profiles} with  $\theta = (k x - \o t )/\e$:
 \begin{equation} \label{ansatz}
  \begin{aligned}
   (u,v)(t,x)  = \sum_{n = 0}^{2K_a} \e^{n/2}  ({\bf u}_n, {\bf v}_n)(t,x,\theta) = \sum_{n = 0}^{2K_a} \e^{n/2} \sum_{p \in {\cal H}_n} e^{ip \theta} (u_{n,p}, v_{n,p})(t,x) , \end{aligned}
  \end{equation}
 where $2K_a\in \Z_+$ determines the precision of the expansion. For $(u,v)$ in \eqref{ansatz},
  $$ (F(u+v, v), G(u,u), H(v,v)) = \left[ \sum_{n = 0}^{2K_a} \e^{n/2} ({\bf F}_n, {\bf G}_n, {\bf H}_n) \right].$$
We let  $({\bf u})_p$ denote the $p$-harmonic of a
trigonometric polynomial ${\bf u}$ in $\theta.$
  For example, in \eqref{ansatz}, $({\bf u}_n, {\bf v}_n)_p = (u_{n,p},v_{n,p}).$
  We denote by ${\mathbb P}$ and ${\mathbb Q}$ the operators acting diagonally on trigonometric polynomials,
  defined as $${\mathbb P} {\bf u} = \sum_p e^{i p \theta} P(p\tilde \b) ({\bf u})_p,\quad {\mathbb Q} {\bf u} = \sum_p e^{i p \theta} Q(p\tilde \b) ({\bf
  u})_p,$$ where $P(p\tilde\b)$ and $Q(p\tilde\b)$  are orthogonal projectors onto $\ker
  L(ip\tilde\b)$ and $\ker
  M(ip\tilde\b)$ respectively with the definitions
$$
 L(ip\tilde\b):=L\big(\o_0,(-ip\o,ipk)\big),\quad
 M(ip\tilde\b):=M\big(\o_0,(-ip\o,ipk)\big),
$$
for which we use the notation \eqref{notation-lm}.

 We now plug \eqref{ansatz} into \eqref{coupled-KG} and write a cascade of WKB equations, the first of which comprises all terms of order $O(\e^{-1}):$
 \begin{equation} \label{wkb0}
  L(\tilde\b \d_\theta) {\bf u}_0 = 0, \qquad M(\tilde\b \d_\theta) {\bf v}_0 = 0.
  \end{equation}
 By \eqref{und-phase}, equation \eqref{wkb0} implies
 \begin{equation}
 {\bf u}_0 = u_{0,-3} e^{- 3 i \theta} + u_{0,0} + u_{0,3} e^{3 i \theta},
 \qquad {\bf v}_0 = v_{0,-1} e^{- i \theta} + v_{0,0} + v_{0,1} e^{ i
 \theta}
 \nn
 \end{equation}
with $u_{0,p}\in\ker L(ip\tilde\b ), v_{0,p}\in\ker M(ip\tilde\b )$. Direct
calculation gives \be\label{e1-e3}v_{0,1}=f \vec e_1,\qquad
u_{0,3}:=g \vec e_3,\ee where  $f$ and $g$ are scalar functions
determined later by \eqref{evol-wkb};
 $\vec e_1$ and $\vec e_3$ are constant vectors, which form the bases of vector spaces $ \ker M(i\tilde\b)$ and $\ker L(i3\tilde\b) $ respectively:
\be\label{e1-e3-exact}\vec e_1:=\left(-\frac{\th_0k}{\o},1,\frac{i
\o_0}{\o}\right),\qquad \vec e_3:=\left(-\frac{k}{\o},1,\frac{i\a_0
\o_0}{3\o}\right).\ee For negative $p$, reality requires
$u_{0,p}=\bar u_{0,-p},~v_{0,p}=\bar v_{0,-p}$. We define $\vec
e_{p}:=(\vec e_{-p})^*$ for $p\in\{-3,-1\}$.

\smallskip

 The $O(\e^{-1/2})$ terms in the expansion are
 \begin{equation} \label{wkb1}
   L(\tilde\b \d_\theta) {\bf u}_1  = {\bf F}_0, \quad
   M(\tilde\b \d_\theta) {\bf v}_1  = {\bf G}_0 + {\bf H}_0.
   \end{equation}
  A consequence of \eqref{wkb1} is the following  necessary condition
  \begin{equation} \label{wt0}
   P(p\tilde \b) ({\bf F}_0)_p  = 0, \quad   Q(p \tilde\b) ({\bf G}_0 + {\bf H}_0)_p  = 0,
   \end{equation}
   for all $p,p' \in \Z.$ By the properties of the characteristic varieties,
   the choice of $\b$ \eqref{und-phase} and the structure of the bilinear terms
   \eqref{def-F-G-H}, condition \eqref{wt0} is equivalent to the  following compatibility condition
   for all $p\in\Z$:
  \begin{equation} \label{wt1} \begin{aligned}
   P(p\tilde \b) \sum_{p_1+p_2=p}F\big((P+Q)(p_1\tilde\b)\cdot,Q(p_2\tilde\b)\cdot\big) & = 0, \\
   Q(p \tilde\b) \sum_{p_1+p_2=p}\Big(G\big(P(p_1\tilde\b)\cdot,P(p_2\tilde\b)\cdot\big)+H\big(Q(p_1\tilde\b)\cdot,Q(p_2\tilde\b)\cdot\big)\Big)& = 0.
   \end{aligned}
   \end{equation}
In our context, we find out  \eqref{wt1} is satisfied. This is in fact the weak transparency condition. Since there is no mean mode in initial
datum \eqref{id-KG}, we simply take $u_{0,0}=v_{0,0}=0$. We remark
that this choice of zero mean mode is not necessary to construct an
approximate solution by WKB expansion.

 The equation \eqref{wkb1} also implies
  \begin{equation} \label{1st-cor}
  (1 - P(p\tilde \b)) u_{1,p}  = L(i p \tilde\beta)^{(-1)} ({\bf F}_0)_p,~ (1 - Q(p \tilde\b)) v_{1,p}  = M(i p \tilde\beta)^{(-1)} ({\bf G}_0 + {\bf H}_0)_p,
  \end{equation}
  where $L^{(-1)}$ and $M^{(-1)}$ denote partial inverses, naturally defined on the
   orthogonal complements of $\ker L$ and $\ker M.$

\medskip

 The $O(\e^0)$ terms are
   \begin{equation} \label{wkb2}
   L(\tilde\b \d_\theta) {\bf u}_2 + L(0,\d) {\bf u}_0  = {\bf F}_1,\quad
   M(\tilde\b \d_\theta) {\bf v}_2 + M(0,\d) {\bf v}_0  = {\bf G}_1 + {\bf H}_1.
   \end{equation}
  From \eqref{wkb2} we deduce
  \begin{equation} \label{evol-wkb}
  P(3\tilde \b) L(0,\d) P(3\tilde\b) u_{0,3} = P(3\tilde \b) ({\bf F}_1)_{3},  \quad
  Q(\tilde\b) M(0,\d) Q( \tilde\b) v_{0,1}  = Q(\tilde\b) ({\bf G}_1 + {\bf H}_1)_1.
  \end{equation}
  By \eqref{wt0}, \eqref{1st-cor} and symmetry of $F,$ the source term in $\eqref{evol-wkb}$ is
  \begin{equation} \begin{aligned}
  P(3\tilde \b) ({\bf F}_1)_3 & = 2 P(3 \tilde\b) F( M(2 i \tilde\b)^{-1} H(v_{0,1}, v_{0,1}), v_{0,1}) + (\tilde {\bf F}_1)_3(u_{0,\pm 3},v_{0,\pm 1}), \end{aligned}
  \nn \end{equation}
   where $\tilde {\bf F}_1$ is a polynomial in $(u_{0,\pm 3},v_{0,\pm 1}).$ Since there is no third order harmonic in the leading terms of the initial data \eqref{id-KG}, we choose always $u_{0,3}(0,\cdot)=0$.
   Then the initial value $
  \tilde {\bf F}_{1}(0,\cdot) = 0.
  $
  By \eqref{wt0} and \eqref{1st-cor}, the source term in \eqref{evol-wkb}(ii)
  is also a polynomial in $(u_{0,\pm 3},v_{0,\pm 1})$ that admits zero initial value. By reality of the nonlinear terms, and symmetry of $L$ and $M,$
  the system in $(u_{0,-3}, v_{0,-1})$ is conjugated to \eqref{evol-wkb}. The differential operators in the
  left-hand side of \eqref{evol-wkb} are transport operators at the group velocities $\d_\xi \l(3 \tilde\b)$ and $\d_\xi \mu(\tilde\b)$
  respectively (for a proof of this fact, see \cite{JMR1, T1}), where
  $\l$ and $\mu$ are defined in \eqref{kg-modes}.

A consequence of the compatibility condition \eqref{wt1}  is that
the formal WKB equations lead to \emph{closed} transport equations in
$(u_{0,\pm 3},v_{0,\pm 1})$ with polynomial source terms. This
implies that, given initial data $(u_{0,3},v_{0,1})(0)=(0,v^0)\in
H^s,~s>d/2$, system \eqref{evol-wkb} and its conjugate system
admit a unique solution $(u_{0,\pm 3},v_{0,\pm 1})$ on $[0,T]$ for
some $T>0$ independent of $\e$. Moreover, there holds the estimate:
\be\label{est-uv0} \|(u_{0,\pm 3},v_{0,\pm
1})\|_{L^\infty([0,T], H^s )}+\|\d_t(u_{0,\pm 3},v_{0,\pm
1})\|_{L^\infty([0,T], H^{s-1})}\leq C (1+T).\ee In
particular, since $u_{0,3}(0,\cdot)\equiv 0$, together with
\eqref{e1-e3}, we have $g(0,\cdot)\equiv0$ and \be\label{est-g}
\|g\|_{L^\infty([0,T], H^s )}+\|\d_t
g\|_{L^\infty([0,T], H^{s-1} )}\leq CT. \ee

\smallskip

One key point here is that, by the structure of \eqref{kg-modes},
the third-order harmonics $u_{0,\pm3}$ grow
in time on $[0,t]$ for some $0<t\leq T$. Indeed, by
$(u_{0,3},v_{0,1})(0)=(0,v^0)$,
$\eqref{evol-wkb}_1$, and polarization condition $v^0=Q(\b)v^0$, there holds for nonzero $v_0$ that
\begin{equation} \label{new-structure}
\big(\d_t u_{0,3}\big)(0,x)=2 P(3\tilde \b) F\big( M(2 i\tilde \b)^{(-1)} H(v^0,
v^0), v^0\big)(x)\neq 0.
 \end{equation}
The equation \eqref{wkb2} also implies
  \begin{equation}  \begin{aligned}
  (1 - P(p\tilde \b)) u_{2,p} & = L(i p \tilde\beta)^{(-1)} \big(-L(0,\d) u_{0,p}+({\bf F}_0)_p\big),  \\
  (1 - Q(p\tilde \b)) v_{2,p} & = M(i p\tilde \beta)^{(-1)} \big(-M(0,\d) v_{0,p}+({\bf G}_0 + {\bf  H}_0)_p\big).
  \end{aligned}
 \nn \end{equation}

\subsection{The approximate solution and Proof of Proposition \ref{prop:wkb}}

With the compatibility condition \eqref{wt1}, a similar proof as
Theorem 2.3 in \cite{JMR2}, or an application of Section 6.6 in
\cite{em4}, we can continue the  WKB expansion in Section
\ref{sec:wkb-exp} up to any order. This implies that, for any
$K_a>0$, there exists a WKB solution $(u^a,v^a)$ of the form
\eqref{ansatz}, such that on $[0,T]$ for some $T>0$ independent of $\e$: \be
 \left\{ \begin{aligned}
     L(\frac{ \o_0}{\e},\d) u^a & = \frac{1}{\sqrt \e} F(u^a + v^a,v^a)+\e^{K_a}r_1^\e, \\
     M(\frac{ \o_0}{\e}, \d) v^a & = \frac{1}{\sqrt \e} \big(G(u^a,u^a) +
     H(v^a,v^a)\big)+\e^{K_a}r_2^\e.
     \end{aligned}
     \right.\nn\ee
The other results in Proposition \ref{prop:wkb} are
obtained from the standard WKB expansion.

\section{Proof of Theorem \ref{theorem1}}\label{sec:slow-ins}

We show in this section that the WKB solution $(u^a,v^a)$
constructed in Section \ref{sec:wkb}. is unstable in the sense stated in Theorem
\ref{theorem1}.

\subsection{Preparation}

 By symmetry of the hyperbolic operator, for $\e
> 0$ the solution $(u,v)$ to the initial value problem
\eqref{coupled-KG},\eqref{id-KG-pert} is defined on time
interval $[0,T(\e)]$ for some $T(\e) > 0.$ By \eqref{coupled-KG},  \eqref{sys:app} and
\eqref{id-KG-pert},  the unknown perturbation
$(\dot u,\dot v)$ defined as
\begin{equation} \label{def:dot-u}
 (u,v) =: (u^a,v^a) + \e^\kappa   (\dot u,\dot v) \qquad  \mbox{for some $1/4<\kappa  \leq  \min\{K,K_a+1/4\}$}
 \end{equation}
 satisfies
\begin{equation} \label{eq:dot-u}
    \left\{ \begin{aligned}
     L(\frac{ \o_0}{\e},\d) \dot u & = \frac{1}{\sqrt \e} \big(F(u^a)\dot v+F(v^a)\dot u+2F(v^a)\dot v\big)+\frac{\e^{\kappa  }}{\sqrt\e} F(\dot u+\dot v,\dot v)-\e^{K_a-\kappa  }r_1^\e, \\
     M(\frac{ \o_0}{\e}, \d) \dot v & =\frac{2}{\sqrt \e} \big(G(u^a)\dot u+H(v^a)\dot v\big)+\frac{\e^{\kappa  }}{\sqrt\e}\big(G(\dot u,\dot u)+H(\dot v,\dot v)\big) -\e^{K_a-\kappa  }r_2^\e,\\
 \dot u(0,x) &=\e^{K-\kappa  } \phi_1^\e(x), \qquad \dot v(0,x) = \e^{K-\kappa  } \phi_2^\e(x),
     \end{aligned}
     \right.
    \end{equation}
where we define $B(a) b := B(a, b) $ for any $B\in\{F,G,H\}$.


\subsubsection{Spectral decompositions, resonances and transparencies}

According to  \eqref{spe-dec}, we write the  decompositions in semiclassical Fourier multipliers:
\be\begin{aligned}
&L(\frac{\o_0}{\e},\d_t,\d_x):=\d_t+\frac{i}{\e}\big(\op_\e(\l_+)\op_\e(P_+)+\op_\e(\l_-)\op_\e(P_-)+\op_\e(\l_0) \op_\e(P_0)\big),\\
&M(\frac{\o_0}{\e},\d_t,\d_x):=\d_t+\frac{i}{\e}\big(\op_\e(\mu_+)\op_\e(Q_+)+\op_\e(\mu_-)\op_\e(Q_-)+\op_\e(\mu_0)
\op_\e(Q_0)\big)\end{aligned}\nn\ee with
$$\l_+=-\l_-:=\l,\quad \mu_+=-\mu_-:=\mu,\quad \l_0=\mu_0:=0.$$
To unify the notations, we denote for $j\in\{+,-,0\}$:
$$
\l_j^L:=\l_j,\quad \l_j^M:=\mu_j,\quad \Pi_j^L:=P_j,\quad
\Pi_j^M:=Q_j.
$$
By Proposition \ref{prop:wkb}, we have for $B\in\{F,G,H\}$: \be\label{lin-sour}
B(u^a)=\sum_{p=\pm3}e^{ip\th}B(u_{0,p})+\sqrt\e B(u^\e_r),\quad
B(v^a)=\sum_{p=\pm1}e^{ip\th}B(v_{0,p})+\sqrt\e B(v_r^\e). \ee This
indicates that the singular linear source terms (of order $1/\sqrt\e$) are essentially those associated
with the leading terms of the WKB solution: $u_{0,\pm 3}$ and
$v_{0,\pm 1}$.

As we introduced in Introduction, resonances  are zero points to the factors appearing in strong transparency condition. They actually correspond to crossing of eigenmodes, and are defined as obstruction to the solvability of homological equations that arise in normal form change of variables used to prove the stabilities of WKB solutions in for example \cite{JMR2}. They play an important role
 in the well-posedness analysis. We give the precise definitions in our setting:
\begin{defi}[Resonances and interaction coefficients]
Given $(i,j) \in \{+,-, 0\}^2,$ $p\in\{-3,-1,1,3\}$ and
$(\de,\sigma)\in\{L,M\}^2$, we define the resonance set
$$
{\bf \cal R}_{ij,p}^{\de,\sigma}:=\big \{\xi\in \R, \quad p\o =
\l^\de_i(\xi+pk) - \l^\sigma_{j}(\xi) \big\}.
$$
For bilinear form $B\in\{F,G,H\}$, the families of matrices $
\Pi^\de_{i}(\xi+pk)B(\vec e_p)\Pi^\sigma_{j}(\xi) $ and $
\Pi^\sigma_{j}(\xi)B(\vec e_{-p})\Pi^\de_{i}(\xi+pk) ,$ indexed by
$\xi \in \R^1,$ are called the interaction coefficients associated
with $(i,j,p,\de,\sigma).$ The scalar function $\xi \to \l_i^\de(\xi
+ pk) - \l^\sigma_j(\xi) - p\o$ is called the resonant phase
associated with $(i,j,p,\de,\sigma).$
\end{defi}
\noindent We recall $e_{p}$ come from the polarization condition \eqref{e1-e3} and are given in \eqref{e1-e3-exact}.

\begin{defi}[Transparencies]
An interaction coefficient $\Pi^\de_{i}(\cdot+pk)B(\vec
e_{p})\Pi^\sigma_{j}(\cdot)$ or $\Pi^\sigma_{j}(\cdot)B(\vec
e_{-p})\Pi^\de_{i}(\cdot+pk)$ is said to be {\it transparent} if the
associated resonant phase can be factored out, which means there holds for all $\xi \in \R^1$ the bound
$$
|\Pi^\de_{i}(\xi+pk)B(\vec e_{p})\Pi^\sigma_{j}(\xi)|\leq C
|\l^\de_i(\xi+pk)-p\o-\l^\sigma_{j}(\xi)|,
$$
or
$$
|\Pi^\sigma_{j}(\xi)B(\vec e_{-p})\Pi^\de_{i}(\xi+pk))| \leq C
|\l^\de_i(\xi+pk)-p\o-\l^\sigma_{j}(\xi)|.$$

\end{defi}

\begin{rem}\label{rem:nore-tran}
(i), Interaction coefficients $\Pi^\de_{i}(\xi+pk)B(\vec e_{p})\Pi^\sigma_{j}(\xi)$ and $\Pi^\sigma_{j}(\xi)B(\vec e_{-p})\Pi^\de_{i}(\xi+pk))$ share the same resonant phase $\l^\de_i(\xi+pk)-p\o-\l^\sigma_{j}(\xi)$ and resonant set ${\bf \cal R}_{ij,p}^{\de,\sigma}$. Equivalently, ${\bf \cal R}_{ij,p}^{\de,\sigma}$ and ${\bf \cal R}_{ji,-p}^{\sigma,\de}$ share essentially (a shift of $\xi$ with $pk$ or $-pk$) the same interaction coefficients, so it is enough to consider one of theses two resonant sets and study the transparency property of the associated interaction coefficients.

(ii), If ${\bf \cal R}_{ij,p}^{\de,\sigma}$ is empty, by the smoothness of
the eigenmodes \eqref{kg-modes}, there holds for some $c>0$ and all $\xi\in\R^1$ the
lower bound:
$$|\l^\de_i(\xi+pk)-p\o-\l^\sigma_{j}(\xi)|\geq c.$$
Then the  related interaction coefficients
$\Pi^\de_{i}(\cdot+pk)B(\vec e_{p})\Pi^\sigma_{j}(\cdot)$ and
$\Pi^\sigma_{j}(\cdot)B(\vec e_{-p})\Pi^\de_{i}(\cdot+pk)$ are all
transparent. This is true because the interaction coefficients are
uniformly bounded in $\xi$, see later \eqref{proj}, \eqref{proj2},
\eqref{proj01} and \eqref{proj02}.
\end{rem}

We now calculate all the non-empty resonance sets and the
transparency property  for the related interaction coefficients.
Remark that, we only need to calculate the interaction coefficients appearing in \eqref{eq:dot-u}.

Associated with the fundamental harmonics $v_{0,\pm1}$ (namely $e_{\pm 1}$), with our choice
$0<\th_0<1$ and $2.5<\a_0<3$, essentially  there are four non-empty
resonance sets:
\be\label{res:p=1}
\begin{aligned}
&{\bf \cal R}_{+0,1}^{M,L}={\bf \cal R}_{+0,1}^{M,M}=\big \{\xi\in
\R, \quad \o =
\mu_+(\xi+k)\}=\{0,-2k\},\\
&{\bf \cal R}_{0-,-1}^{L,M}={\bf \cal R}_{0-,-1}^{M,M}=\big \{\xi\in
\R, \quad \o = 0-\mu_-(\xi-k)\}=\{0,2k\}.
\end{aligned}
\ee To check the transparency property of the interaction
coefficients,  we need to calculate the eigen-projections:
\\(i) for $j\in\{+,-\}$, $\de\in\{L,M\}$, there holds
\be\label{proj} \Pi_j^\de(\xi)
V=\frac{\big(V,\O^\de_j(\xi)\big)\O^\de_j(\xi)}{|\O^\de_j(\xi)|^2}=\frac{\big(V,\O^\de_j(\xi)\big)\O^\de_j(\xi)}{2},
\ee where \be\label{proj2} \O^L_j(\xi):=\left(\frac{-\xi}{\l_j(\xi)}
,1,\frac{i\a_0\o_0}{\l_j(\xi)}\right),\quad \O^M_{j}:=\left(
\frac{-\th_0\xi}{\mu_{j}(\xi)} ,1,\frac{i\o_0}{\mu_{j}(\xi)}\right).
\ee
(ii) for $j=0$, $\de\in\{L,M\}$, there holds \be\label{proj01}
\Pi_0^\de(\xi)
V=\frac{\big(V,\O^\de_0(\xi)\big)\O^\de_0(\xi)}{|\O^\de_0(\xi)|^2},
\ee where \be\label{proj02} \O^L_0(\xi):=\left(1
,0,-\frac{i\xi}{\a_0\o_0}\right),\quad \O^M_{0}:=\left(
1,0,-\frac{i\th_0\xi}{\o_0}\right). \ee To calculate the interaction
coefficients related to the non-empty resonance sets in
\eqref{res:p=1}, by \eqref{eq:dot-u} and \eqref{lin-sour}, it is
sufficient to calculate \be\label{proj3}
\begin{aligned}
&P_0(\xi)F(\vec e_{-1})Q_+(\xi+k),\quad Q_0(\xi)H(\vec
e_{-1})Q_+(\xi+k),\quad Q_+(\xi+k)H(\vec e_{1})Q_0(\xi),\\
&P_0(\xi)F(\vec e_{1})Q_-(\xi-k),\quad Q_0(\xi)H(\vec
e_{1})Q_-(\xi-k),\quad Q_-(\xi-k)H(\vec e_{-1})Q_0(\xi).
\end{aligned}
\ee
With our choice of the bilinear forms as in \eqref{def-F-G-H}, we
find the interaction coefficients in \eqref{proj3} are all
identically zero. Together with Remark \ref{rem:nore-tran} (ii), we have:
\begin{lem}\label{lem:v01-tran}
All the interaction coefficients related to the fundamental
harmonics $ v_{\pm1}$ are transparent.
\end{lem}

Associated with the third order harmonics $u_{0,\pm3}$ (namely $\e_{\pm3}$), the non-empty
resonance sets are
\be\label{res:p=3}
\begin{aligned}
&{\bf \cal R}_{+0,3}^{L,L}={\bf \cal R}_{+0,3}^{L,M}=\big \{\xi\in
\R, \quad 3\o =
\l_+(\xi+3k)\}=\{0,-6k\},\\
&{\bf \cal R}_{0-,-3}^{L,L}={\bf \cal R}_{0-,-3}^{M,L}=\big \{\xi\in
\R, \quad 3\o = 0-\l_-(\xi-3k)\}=\{0,6k\},\\
&{\bf \cal R}_{+0,3}^{M,L}={\bf \cal R}_{+0,3}^{M,M}=\big \{\xi\in
\R, \quad 3\o = \mu_+(\xi+3k)\}=\{-\xi_1-3k,\xi_1-3k\},\\
&{\bf \cal R}_{0-,-3}^{L,M}={\bf \cal R}_{0-,-3}^{M,M}=\big \{\xi\in
\R, ~ 3\o = 0-\mu_-(\xi-3k)\}=\{\xi_1+3k,-\xi_1+3k\},\\
&{\bf \cal R}_{++,3}^{L,M}=\big \{\xi\in \R, \quad 3\o =
\l_+(\xi+3k)\}-\mu_+(\xi)\}=\{\xi_2,\xi_3\},\\
&{\bf \cal R}_{--,-3}^{M,L}=\big \{\xi\in \R, \quad 3\o =
\mu_-(\xi)-\l_-(\xi-3k)\}=\{-\xi_2,-\xi_3\},
\end{aligned}
\ee where $\xi_1=\sqrt{9\o^2-\o_0^2}/\th_0$ such that
$\mu(\xi_1)=3\o$, $\xi_2$ and $\xi_3$ are solutions to
$$
3\o=\sqrt{(\xi+3k)^2+\a_0^2\o_0^2}-\sqrt{\th_0^2\xi^2+\o_0^2}.
$$

\smallskip

For \eqref{eq:dot-u}, the interaction coefficients related to the
non-empty resonance sets in \eqref{res:p=3} and the corresponding
resonances are
$$
\begin{aligned}
&P_+(\xi+3k)F(\vec e_{3})Q_0(\xi),\quad Q_0(\xi)G(\vec
   e_{-3})P_+(\xi+3k),\qquad \xi=0,-6k;\\
&P_-(\xi-3k)F(\vec e_{-3})Q_0(\xi),\quad Q_0(\xi)G(\vec
   e_{3})P_-(\xi-3k),\qquad \xi=0,6k;\\
& P_0(\xi)F(\vec e_{-3})Q_+(\xi+3k),\quad Q_+(\xi+3k)G(\vec
   e_3)P_0(\xi),\qquad \xi=-\xi_1-3k,\xi_1-3k;\\
& P_0(\xi)F(\vec e_{3})Q_-(\xi-3k),\quad Q_-(\xi-3k)G(\vec
   e_{-3})P_0(\xi),\qquad \xi=\xi_1+3k,-\xi_1+3k;\\
&P_+(\xi+3k)F(\vec e_{3})Q_+(\xi),\quad Q_+(\xi)G(\vec
   e_{-3})P_+(\xi+3k),\qquad \xi=\xi_2,\xi_3;\\
&P_-(\xi-3k)F(\vec e_{-3})Q_-(\xi),\quad Q_-(\xi)G(\vec
   e_{3})P_-(\xi-3k),\qquad \xi=-\xi_2,-\xi_3.
\end{aligned}
$$
By direct calculation,  the above interaction coefficients are all
transparent except the following ones around the following resonance
points: \be\label{non-tran-int}\begin{aligned} &P_+(\xi+3k)F(\vec
e_{3})Q_0(\xi),\quad \xi=-6k;\qquad P_-(\xi-3k)F(\vec e_{-3})Q_0(\xi),\quad \xi=6k;\\
&P_+(\xi+3k)F(\vec e_{3})Q_+(\xi),\quad Q_+(\xi)G(\vec
   e_{-3})P_+(\xi+3k),\qquad \xi=\xi_2,~\xi_3;\\
&P_-(\xi-3k)F(\vec e_{-3})Q_-(\xi),\quad Q_-(\xi)G(\vec
   e_{3})P_-(\xi-3k),\qquad \xi=-\xi_2,~-\xi_3.\end{aligned}
\ee We denote the set of all the non-transparent resonance points
appeared in \eqref{non-tran-int} by \be {\cal
R}:=\{-6k, 6k, \xi_2, \xi_3, -\xi_2,-\xi_3\}.\nn\ee We choose $\th_0$
and $\a_0$ such that the resonance points in ${\cal R}$ are pairwise distinct.
This is true except for finite choices of $\th_0$ and $\a_0$.

We will show that these non-transparent interaction coefficients in
\eqref{non-tran-int} will cause the solution $(\dot u,\dot v)$ to be
amplified instantaneously.

\subsubsection{Projections and and frequency shifts}

We observe that the leading terms of the WKB solutions are highly
oscillating, in the sense that they have prefactors  $e^{ip\th}$ with $\th=(kx-\o t)/\e,
p\in\{-3,-1,1,3\}$. These highly oscillating factors will cause technical difficulties in our analysis.
Indeed, we focus on the non-transparent interaction coefficients by
localizing the frequencies near resonances. This localization is done
by applying a semiclassical Fourier multiplier $\op_\e(\chi)$
to the related interaction coefficients (see Section \ref{s-f loc} for further
details). The symbol $\chi(\xi)$ is a cut-off function compactly
supported in a neighborhood of some resonance set. If there is a highly oscillating multiplier $e^{ip\th}$
to some interaction coefficient $b_{jj'}$, then
$$\op_\e(\chi)(e^{ip\th}b_{jj'})= e^{ip\th}\op_\e(\chi)b_{jj'}+ [\op_\e(\chi),e^{ip\th}]b_{jj'},$$
where the commutator $[\op_\e(\chi),e^{ip\th}]$ is $O(1)$ due  to
the high oscillation. We need the commutator to be
of order at least $O(\sqrt\e)$ because the linear source terms in
\eqref{eq:dot-u}, where the  interaction coefficients come from,
have prefactor $1/\sqrt\e$.

Thus we would like to eliminate those highly oscillating prefactors $e^{ip\th}$ of the non-transparent interaction coefficients. To achieve this, we introduce the following frequency shifts by defining $U:=(u_+,u_-,u_0,
v_+,v_-,v_0) \in \R^{18}$ as \be\label{de-U}\begin{aligned}
&u_{+}:=e^{-i3\th}\op_\e(P_{+})\dot u,\quad u_-
:=e^{i3\th}\op_\e(P_{-}) \dot u,\quad u_0:= \op_\e(P_{0}) \dot
u,\\&v_{+}:=\op_\e(Q_{+})\dot v,\quad v_- :=\op_\e(Q_{-}) \dot
v,\quad v_0:= \op_\e(Q_{0}) \dot v.
\end{aligned}\ee The perturbation unknown $(\dot u,\dot v)$ can be
reconstructed from $U$ via
 \begin{equation}
  \dot u =  e^{i3\th}u_{+}+e^{-i3\th}u_-+u_0,\qquad   \dot v =
  v_{+}+v_-+v_0.\nn
 \end{equation}

From \eqref{eq:dot-u} and \eqref{de-U}, the new unknown $U $ solves
\begin{equation} \label{sys U}
 \d_t U + \frac{i}{\e} \op_\e( \mathcal{A}) U  = \frac{1}{\sqrt\e} \op_\e({\cal B}) U + F.
\end{equation}
The symbol of the propagator is the diagonal matrix
$$ \mathcal{A} := {\rm diag}\, \big( \l_{+}(\cdot+3k)-3\o, \l_-(\cdot-3k)+3\o,0, \mu_+, \mu_-, 0 \big).$$
 The symbol of the singular source term is
$$ {\cal B} := \left(\begin{array}{c|c} {\cal B}_{[P,P]} & {\cal B}_{[P,Q]} \\ \hline {\cal B}_{[Q,P]} & {\cal B}_{[Q,Q]} \end{array}\right),$$
where the blocks are:
$$
\begin{aligned}
{\cal B}_{[P,P]} :&= \sum_{p =\pm1}e^{ip \theta}\\
\times &~ \bp P_{+,p+3}
F(v_{0p}) P_{+,3}&  e^{-i6\theta} P_{+,p-3} F(v_{0p}) P_{-,-3} &e^{-i3\theta} P_{+,p} F(v_{0p}) P_{0}
\\ e^{i6\th}P_{-,p+3} F(v_{0p}) P_{+,3}&  P_{-,p-3} F(v_{0p}) P_{-,-3} & e^{i3\theta} P_{-,p} F(v_{0p}) P_{0}
\\e^{i3\th} P_{0,p+3}
F(v_{0p}) P_{+,3}&  e^{-i3\theta} P_{0,p-3} F(v_{0p}) P_{-,-3} & P_{0,p} F(v_{0p}) P_{0}\ep;\\
\end{aligned}
$$
$$
\begin{aligned}
 {\cal B}_{[P,Q]} :&= {\cal B}_{[P,Q]}^1+ {\cal B}_{[P,Q]} ^2 \qquad{\rm with}\\
 {\cal B}_{[P,Q]}^1:&= \sum_{p =\pm3} e^{ip\th}\bp e^{-i3\th}P_{+,p}
F(u_{0p}) Q_{+}&  e^{-i3\theta} P_{+,p} F(u_{0p}) Q_{-}
&e^{-i3\theta} P_{+,p} F(u_{0p}) Q_{0}
\\ e^{i3\th}P_{-,p} F(u_{0p}) Q_{+}& e^{i3\th} P_{-,p} F(u_{0p}) Q_{-} & e^{i3\theta} P_{-,p} F(u_{0p}) Q_{0}
\\ P_{0,p}
F(u_{0p}) Q_{+}&  P_{0,p} F(u_{0p}) Q_{-} & P_{0,p} F(u_{0p}) Q_{0}\ep,\\
{\cal B}_{[P,Q]}^2:&= 2\sum_{p =\pm1} e^{ip\th}\bp e^{-i3\th}P_{+,p}
F(v_{0p}) Q_{+}&  e^{-i3\theta} P_{+,p} F(v_{0p}) Q_{-}
&e^{-i3\theta} P_{+,p} F(v_{0p}) P_{0}
\\ e^{i3\th}P_{-,p} F(v_{0p}) Q_{+}& e^{i3\th} P_{-,p} F(v_{0p}) P_{-} & e^{i3\theta} P_{-,p} F(v_{0p}) Q_{0}
\\ P_{0,p}
F(v_{0p}) Q_{+}&  P_{0,p} F(v_{0p}) Q_{-} & P_{0,p} F(v_{0p}) Q_{0}\ep;\\
\end{aligned}
$$
$$\begin{aligned} {\cal B}_{[Q,P]} :&=2\sum_{p =\pm3} e^{ip\th}\\
&\times \bp e^{i3\th}Q_{+,p+3} G(u_{0p}) P_{+,3}&  e^{-i3\theta}
Q_{+,p-3}G(u_{0p}) P_{-,-3} & Q_{+,p} G(u_{0p}) P_{0}
\\ e^{i3\th}Q_{-,p+3} G(u_{0p}) P_{+,3}& e^{-i3\th} Q_{-,p-3} G(u_{0p}) P_{-,-3} & Q_{-,p} G(u_{0p}) P_{0}
\\ e^{i3\th}Q_{0,p+3} G(u_{0p}) P_{+,3}& e^{-i3\th} Q_{0,p-3} G(u_{0p}) P_{-,-3} & Q_{0,p} G(u_{0p}) P_{0}\ep,\\
{\cal B}_{[Q,Q]} :&=2\sum_{p =\pm1} \left( e^{ip \theta} Q_{j_1,+p}
H(v_{0p}) Q_{j_2} \right)_{j_1,j_2\in\{+,-,0\}},
\end{aligned}$$
where we use the notations
\begin{equation}
P_{j,+q}(\xi) := P_j(\xi + q k),\quad Q_{j,+q}(\xi) := Q_j(\xi + q
k), \quad q \in \Z,\quad j\in\{+,-,0\}.
 \nn\end{equation}

In \eqref{sys U}, the remainder $F$ is the sum of the quadratic
terms, the contribution of the higher-order WKB terms and remainder
terms arising from compositions of pseudo-differential operators;
for precise information, one may check Section 3.1.2 in \cite{em4}.
Similarly as Lemma 3.1 in \cite{em4}, we have here:
\begin{lem} There holds for all multiple index
$\a\in\N^d$ with $|\a|\leq d/2+d+2+(q_0+3)/4$:
  $$| (\e\d_x)^\a F(t,\cdot) |_{L^2_x} \leq C  \e^{\kappa  -1/2}|(\dot u,\dot v)(t,\cdot)|_{L^\infty_x}|(\e\d_x)^\a U(t,\cdot)|_{L^2_x} +C  \e^{
  K_a-\kappa  }.$$
 \end{lem}

\subsubsection{Normal form reduction}

By Lemma \eqref{lem:v01-tran}, the interaction coefficients in the
blocks ${\cal B}_{[P,P]}$, ${\cal B}_{[P,Q]}^2$ and ${\cal
B}_{[Q,Q]}$ are all transparent.

We then separate the non-transparent interaction coefficients
\eqref{non-tran-int} from ${\cal B}_{[P,Q]}^1$ and ${\cal
B}_{[Q,P]}$ as we write:
$${\cal B}_{[P,Q]}^1=  {\cal B}_{[P,Q]}^{nt}+{\cal B}_{[P,Q]}^{1,t},\quad {\cal B}_{[Q,P]}=
{\cal B}_{[Q,P]}^{nt}+{\cal B}_{[Q,P]}^{t},$$ where
$$
\begin{aligned}
& {\cal B}_{[P,Q]}^{nt}:=\bp P_{+,3} F(u_{03}) Q_{+}& 0 & P_{+,3}
F(u_{03}) Q_{0}
\\ 0 &  P_{-,-3} F(u_{0,-3}) Q_{-} &  P_{-,-3} F(u_{0,-3}) Q_{0}
\\ 0 &  0 &0\ep,\\
& {\cal B}_{[Q,P]}^{nt}:=2\bp Q_{+} G(u_{0,-3}) P_{+,3}& 0 & 0
\\ 0&  Q_{-} G(u_{0,3}) P_{-,-3} & 0
\\ 0 & 0 & 0\ep.
\end{aligned}
$$
The index $t$ means transparent and $nt$ means non-transparent. We
remark that, with our choice of shift of frequencies in
\eqref{de-U}, the non-transparent interaction coefficients, which are
all included in ${\cal B}_{[P,Q]}^{nt}$ and ${\cal B}_{[Q,P]}^{nt}$,
are no longer multiplied by highly oscillating terms $e^{ip\th}$.

According to \eqref{non-tran-int}, we consider frequency cut-off functions near resonance sets for any $(i,j) \in \{(+,0),(-,0),(+,+),(-,-)\}$:
\be\label{cut-off1}
\chi_{[i,j]}\in C_c^\infty\big({\mathcal
R^h_{[i,j]}}\big),\quad\chi_{[i,j]}\equiv 1 ~{\rm on}~{\mathcal
R^{h/2}_{[i,j]}}
 \ee
  where $h>0$ is a small constant fixed later on and ${\mathcal R}^r_{[i,j]}$
are balls of radial $r$ centered by the resonance points appeared in \eqref{non-tran-int}:
\be\label{cut-off2}
\begin{aligned}
&{\mathcal R^{r}_{[+,0]}}=B(-6k,r);\quad {\mathcal R^{r}_{[-,0]}}=B(6k,r);\\
& {\mathcal R^{r}_{[+,+]}}:= \big(B(\xi_2,r)\cup
B(\xi_3,r)\big);\quad {\mathcal R^{r}_{[-,-]}}:=\big(B(-\xi_2,r)\cup B(-\xi_3,r)\big).
\end{aligned}
\ee We choose $h$ small such that \be B(\xi,h)\cap
B(\eta,h)=\emptyset,\quad \forall ~\xi,\eta\in{\cal R},~\xi\neq
\eta,\nn\ee which holds true because the resonance points in ${\mathcal
R}$ are pairwise distinct. We then decompose the non-transparent
interaction coefficients ${\cal B}_{[P,Q]}^{nt}$ and ${\cal
B}_{[Q,Q]}^{nt}$ as follows:
$$
\begin{aligned}
 {\cal B}_{[P,Q]}^{nt}= {\cal B}_{[P,Q]}^{nt,1}+{\cal B}_{[P,Q]}^{nt,2},\quad
{\cal B}_{[Q,P]}^{nt}={\cal B}_{[Q,P]}^{nt,1}+{\cal
B}_{[Q,P]}^{nt,2},
\end{aligned}
$$
where
$$
\begin{aligned}
& {\cal B}_{[P,Q]}^{nt,1}:=
(\chi_{[+,+]})B^{(1)}_{+,+}+(\chi_{[-,-]})B^{(1)}_{-,-}+(\chi_{[+,0]})B_{+,0}^{(1)}+(\chi_{[-,0]})B_{-,0}^{(1)},\\
& {\cal B}_{[Q,P]}^{nt,1}:=
(\chi_{[+,+]})B^{(2)}_{+,+}+(\chi_{[-,-]})B^{(2)}_{-,-}
\end{aligned}
$$
with \be\label{B12-ij}\begin{aligned}&B^{(1)}_{+,+}:=\bp P_{+,3}
F(u_{03}) Q_{+}& 0 & 0 \\ 0 &0&0 \\ 0 &  0 &0\ep,\quad
B^{(1)}_{-,-}:=\bp  0& 0 & 0 \\ 0 &P_{-,-3} F(u_{0,-3}) Q_{-}&0 \\ 0
& 0 &0\ep,\\
&B^{(2)}_{+,+}:=2\bp  Q_{+} G(u_{0,-3}) P_{+,3} & 0 & 0 \\ 0 &0&0 \\
0 & 0 &0\ep,\quad B^{(2)}_{-,-}:=2\bp  0& 0 & 0 \\ 0 &Q_{-}
G(u_{0,3}) P_{-,-3}&0 \\ 0
& 0 &0\ep,\\
&B^{(1)}_{+,0}:=\bp  0& 0&P_{+,3} F(u_{03}) Q_{0}  \\ 0 &0&0 \\ 0 &
0 &0\ep,\quad B^{(1)}_{-,0}:=\bp  0& 0 & 0 \\ 0 & 0 &P_{-,-3}
F(u_{0,-3}) Q_{0}\\ 0 & 0 &0\ep.
\end{aligned}\ee
Since the cut-off functions are identically one near
resonances, the parts ${\cal B}_{[P,Q]}^{nt,2}$ and ${\cal
B}_{[Q,P]}^{nt,2}$ are supported away from resonances, then are
transparent. The following proposition states that the operator with symbol
 $$ {\cal D} := \left(\begin{array}{cc} {\cal B}_{[P,P]}  & {\cal B}_{[P,Q]}^{2}+{\cal B}_{[P,Q]}^{1,t}+{\cal B}_{[P,Q]}^{nt,2}
  \\ {\cal B}_{[Q,P]}^{t}+{\cal B}_{[Q,P]}^{nt,2} & {\cal B}_{[Q,Q]} \end{array}\right)$$
 can be eliminated, via a normal form reduction, from the evolution equation \eqref{sys U}.
\begin{prop}\label{normal-form-prop}There exists $\Lambda \in S^0,$ with $\d_t \Lambda \in
S^0$ uniformly in $t\in [0,\tilde T]$, where $\tilde T>0$ independent of $\e$ is from Proposition \eqref{prop:wkb}, such that \be\label{homo eq2}\e [\d_t,\op_\e(\Lambda)]+
i[\mathcal{A},\op_\e(\Lambda)] = \op_\e({\cal D}) +\e R_{(0)},\ee where
$R_{(0)}(t)$ a semiclassical pseudodifferential operator form $L^2_x$ to $L^2_x$ and is uniformly bounded with respect to $t\in [0,\tilde T]$ and in $\e>0$. The symbol set
$S^m,~m\in\R$ are defined  in {\rm Appendix A}.
\end{prop}
\begin{proof} We  write
$ \ds{{\cal D}=\sum_{|q|\leq 6} e^{iq\th}{\cal D}_q} $ and look
for solution $\Lambda$ to \eqref{homo eq2} of the form
$\ds{\Lambda=\sum_{|q|\leq 6} e^{iq\th}\Lambda_q}$. By direct calculation, we
deduce that in order to solve \eqref{homo eq2}, it is sufficient to
solve \begin{equation}  i \Big(- q \o +
\l_{j,+q}^{\de} - \l_{j'}^{\sigma}\Big) (\Lambda_q)_{(j,j'\de,\sigma)} =
({\cal D}_q)_{(j,j',\de,\sigma)}
\nn \end{equation}
for all index $|q| \leq
 6,(j,j')\in \{+,-,0\}^2,~(\de,\sigma)\in\{L,M\}^2.$
By transparencies and the definitions of the cut-off functions, the
solutions $(\Lambda_q)_{(i,j,\de,\sigma)}$ are well defined. For more details, one can check the proof of Proposition 3.3 in
\cite{em4}.
\end{proof}

By Proposition \ref{prop:action}, with $\Lambda$ given in Proposition
\ref{normal-form-prop}, the operator $\op_\e(\Lambda)(t)$ is bounded from
$L^2_x$ to $L^2_x$ uniformly in $t\in [0,\tilde T]$ and in $\e>0$.  In particular, for $\e$ small, $\Id + \sqrt \e \op_\e(\Lambda)$ is invertible. We consider the
change of variable
 \begin{equation} \label{def:check-u}
  \check U(t) := \Big(\Id + \sqrt \e \op_\e\big(\Lambda(\e^{1/4} t)\big)\Big)^{-1}U(\e^{1/4} t)
  \end{equation}
corresponding to a normal form reduction and a rescaling in time.
Then we have:
 \begin{cor} The equation in $\check U$ is
 \begin{equation}
 \d_t \check U + \frac{i}{\e^{3/4}} \op_\e({\cal A}) \check U =\frac{1}{\e^{1/4}} \op_\e(\check {\cal B}) \check U + \check F,
 \quad \check {\cal B} := \left(\begin{array}{cc} 0 & {\cal
B}_{[P,Q]}^{nt,1} \\ {\cal B}_{[Q,P]}^{nt,1} & 0
\end{array}\right)(\e^{1/4}t),
 \nn\end{equation}
 where for all multiple
 index $\a\in\N^d$ with $|\a |\leq d/2+d+2+(q_0+3)/4$:
 $$|(\e\d_x)^\a  \check F (t,\cdot)|_{L^2_x} \leq C \e^{\kappa  -1/4} |(\dot u,\dot v)(\e^{1/4} t,\cdot)|_{L^\infty_x}|(\e\d_x)^\a \check U(t,\cdot)|_{L^2_x} +C \e^{K_a+1/4 - \kappa  }.$$
\end{cor}

The proof is identical to the proof for Corollary 3.5 in \cite{em4},
and is omit here.

\subsubsection{Space-frequency localization}\label{s-f loc}

We define the following two quantities for the frequencies near non-transparent
resonance sets:
\begin{equation} \label{def:gamma12}
\begin{split}
& \gamma_1(\xi) := 2{\rm tr}\,\big(P_{+}(\xi+3k)F(\vec
e_{3})Q_{+}(\xi)G(\vec e_{-3})P_{+}(\xi+3k)\big),\quad {\rm for}~\xi\in {\mathcal R}_{[+,+]}^h;\\
& \gamma_2(\xi) := 2{\rm tr}\,\big(P_{-}(\xi-3k)F(\vec
e_{-3})Q_{-}(\xi)G(\vec e_{3})P_{-}(\xi-3k)\big),\quad {\rm
for}~\xi\in {\mathcal R}_{[-,-]}^h.
\end{split}
\end{equation}
By \eqref{proj} and \eqref{proj2}, we have \be\label{gamma1-gamma2}
\gamma_1(\xi)=\frac{\a_0\o_0^2}{6\o \mu_+(\xi)},\quad
\gamma_2(\xi)=\frac{-\a_0\o_0^2}{6\o \mu_-(\xi)}. \ee Since
$\mu_+=-\mu_-=\mu$ with $\mu$ defined in \eqref{kg-modes} which is
always positive, there holds \be\label{gamma1>0}
\gamma_1(\xi)=\gamma_2(-\xi)=\frac{\a_0\o_0^2}{6\o \mu(\xi)}>0. \ee
Equation \eqref{def:gamma12}-\eqref{gamma1-gamma2} and the fact
${\mathcal R}_{[+,+]}^h=-{\mathcal R}_{[-,-]}^h$ imply that
$\gamma_1$ and $\gamma_2$ have the same maximum and minimum over
their domains of definition.

Denote $\xi_0,\xi_0^r$ the points in resonance set ${\mathcal
R_{+,+}}=\{\xi_2,\xi_3\}$ such that
\be\label{cho-xi0}\mu(\xi_0)=\min\{\mu(\xi_2),\mu(\xi_3)\},\quad\mu(\xi_0^r)=\max\{\mu(\xi_2),\mu(\xi_3)\}.\ee
This implies $\g_1(\xi_0)=\g_2(-\xi_0)\geq
\g_1(\xi_0^r)=\g_2(-\xi_0^r)>0$. For $h$ small, by the continuity of
$\g_1$ and $\g_2,$ there holds the lower bound: \be\label{lb-gamma}
\inf_{{\mathcal R}_{[+,+]}^h}\g_1(\xi):=\inf_{{\mathcal
R}_{[-,-]}^h}\g_2(\xi)\geq \g_1(\xi_0^r)/2>0. \ee

By  \eqref{est-g} and \eqref{new-structure},
$\d_t g (0,x)$ is not null, and is continuous and decaying at
spatial infinity by Sobolev embedding. Then there
exists $x_0\in\R$ such that \be |\d_t g(0,x_0)|=\sup_{x\in
\R }|\d_tg(0,x)|>0.\nn\ee
Then, $\Gamma_1$ introduced as the square root of the left-hand side of \eqref{def:gamma1} is actually
\be\label{def:gamma1-0} \Gamma_1= \sup_{x\in \R}|\d_tg(0,x)|
\sup_{\xi\in\{\xi_2,\xi_3\}}(\g_1(\xi))^{1/2}= |\d_t
g(0,x_0)|(\g_1(\xi_0))^{1/2}>0,
\ee
where the computation of $\d_t g(0,x)$ (and furthermore $\Gamma_1$) does not require any knowledge of the solution, only a knowledge of the
datum and the equations.

We consider extensions of the frequency cut-off functions in
\eqref{cut-off1}, by choosing $\chi_{[i,j]}^{(0)}$ and
$\chi_{[i,j]}^{(1)}$ in $C_c^\infty\big({\mathcal
R}_{[i,j]}^h\big)$ for any $(i,j)\in\{(+,0),(-,0),(+,+),(-,-)\}$
such that
\be
\chi_{[i,j]}^{(0)}\big|_{\supp\, \chi_{[i,j]}}=\chi_{[i,j]}^{(1)}\big|_{\supp\,
\chi_{[i,j]}^{(0)}}=1.
\nn\ee
 We then define the sums for $\vartheta\in\{0,1\}$:
$$
\chi:=\chi_{[+,+]}+\chi_{[-,-]}+\chi_{[+,0]}+\chi_{[-,0]},\quad \chi_\vartheta:=\chi_{[+,+]}^{(\vartheta)}+\chi_{[-,-]}^{(\vartheta)}+\chi_{[+,0]}^{(\vartheta)}+\chi_{[-,0]}^{(\vartheta)}.
$$
 We also consider the space cut-off functions $\varphi,\varphi_0,\varphi_1 \in C^\infty_c(\R_x)$, such that $\varphi \equiv 1$ in a
neighborhood of $x_0,$ and $\varphi_0|_{\supp\,\varphi}=\varphi_1|_{\supp\,\varphi_0}\equiv  1.$ We will choose and fix $\varphi$ later on such that \eqref{choi-phi} is satisfied.

For any $(i,j)\in\{(+,+),(-,-),(+,0),(-,0)\}$, we let
\be\label{V-loc0}V_{i,j}:=\op_\e\big(\chi_{[i,j]}^{(0)}\big)\big(\varphi_0 \check U\big),
~ W_1:= \op_\e\big( \chi_0\big)\big((1-\varphi_0)\check
U\big), ~W_2:= \big(1-\op_\e(\chi_0)\big)\check U
\ee so that
\begin{equation}
 \check U = V_{+,+}+V_{-,-}+V_{+,0}+V_{-,0} + W_1 + W_2.\nn
 \end{equation}
Similarly as Lemma 3.8 in \cite{em4}, we have
\begin{lem} \label{lem:sys-V} The system in ${\bf V}:=(V_{+,+},V_{-,-},V_{+,0},V_{-,0},W_1,W_2)$ is
\begin{equation}\label{V11}
\left\{ \begin{aligned}
         \d_t V_{i,j} + \frac{1}{\e^{3/4}} \op_\e^\psi(M_{i,j}) V_{i,j} & = F_{V_{i,j}}, \\
        \d_t W_1 + \frac{i}{\e^{3/4}} \op_\e\big(\mathcal{A}) W_1 & =
        \op_\e({\cal B}_1)
         W_1 +  F_{W_1}, \\
      \d_t W_2 + \frac{i}{ \e^{3/4} } \op_\e(\mathcal{A}) W_2 & = F_{W_2},
                          \end{aligned} \right.
\end{equation}
with symbols \be\label{M-symbols}
\begin{aligned}&M_{i,j}:=i\chi_{[i,j]}^{(1)}{\cal A}-\sqrt\e\, \chi_{[i,j]}
\left(\begin{array}{c|c}
0 & B_{i,j}^{(1)}\\
\hline B_{i,j}^{(2)}&0
\end{array}\right)(\e^{1/4}t), \quad {\cal B}_1 :=
(1-\varphi) \chi_1 \,t\, \d_t \check {\cal
 B}(0,x,\xi),\end{aligned}
 \ee
where $B_{i,j}^{(1)}$ and $B_{i,j}^{(2)}$ are defined in
\eqref{B12-ij} for which we introduce
$B_{+,0}^{(2)}=B_{-,0}^{(2)}:=0$. The source term $F_{\bf V}:=(F_{V_{i,j}},F_{W_1},F_{W_2})$ satisfies for $\a\in\N^d$ with $|\a|\leq d/2+d+2+(q_0+3)/4$:
\begin{equation} \label{source:V-W}
 |(\e\d_x)^\a (F_{\bf V}(t)|_{L^2_x} \leq C \e^{\kappa-1/4}|(\dot u,\dot v)(\e^{1/4}t)|_{L^\infty_x} |(\e\d_x)^\a{\bf V}(t)|_{L^2_x } + C \e^{K_a+1/4-\kappa }.
 \end{equation}
\end{lem}

\subsection{Duhamel representation and an upper bound}

Our goal in this Section is to write an integral representation formula
for $V_{i,j}$ by using Theorem \ref{th:duh} from
Appendix \ref{app:duh}, then derive an upper bound for $|{\bf V}|_{L^2}$.

Assumption \ref{ass:B} of Theorem \ref{th:duh} is satisfied: support
because of $\varphi_1$ and $\chi_1$; regularity and bound for
${\bf M}_{0,d}(M_{i,j})$ simply by \eqref{kg-modes} and
\eqref{est-uv0}.

For any $(i,j)\in\{(+,+),(-,-),(+,0),(-,0)\}$, the symbolic flow
$S_0^{i,j}$ of $M_{i,j}$ is defined as the solution to the
initial-value problem
\begin{equation}\label{S} \d_t S_0^{i,j}(\t;t)+{\e^{-3/4}}M_{i,j}(t) S_0^{i,j}(\t;t)=0, \qquad S_0^{i,j}(\t;\t)=\Id.
                          \end{equation}

The following proposition gives a pointwise bound for the symbolic
flow. This ensures  that Assumption \ref{ass:BS} is satisfied.
The proof is postponed to Appendix \ref{app:symb-bound}.

\begin{prop}\label{est-flow-sym} For all $0 \leq \t \leq t \leq T_1 |\ln \e|^{1/2}$ with any $T_1>0$ independent of $\e$, for all $(x,\xi)$,
for  all $\a \in \N^d$ with $|\a|\leq d+1+(q_0+3)/4$, there holds for any
$(i,j)\in\{(+,+),(-,-),(+,0),(-,0)\}$: \be  |\d_x^\a
S_0^{i,j}(\t;t)| \leq
C|\ln\e|^{\a/2}\exp(C(1+|\a|)|\ln\e|^{1/2})\exp\big((t^2-\t^2)
\g^+/2\big),\nn\ee where
\begin{equation} \label{def:g+}
 \g^+  : = |\d_t g(0,x_0)| \sup_{{\mathcal R}_{[+,+]}^h} \g_1(\xi)^{1/2}=|\d_t g(0,x_0)| \sup_{{\mathcal R}_{[-,-]}^h}
 \g_2(\xi)^{1/2}.
 \end{equation}
\end{prop}
Recall that $\g_1$ and $\g_2$ are given in \eqref{def:gamma12}. Together with \eqref{def:gamma1-0}, we have
$$\g^+\to \Gamma_1 \quad  \mbox{as} ~~h\to 0.$$
For equation $\eqref{V11}_1$, as long as \be\label{fij-in
l2}F_{V_{i,j}}\in L^\infty\big([0,t_0],H^s\big),\quad \mbox{for
some $s>d/2$},\ee Theorem \ref{th:duh} implies the representation
for all $0\leq t\leq t_0$:
\begin{equation} \label{rep:V}
 V_{i,j}(t) = \ope(S^{i,j}(0;t)) V_{i,j}(0) + \int_0^t \ope(S^{i,j}(t';t)) \tilde F_{V_{i,j}} (t')\,dt',
\end{equation}
where $S^{i,j}:= \sum_{0\leq q \leq q_0} S_q^{i,j},$ with the
leading term $S_0^{i,j}$ solution of \eqref{S}, and the
correctors $S_q^{i,j},~q\geq 1$ defined as in \eqref{resolventk}. The
order $q_0$ of the expansion is a function of $\g^+$ and $T_1,$ as
seen in \eqref{def:zeta}. The source term $\tilde F_{V_{i,j}}$ can
be expressed in terms of $F_{V_{i,j}}$ and $V_{i,j}(0)$ as in
\eqref{buff0.01}. The bound \eqref{bd-R121} implies
 \begin{equation}  | \tilde F_{V_{i,j}}(t) |_{L^2_x} \lesssim | F_{V_{i,j}}(t) |_{L^2_x} + \e^{\zeta} | V_{i,j}(0)
 |_{L^2_x},\nn
 \end{equation}
where $\zeta $ is defined in \eqref{def:zeta} and is strictly
positive. The notation $\lesssim$ is introduced in
\eqref{def:lesssim} and satisfies the property stated in Remark C.7.
 Then for $V:=(V_{+,+},V_{-,-},V_{+,0},V_{-,0})$, by Lemma \ref{lem:sys-V}, Proposition \ref{est-flow-sym} and Lemma \ref{lem:bd-actionS}, we have
\be\label{est-vi}\begin{aligned} |V(t)|_{L^2}\lesssim & e^{t^2 \g^+/2}
|V(0)|_{L^2} +   \int_0^t e^{(t^2 - t'^2) \g^+/2} \\&\Big(
\e^{\kappa-1/4} |(\dot u,\dot v)(\e^{1/4}t')|_{L^\infty}\big| \textbf{V}
(t')\big|_{L^2} + \e^{K_a+1/4 - \kappa}+\e^\zeta |V(0)|_{L^2} \Big)
\,dt'.\end{aligned}\ee

By symmetry of ${\cal A},$ we implement $L^2$ estimate in
$\eqref{V11}_2$ and $\eqref{V11}_3$. This yields for $W:= (W_1,W_2),$
$$ \d_t\big(|W(t)|^2_{L^2}\big) \leq |\op_\e({\cal B}_1)(t)|_{L^2\rightarrow L^2}
|W_1(t)|_{L^2}^2 + |(F_{W_1},F_{W_2})(t)|_{L^2} |W(t)|_{L^2}.
$$
By Proposition \ref{prop:action} and \eqref{M-symbols}, \be\label{cho-s} |\op_\e({\cal
B}_1)(t)|_{L^2\rightarrow L^2}\leq C t \sup_{|\a| \leq d} \big|
\d_x^\a \big((1 - \varphi) \d_t g(0,\cdot) \big)
\big|_{L^\infty}.\ee By our choice for the Sobolev regularity
$s>d/2+K_a+d+2+(q_0+3)/4$, we have $\d_t g(0,\cdot)\in H^{s-1}\subset
H^{d/2+d+1+(q_0+3)/4}$, then its spatial derivatives up to order $d+1$ tend to
zero at infinity by Sobolev embedding. We now choose and fix $\varphi$ such that $\varphi=1$ on $B(x_0,r_0)$ with $r_0>0$ sufficient large such that
\be\label{choi-phi}
\sup_{|\a| \leq d} \big|
\d_x^\a \big((1 - \varphi) \d_t g(0,\cdot) \big)
\big|_{L^\infty} \leq \frac{\G_1}{C}.
\ee
Then
\begin{equation} \label{bd:W}\begin{split} |W(t)|_{L^2}^2 \lesssim &|W(0)|_{L^2}^2 + \int_0^t \G_1 t'
|W(t')|_{L^2}^2 + \\&\big( \e^{\kappa-1/4} |(\dot u,\dot
v)(\e^{1/4}t')|_{L^\infty}\big| \textbf{V} (t')\big|_{L^2}+ \e^{K_a+1/4-\kappa }
\big) |W(t')|_{L^2}\, dt'.\end{split}
\end{equation}
By \eqref{est-vi} and \eqref{bd:W} and the fact $\G_1\leq \g^+$, Gronwall's inequality implies
\begin{equation} \label{upper:bold-V}\begin{split}
 |{\bf V}(t)|_{L^2}\lesssim &\exp\left({t^2 \g^+/2
 +C\e^{\kappa-1/4}\exp\big({C|\ln\e|^{1/2}}\big)|\ln\e|^{C}|(\dot u,\dot
v)(\e^{1/4}t)|_{L^\infty}}\right)\\
& ~~~~~~~~~~~\times \big( |{\bf V}(0)|_{L^2} +  \e^{K_a  +
1/4-\kappa }\big).\end{split}
\end{equation}


\subsection{Existence in logarithmical time and upper bound}

In this section, we will prove existence and uniqueness of
solution ${\bf V}$ in time of order $O(|\ln \e|^{1/2})$ to \eqref{V11}
issued from ${\bf V}(0)$ such that $\| {\bf V}(0)\|_{\e,s} \leq C
\e^{K-\kappa}$, where the semi-classical Sobolev norm is defined by
$$
\|u\|^{2}_{\e,s}:=\int (1+|\e\xi|^2)^s |\hat u(\xi)|^2d\xi.
$$
We already have the $L^2$ estimate for ${\bf V}$ in
\eqref{upper:bold-V}. Define the semi-classical Fourier multiplier
$\Lambda^s := \op_\e\big((1 + |\xi|^2)^{s/2}\big),~s>d/2$.
Denoting $(V_{i,j,s},W_{1s},W_{2s}) = \Lambda^s(V_{i,j},W_1,W_2),$ then
$$
\left\{ \begin{aligned}
 \d_t V_{i,j,s} + \frac{1}{\e^{3/4}} \op_\e^\psi(M_{i,j}) V_{i,j,s} & =-\frac{1}{\e^{3/4}}
 [\Lambda^s,\op_\e^\psi(M_{i,j})]V_{i,j}+\Lambda^s  F_{V_{i,j}},\\
        \d_t W_{1s} + \frac{i}{\e^{3/4}} \op_\e\big(\mathcal{A}) W_{1s} & =
         \op_\e({\cal B}_1)
         W_{1s} +  [\Lambda^s, \op_\e({\cal B}_1 )] W_1 + \Lambda^s F_{W_1}, \\
      \d_t W_{2s} + \frac{i}{\e^{3/4}} \op_\e(\mathcal{A}) W_{2s} & = \Lambda^s F_{W_2}.
                          \end{aligned} \right.
$$

By Proposition \ref{prop:composition} and Proposition \ref{commu},
we deduce the commutator estimes:
$$ \big|[\Lambda^s,\op_\e^\psi(M_{i,j})]V_{i,j}(t)\big|_{L^2}\leq C \e \| V_{i,j} (t)\|_{\e,s},\quad
\big|[\Lambda^s, \op_\e({\cal B}_1)] W_1(t)\big|_{L^2} \leq C \e \| W_1(t)
\|_{\e,s}.$$
By following the proof of $L^2$ estimate
\eqref{upper:bold-V}, as long as \eqref{fij-in l2} is true, we have
\begin{equation} \label{upper:bold-V1}
\begin{aligned} \|{\bf V}(t)\|_{\e,s}&\lesssim \exp\big(t^2 \g^+/2
 +C\e^{\kappa-1/4}\exp\big({C|\ln\e|^{1/2}}\big)|\ln\e|^{C}|(\dot u,\dot
v)(\e^{1/4}t)|_{L^\infty}\big)\\
&~~ \times\big( \|{\bf V}(0)\|_{\e,s} +  \e^{K_a  +
1/4-\kappa }\big).\end{aligned}
\end{equation}
Then we deduce the following existence and uniqueness result:
\begin{prop}\label{prop:L-B-1} Let ${\bf
V}(0)\in H^s$ satisfying $\|{\bf V}(0)\|_{\e,s}\leq C \e^{K-\kappa}$
with $s>d/2$, $K_a+1/4>K>d/2+1/4$ and $1/4<\kappa\leq K$.  Then there exists $\e_0>0$ such that for any
$0<\e<\e_0$, for any $\e$-independent $T < T_0^*$,
 there exists a unique solution ${\bf V} \in C^0\big([0, T |\ln \e|^{1/2}], H^s\big)$
 to \eqref{V11} issued form ${\bf V}(0)$, and there holds
\be \label{bd:VV} \| {\bf V}(t)\|_{\e,s} \lesssim \e^{K-\kappa} e^{t^2
\g^+/2}, \qquad \sup_{0\leq t\leq T|\ln\e|^{1/2}}\| {\bf
V}(t)\|_{\e,s}\leq \e^{-\kappa+d/2+1/4+\iota}, \ee
where
\be\label{def:iota}
\iota:=\frac{2(K-d/2-1/4)/\Gamma_1-T^2}{5}>0.
\ee
\end{prop}

\begin{proof}
We introduce
$$
{\bf T} :=\sup\Big\{t > 0,~{\rm ~for ~all~ }0\leq t'\leq t, {\rm
~there ~holds~}\|{\bf V}(t')\|_{\e,s}\leq \e^{-\kappa+d/2+1/4+\iota}\Big\}.
$$
By local-in-time existence and continuity, the sup is well defined.
By Sobolev embedding in semi-classical norms, there holds
 \ba \label{u-dot-L-infty}\sup_{0 \leq t \leq {\bf T}} |(\dot u,\dot
 v)(\e^{1/4}t)|_{L^\infty} &\leq C \e^{-d/2}
  \sup_{0 \leq t \leq {\bf T}} \|(\dot u,\dot v)(\e^{1/4}t)\|_{\e,s}\\
  &\leq C \e^{-d/2}
  \sup_{0 \leq t \leq {\bf T}} \|{\bf V}(t)\|_{\e,s} \leq C \e^{-\kappa+1/4+\iota}.
  \ea
For $0\leq t\leq {\bf T}$, equation \eqref{fij-in l2} holds true by
estimate \eqref{source:V-W}. Then upper bound
\eqref{upper:bold-V1} holds true for such time. Moreover, with \eqref{u-dot-L-infty}, there holds for $0\leq t\leq {\bf T}$:
\begin{equation}
\begin{aligned}
 \|{\bf V}(t)\|_{\e,s}&\lesssim \exp\left(t^2 \g^+/2
 +C\e^{\iota}\exp\big({C|\ln\e|^{1/2}}\big)|\ln\e|^{C}\right)\times \big( \e^{K-\kappa} +  \e^{K_a  + 1/4-\kappa }\big).\end{aligned}\nn
\end{equation}
Since $\iota>0$, we have for $\e$ small:
\begin{equation} \label{upper:bold-V3}
 \|{\bf V}(t)\|_{\e,s}\lesssim \e^{K-\kappa} e^{t^2 \g^+/2} .
\end{equation}
Recall $\g^+\to \Gamma_1$ as $h\to 0$. Then for $h$ small, we have  $T^2 \leq
2(K-d/2-1/4)/\g^+-4\iota$. Then by
\eqref{upper:bold-V3} we obtain for $0\leq t\leq T|\ln\e|^{1/2}$,
provided $T|\ln\e|^{1/2}\leq {\bf T} $:
\begin{equation}
 \|{\bf V}(t)\|_{\e,s}\lesssim \e^{K-\kappa }
 \e^{-K+d/2+1/4+2\iota}=\e^{-\kappa+d/2+1/4+2\iota }.\nn
\end{equation}
By Remark C.7, for $\e$ small,  there holds
\begin{equation}
 \|{\bf V}(t)\|_{\e,s}\leq \e^{-\kappa+d/2+1/4+3\iota/2 }=\e^{\iota/2}\e^{-\kappa+d/2+1/4+\iota }.\nn
\end{equation}
 Finally, the classical continuation argument implies the existence time ${\bf
T}\geq T|\ln\e|^{1/2}$.

\end{proof}


\subsection{Lower bound} \label{sec:lower}

For the initial datum \eqref{id-KG-pert} and the
definition of perturbation \eqref{def:dot-u}, we have \be
\label{datum-dot-u}
 \dot u(0,x)  = \e^{K-\kappa } \phi_1^\e (x),\quad   \dot v(0,x)  = \e^{K-\kappa } \phi_2^\e
 (x).
 \ee
We choose here $\phi_1^\e $ and $\phi_2^\e $ are of the forms
\be\label{datum-dot-u2}
 \phi_1^\e (x):=e^{i x (\xi_0 + 3k)/\e} \Psi(x) (\vec e_{0},0,0) \quad   \phi_2^\e
 (x):=0
 \ee
 where $\xi_0$ is determined by \eqref{cho-xi0}, $\Psi$ is a spatial cut-off function with small support around $x_0,$ such
that $\Psi(x_0) = 1$ and $\varphi \Psi\equiv \Psi$ where $\varphi$
is the spatial cut-off function introduce in Section \ref{s-f loc},
and
 \be\label{choi-e0} \vec e_{0} \in \mbox{Image}\,P_+(\xi_0 + 3k) F(\vec e_3) Q_+(\xi_0) G(\vec e_{-3}) P_+(\xi_0 +
 3k).\ee

 Since $\rank\, P_+(\xi_0 + 3k)=\rank\,P_+(\xi_0 + 3k) F(\vec e_3) Q_+(\xi_0) G(\vec e_{-3}) P_+(\xi_0 +
 3k)=1,$  \eqref{choi-e0} is equivalent to $\vec e_{0} \in \mbox{Image}\,P_+(\xi_0 +
 3k)$. More precisely, we choose $\vec{e}_0$ such that $P_+(\xi_0 +
 3k)\vec{e}_0=\vec{e}_0.$

\begin{lem}\label{lem:choi-ini}  With the choices in \eqref{datum-dot-u}, \eqref{datum-dot-u2} and \eqref{choi-e0}, the initial value of $V_{+,+}$ satisfies
$$ V_{+,+}(0,x) = \e^{K-\kappa } \Big( e^{i x  \xi_0/\e} \Psi(x) \vec e_0, \, 0, \, \dots, \, 0 \Big) + \e^{K-\kappa  + 1/2} \tilde V_\e(x)$$
for some $\tilde V_\e(x)$ such that $\displaystyle{\sup_{\e > 0} |\tilde V_\e|_{L^2} < \infty.}$
\end{lem}

\begin{proof} By \eqref{de-U}, \eqref{datum-dot-u} and \eqref{datum-dot-u2}, the datum for $U=(u_+,u_-,u_0,v_+,v_-,v_0)$ is
 $$ \left\{\begin{aligned}&u_+(0,x)=\e^{K-\kappa }e^{ix\xi_0/\e}P_+(\e\D_x+\xi_0+3k) (\Psi (x)\vec e_0), \\
& u_-(0,x)=\e^{K-\kappa }e^{i(x\xi_0+6kx)/\e}P_-(\e\D_x+\xi_0+3k) (\Psi (x)\vec e_0),\\
& u_0(0,x)=\e^{K-\kappa }e^{(ix\xi_0+3kx)/\e}P_0(\e\D_x+\xi_0+3k) (\Psi (x)\vec e_0),\\
&(v_+,v_-,v_0)(0,x)=0.
 \end{aligned}\right. $$
We compute  $$P_+(\e\D_x+\xi_0+3k)=P_+(\xi_0+3k)+\e \int_0^1 (\d_\xi
P_+) (s\e \D_x+\xi_0+3k) ds.
$$
The choice of $\vec e_0$ such that $P_+(\xi_0 + 3k) \vec e_0 = \vec e_0$ gives us
$$
u_+(0,x)=\e^{K-\kappa }e^{ix\xi_0/\e}(\Psi (x)\vec e_0)+\e^{K-\kappa +1}
\tilde u_+(x),
$$
where
$$
\tilde u_+(x)=e^{ix\xi_0/\e}\int_0^1 (\d_\xi P_+) (s\e
\D_x+\xi_0+3k) (\Psi (x)\vec e_0) ds \in L^2 \mbox{~~bounded uniformly in
$\e$}.
$$
By similar calculation as above and the orthogonality of the eigenprojectors,
we have:
$$
(u_-,u_0)(0,x)=\e^{K-\kappa +1} (\tilde u_-(x),\tilde u_0(x))
$$
with $(\tilde u_-(x),\tilde u_0(x))$ uniformly bounded with respect to $\e$ in $L^2$.

By \eqref{def:check-u}, we can write
$$
\check U(0,x)=U(0,x)+\e^{1/2} \tilde U(x)$$ with $ \tilde U(x)$ uniformly bounded in $L^2$ with respect to $\e$. Then by \eqref{V-loc0}, the initial value of $V_{+,+}$  appears as
 $$ \begin{aligned} V_{+,+}(0,x) & = \e^{K-\kappa } \Big( \op_\e(\chi^{(0)}_{[+,+]}) \big( \varphi e^{i x  \xi_0/\e} \Psi\big) \vec e_0, 0, \dots, 0 \Big)
 +\e^{K-\kappa +1/2}\tilde V(x)
 \end{aligned}$$
 with $\tilde V(x)$ uniformly bounded in $L^2$ with respect to $\e$. By our choice such that   $\Psi \varphi\equiv
 \Psi$, and the fact $\chi^{(0)}_{[+,+]}(\xi_0)=1$ we have
 $$
 \op_\e(\chi^{(0)}_{[+,+]}) \big( e^{i x  \xi_0/\e}
 \Psi\big)=e^{i x  \xi_0/\e} \chi^{(0)}_{[+,+]}(\e \D_x+\xi_0) \Psi=e^{i x
 \xi_0/\e}\Psi+\e \tilde \Psi
 $$
with $\tilde \Psi$ uniformly bounded in $L^2$ with respect to $\e$. This
completes the proof.

\end{proof}

\begin{lem}\label{lem-S0}
 For the datum $V_{+,+}(0,\cdot)$ described in Lemma \ref{lem:choi-ini}, there holds for some small $\rho>0$, for some small $c > 0$ and some large $C > 0:$
 \begin{equation}
\begin{aligned}
&\big|\ope(S^{+,+})(0;t) V_{+,+}(0,\cdot) \big|_{L^2(B(x_0, \rho))} \\
&~~\geq\e^{K-\kappa } \Big(c\,\rho^d\,\exp\big({t^2 \g^-/2}\big) -C \e^{1/4}
|\ln \e|^{C} \exp\big({C|\ln\e|^{1/2}}\big) \exp\big({t^2
\g^+/2}\big)\Big),
\end{aligned}\nn
\end{equation}
where \be \g^- := \min_{|x - x_0| \leq \rho} |\d_t
g(0,x)| \,\big(\g_{1}(\xi_0)\big)^{1/2}.\nn\ee
\end{lem}
By \eqref{def:g+} and \eqref{def:gamma1-0}, there holds
  \be\label{conv-g+-g-}
\lim_{\rho\to 0+} \g^-=\lim_{h\to 0+} \g^+=\Gamma_1,
  \ee
\begin{proof}To simplify the notation, in this proof, we omit the index $(+,+)$ by
denoting $S=S^{+,+}$ and $V=V_{+,+}$. We let $V_0 := \Psi(x) \big(
\vec e_{0}, \, 0, \,0, \,0, 0,\,0\big).$ By Remark \ref{rem:para},
 $$ \ope(S(0;t)) \big( e^{i x\xi_0/\e}  V_0\big) = e^{i x \cdot \xi_0/\e} \int e^{i x \cdot \xi} {\bf S}(0;t,x,\xi_0 + \e \xi)  \, \hat V_0(\xi)  \, d\xi,
 $$
 where
 $ {\bf S}(0;t,x,\xi) := \displaystyle{\left({\cal F}^{-1} \psi \star \tilde S(0;t) \right)\left(\frac{x}{\e}, \xi\right)},$ with $\tilde S(x,\xi) := S(\e x,\xi).$ Direct calculation gives
\begin{equation}
  e^{-i x \cdot\xi_0/\e} \ope(S(0;t)) \big( e^{i x \cdot\xi_0/\e}  V_0(x)\big)  =  S(0;t,x,\xi_0) V_0(x) + \e^{1/4} \tilde V_0(t,x),
 \nn\end{equation}
where \be\label{tilde v0} |\tilde V_0(t,\cdot)|_{L^2} \lesssim
e^{t^2\g^+/2}. \ee

 We now show a lower bound for the leading term $S(0;t,x,\xi_0) V_0(x).$ By \eqref{m01}, \eqref{eq:widetildeS2}, \eqref{eq:tildeS0}, the construction
 of solution operator $\op_\e^\psi(S)$ in Appendix \ref{app:duh} (specifically, by definition of $S$ just above Lemma
 \ref{lem:duh-remainder} and
the upper bounds for the correctors in Lemma \ref{lem:bd-actionS}), and the fact that the frequency cut-off functions $\chi_{[+,+]},~\chi_{[+,+]}^{(0)}$ and $\chi_{[+,+]}^{(1)}$ are of value  one near $\xi_0$, we have
  \be\label{dec-s-r} S(0,t,x,\xi_0) =  S_r (0,t,x,\xi_0) + \e^{1/4}
 \widetilde S_r(0,t,x,\xi_0),\quad |\widetilde
  S_r(0,t,x,\xi_0)|\lesssim e^{t^2\g^+/2},
  \ee
where the block diagonal matrix \be\label{def:sr} S_r(0;t,x,\xi_0) =
\mbox{diag}\,\Big(\widetilde S_0(0;t,x,\xi_0), \, e^{-\e^{-3/4}i t
\l_2(\xi_0)}\Id_3,\Id_3\Big)\ee with  $\widetilde S_0(0;t,x,\xi_0)$ the solution to \eqref{eq:tildeS0} at
point $(0;t,x,\xi_0)$. At $\xi = \xi_0,$ there holds $\l_1 = \mu,$  then the solution
$\widetilde S_0$ for \eqref{eq:tildeS0} has the following explicit
expression:
 $$ \widetilde S_0(0;t,x,\xi_0) = \exp\big({-\e^{-3/4}i\l_1t}\big)\exp\big({t^2 \varphi_1(x)|\d_t g (0,x)|\widetilde M_0/2}\big),$$
 where $$\displaystyle{\widetilde M_0(x,\xi_0) := \left(\begin{array}{cc} 0 &  \frac{\d_t g (0,x)\tilde b_1(\xi_0)}{|\d_t g (0,x)|}
 \\ \frac{\d_t \bar g (0,x)\tilde b_2(\xi_0)}{|\d_t g (0,x)|}& 0
 \end{array}\right)},\quad \left\{\begin{aligned} &\tilde b_1(\xi):=P_+(\xi+3k)F(\vec e_{3})Q_+(\xi),\\
 &\tilde b_2(\xi):=2Q_+(\xi)G(\vec e_{-3})P_+(\xi+3k), \end{aligned}\right.$$
where we forcibly define $\frac{\d_t  g (0,x)}{|\d_t g (0,x)|}=\frac{\d_t \bar g (0,x)}{|\d_t g (0,x)|}=1$ if $|\d_t g (0,x)|=0$.
We then rewrite
\be\label{rewrite-M0}
\widetilde M_0= \widetilde G(x)  \left(\begin{array}{cc} 0 &  \tilde b_1(\xi_0) \\ \tilde b_2(\xi_0)& 0 \end{array}\right)\widetilde G(x)^{-1},~ \widetilde G(x):=\left(\begin{array}{cc}  \frac{\d_t g (0,x)}{|\d_t g (0,x)|}{\rm Id} &0
 \\ 0& {\rm Id}
 \end{array}\right),
\ee
where $\widetilde G(x)$ is unitary for all $x$.
By Lemma \ref{lem-det}, since $\rank\,\tilde b_{1} \leq 1,~ \rank \tilde
b_{2}\leq 1,$  the eigenvalues of $\widetilde M_0$ are $0,\pm
\big(\tr\,\tilde b_{1} \tilde b_{2} (\xi_0)\big)^{1/2}$. By
\eqref{def:gamma12}, \eqref{gamma1-gamma2} and \eqref{gamma1>0}, we
have $\tr\,\tilde b_{1} \tilde b_{2}(\xi_0)=\g_1(\xi_0)>0$ implying
that the eigenvalues of $\widetilde M_0 $ are strictly separated. Moreover, again by $\rank\,\tilde b_{1} \leq 1,~ \rank \tilde
b_{2}\leq 1,$ we can conclude that the geometric multiplicity and the algebra multiplicity for $\widetilde M_0$ are always equal.
Then we have the smooth spectral decomposition
$$
\widetilde M(x,\xi_0) = \big(\tr\, \tilde b_{1}\tilde
b_{2}(\xi_0)\big)^{1/2} \big( \Pi_{+} - \Pi_{-}\big)(x,\xi_0),$$
and there holds for all $x$:
$$
|\Pi_+(x,\xi_0)|+|\Pi_-(x,\xi_0)|\leq C.
$$
The eigenspace associated with the positive eigenvalue
$\big(\tr\,\tilde b_{1} \tilde b_{2}\big)^{1/2}$ is of dimension one,
described by vectors $\big(r_{12}, \, \big(\tr\,\tilde b_{1} \tilde
b_{2}(\xi_0)\big)^{-1/2} \tilde b_{2}(\xi_0) r_{12}\big)$ with
$r_{12}\in {\rm Image}\,\tilde b_{1}\tilde  b_{2}(\xi_0)={\rm Image}\,P_+(\xi_0+3k).$ Then
by our choice \eqref{choi-e0} for $\vec e_0\in {\rm Image}\,P_+(\xi_0+3k)$, there holds
  $\Pi_+(x,\xi_0) (\vec e_0,0) \neq 0$ for all $x$. This gives for $x,~|x-x_0|\leq \rho$ with $\rho$ small such that $\Psi(x) = 1$:
\ba\label{low-1} \big| \widetilde S_0(0;t,x,\xi_0) (\Psi(x)\vec e_0) \big| &\geq \exp\big({t^2
 \g^-/2}\big) \big| \Pi_+ \vec e_0 \big| -
\exp\big(-{t^2
 \g^-/2}\big)  \big| \Pi_- \vec e_0 \big|\\
 &\geq c\exp\big({t^2
 \g^-/2}\big)-C,
\ea
  for some $c>0$. Then by \eqref{def:sr} and \eqref{low-1}, we obtain
 \begin{equation}
 | S_r(0;t,x,\xi_0) V_0(x) |_{L^2_x(B(x_0,\rho))} \geq c \,\rho^d\, \exp\big({t^2 \g^-/2}\big)-C.
 \nn\end{equation}
Then \eqref{tilde v0} and \eqref{dec-s-r} imply
\begin{equation} \begin{aligned}
 | \ope(S(0;t)\big(e^{i x\xi_0/\e} V_0\big) |_{L^2(B(x_0,\rho))} &\geq c\,\rho^d\, \exp\big({t^2
 \g^-/2}\big) \\
 &~~- C \e^{1/4} |\ln\e|^{C}\exp\big({C |\ln\e|^{1/2}}\big)\exp\big({t^2
 \g^+/2}\big).\end{aligned}\nonumber
 \end{equation}

By Lemma \ref{lem-S0} where we show $V(0,\cdot)=\e^{K-\kappa }e^{i x  \xi_0/\e}V_0+ \e^{K-\kappa  + 1/2} \tilde V_\e(x)$,  we can complete the proof.
\end{proof}

\subsection{Proof of the deviation estimate \eqref{est:th3}}

Given $d/2+1/4<\kappa_0 <K<K_a+1/4$, we choose $\kappa=d/2+1/4+\iota_0 $ in \eqref{def:dot-u} where
$\iota_0 =(\kappa_0 -d/2-1/4)/2>0$. Clearly, $\kappa_0=d/2+1/4+2\iota_0>\kappa$.

With the initial datum given by
\eqref{datum-dot-u}, \eqref{datum-dot-u2} and \eqref{choi-e0}, the
assumptions in Proposition \ref{prop:L-B-1} are all satisfied. Then
for $T_0^*$ defined in \eqref{def:Tstar}, there exists a
unique solution to \eqref{V11} in time interval
$[0,T|\ln\e|^{1/2}]$ for any $T<T^*_0$, and the solution satisfies the
estimates in \eqref{bd:VV}.

From the Duhamel representation \eqref{rep:V} and Lemma
\ref{lem:bd-actionS}, we deduce the lower bound for $V_{+,+}$
 $$ \begin{aligned} 
  | V_{+,+}(t)|_{L^2}  \geq | \ope(S^{+,+}(0;t)) V_{+,+}(0) |_{L^2(B(x_0,\rho))}  -  \int_0^t  |\ope(S^{+,+}(t';t)) \tilde F_{V_{+,+}}
(t')|_{L^2}\, dt'.  \end{aligned}
  $$
By \eqref{source:V-W},  Proposition \ref{prop:L-B-1} and Lemma \ref{lem:bd-actionS}, we find
 $$ |\ope(S^{+,+}(t';t)) \tilde F_{V_{+,+}}
(t')|_{L^2} \lesssim\e^{K-\kappa} (\e^{\kappa - 1/4-d/2}+\e^{K_a+1/4-K } ) e^{t^2
\g^+/2}.$$ By Lemma \ref{lem-S0}, we have for all $0\leq t\leq
T|\ln\e|^{1/2}$ with $T<T_0^*$:
 \be \label{up:2}
   \| V_{+,+}(t)\|_{L^2} \geq  \e^{K-\kappa } e^{t^2 \g^-/2} \Big( c \rho^d- C \e^{\iota^*} |\ln \e|^{C}\exp(C|\ln\e|^{1/2}) e^{t^2 (\g^+ - \g^-)/2}\Big),
 \ee
 where $\iota^* := \min\{\kappa - d/2-1/4,K_a+1/4-K\}>0.$

 By \eqref{conv-g+-g-}, we can choose $\rho $ and $h$ small such that
 $$
(T_0^*)^2(\g^+-\g^-)/2<\iota^* ,\quad (T_0^*)^2 \g^-/2 \geq (T_0^*)^2
\Gamma_1/2-\iota_0 /2=K-d/2-1/4-\iota_0 /2.
 $$
We then choose $T$ close to $T_0^*$ such that $$T^2 \g^-/2 \geq (T_0^*)^2
\g^-/2-\iota_0 /2= K-d/2-1/4-\iota_0 .$$ Together with \eqref{up:2}, for $\e$
small, we obtain
 \be
   \| V_{+,+}(T|\ln\e|^{1/2})\|_{L^2} \geq \frac{c\rho^d}{2}\,
   \e^{d/2+1/4+\iota_0 -\kappa}= \frac{c\rho^d}{2}.\nn
 \ee
Back to the original variable and original time scaling, we have
 \be
   \big|\big((u,v)-(u^a,v^a)\big)(T\e^{1/4}|\ln\e|^{1/2}) \big|_{L^2} \geq
    \frac{c\rho^d}{2}\,\e^{\kappa}= \frac{c\rho^d}{2}\,\e^{\kappa_0 -\iota_0 }.\nn
    \ee
This gives \eqref{est:th3} by multiplying $\e^{-\kappa_0 }$ and taking
the supreme in $\e \in (0,\e_0)$.


\subsection{Proof of the deviation estimate \eqref{bd:insta-theo}}

For any $1/4<\kappa_1< 1/2$, let $\kappa=\kappa_1$ in \eqref{def:dot-u}. We work by contradiction we suppose that
\eqref{bd:insta-theo} does not hold. This provides a uniform
$L^\infty$ bound for any $T_1<T_1^*$:
 \begin{equation} \label{ap:dot-u}
  \sup_{0 < \e < \e_0} \sup_{0 \leq t \leq T_1  \e^{1/4} |\ln \e|^{1/2}} |(\dot u,\dot v)(t)|_{L^2 \cap L^\infty} < \infty.
 \end{equation}
 We use \eqref{ap:dot-u} in
\eqref{upper:bold-V}, to find the upper bound
$$
 |{\bf V}(t)|_{L^2}\lesssim \exp\big({t^2 \g^+/2
 +C\e^{\kappa-1/4}e^{C|\ln\e|^{1/2}}|\ln\e|^{C}}\big)
 \big( |{\bf V}(0)|_{L^2} +  \e^{K_a  + 1/4-\kappa }\big).
$$
Since $1/4 < \kappa < 1/2$ and $K<K_a+1/4$, by the argument in Remark C.7, for small $\e$, the above
upper bound implies
$$
 |{\bf V}(t)|_{L^2}\lesssim e^{t^2\g^+/2}\big( |{\bf V}(0)|_{L^2} +  \e^{K_a  + 1/4-\kappa }\big)\lesssim \e^{K-\kappa }e^{t^2\g^+/2}.
$$
Then from the Duhamel representation
\eqref{rep:V} and Lemma \ref{lem:bd-actionS},  we deduce \be\label{low-fin2}
|V_{+,+}(t)|_{L^2} \geq  \e^{K-\kappa }e^{t^2 \g^-/2}
\Big( c\rho^d- C\e^{\iota^* } |\ln \e|^{C}\exp({C|\ln\e|^{1/2}}) e^{t^2
(\g^+-\g_-)/2}\Big),\ee where  $\iota^* =\min\{\kappa-1/4,K_a+1/4-K\}>0.$

We now choose $h > 0$ and $\rho > 0$ small enough, $T_1$ close enough to $T_1^*$ so that $$(T_1^*)^2
(\g^+ - \g^-)/2 < \iota^*/2 ,\quad T_1^2 \g^-/2 \geq (T_1^*)^2 \Gamma_1/2
-\iota_1 =K-1/4-\iota_1 , $$ where $\iota_1:=(\kappa-1/4)/2>0$. Then for $\e$ small,
\eqref{low-fin2} implies
$$
|V_{+,+}(T_1|\ln\e|^{1/2})|_{L^2} \geq
 \frac{c\rho^d}{2}\,\e^{-\kappa+1/4+\iota_1 }=  \frac{c\rho^d}{2}\,\e^{-\iota_1 }.
$$
This implies
$$ \sup_{0 < \e < \e_0} \sup_{0 \leq t \leq T_1 |\ln \e|^{1/2}} |V_{+,+}(t)|_{L^2} = \infty,$$
 which contradicts to \eqref{ap:dot-u} because $|V_{+,+}(t)|_{L^2} \leq C |(\dot u,\dot
v)(\e^{1/4}t)|_{L^2}.$


\begin{appendix}

\section{Symbols and operators}

In this appendix, we recall some definitions and properties for the symbols and pseudo-differential (including para-differential) operators that we used in this paper. This is a short version of Section 6.1 in \cite{em4}, where one can find the details and proofs which we omit here.

Given $m \in \R,$ we denote $S^m$ the set of matrix-valued symbols with finite spatial regularity
$a \in C^{\bar s}\big(\R^d_x; C^\infty(\R^d_\xi)\big),$ such that for all
$\a \in \N^d$ with $|\a| \leq \bar s,$ for all $\b \in \N^d,$ there exists
some $C_{\a\b} > 0$ such that for all $(x,\xi),$
 $$ |\d_x^\a \d_\xi^\b a(x,\xi)| \leq C_{\a\b} \langle \xi \rangle^{m - |\b|}, \qquad \langle \xi \rangle := (1 + |\xi|^2)^{1/2}.$$
The regularity index $\bar s$ is  determined by the regularity of the approximate solution $(u^a,v^a)$. Motivated by Remark \ref{choi-s}, we let ${\bar
s}> d/2+d+1+(q_0+3)/4$ in the definition of $S^m$.

Given $a \in S^m,$ the definitions for the associated family of pseudo-differential operators and para-differential operators  in both classical and semi-classical quantization are given in Section 6.1 of \cite{em4}, and we do not repeat here; one can also check Bony \cite{Bony} and H\"omander \cite{Hom3}. We recall the following remark:
\begin{rem} \label{rem:para}  The classical symbol of $\ope(a)$ is
 $$ \begin{aligned} (x,\xi) \to & \Big( {\cal F}^{-1} \psi \star \tilde a\Big)\left(\frac{x}{\e}, \e \xi\right) = \int {\cal F}^{-1} \psi(y, \e \xi) a(x - \e y, \e \xi) \, dy.\end{aligned}$$
\end{rem}
For the following proposition, the first result is deduced from Theorem 18.8.1 in H\"ormander \cite{Hom3} and  a simple proof for the second result can be found in Hwang \cite{Hw}.
 \begin{prop}[Action] \label{prop:action} Given $a \in S^m,$ we have
 \begin{itemize} \item There holds for all $u \in L^2$ the bound
  $$| \op_\e^\psi(a) u |_{L^2}+ | \op_\e(a) u |_{L^2} \leq C_d \sum_{|\a| \leq d+1} \sup_{\xi \in \R^d} | \d_x^\a a(\cdot,\xi) |_{L^1(\R^d_x)} | u |_{L^2}.$$
 \item For all $m, s \in \R, k \in \N,$ there exits $C = C(m,s,k) > 0$ such that for all $u \in H^{s+m}_\e$ there holds
 $$ \|\op_\e(a) u\|_{s} \leq C {\bf M}^m_{d,d}(a) \|u\|_{\e,s+m},\quad\|\op_\e^\psi(a) u\|_{s} \leq C {\bf M}^m_{0,d}(a) \|u\|_{\e,s+m},$$
   where
$$\displaystyle{{\bf M}^m_{k,k'}(a): = \sup_{\begin{smallmatrix}
(x,\xi) \in \R^d \times \R^d \\ |\a| \leq k, |\b| \leq k'
\end{smallmatrix}} \langle \xi \rangle^{-(m - |\b|)} |\d_x^\a
\d_\xi^\b a(x,\xi)|}.$$
 \end{itemize}
 \end{prop}

We give two para-linearization estimates which are derived from Proposition 5.2.2 and Theorem 5.2.8 in \cite{M}.
 \begin{prop}  For any $r \in \N^*,$ $s \geq r,$ given $a \in H^s,$ there exists $C>0$ such that for all $u \in L^\infty,$
  \be  \big\| \big( a - \op^\psi_\e(a)\big) u \big\|_{\e,s} \leq C \| (\e \d_x)^r a \|_{\e,s-r} \|u\|_{L^\infty},\nn
  \ee
  and for all $u \in L^2,$
  \be  \big\| \big( a - \op^\psi_\e(a)\big) u \big\|_{\e,s} \leq C|(\e\d_x)^s a|_{L^\infty} \|u\|_{L^2},\nn \ee
 \end{prop}

We give the composition estimate which is derived from Theorem 6.1.4 of \cite{M}:
 \begin{prop}[Composition of para-differential operators]  \label{prop:composition}
 There holds for all $m_1, m_2, r \in \N^*,$
 $$ \op^\psi_\e(a_1) \op^\psi_\e(a_2) u = \op^\psi_\e\big( a_1 \sharp_\e a_2\big) u + \e^{r} R^\psi_r(a_1,a_2),$$
 with the notation
$$
 a_1 \sharp_\e a_2 = \sum_{|\a| < r} \e^{|\a|} \frac{(-i)^{|\a|}}{\a!} \d_\xi^\a a_1 \d_x^\a a_2,
$$
 and for some $ d^* \leq 2d + r + 1,$  for all $s \in \R,$ some $C = C(m_1,m_2,d,s,r) > 0,$ for all $u \in H^{s + m_1 + m_2 - r}$, there holds
 $$ \big\| R^\psi_r(a_1,a_2) u \big\|_{\e,s} \leq  C  \Big( {\bf M}^{m_1}_{0,d^*}(a_1) {\bf M}^{m_2}_{r,d^*}(a_2) + {\bf M}^{m_1}_{r,d^*}(a_1) {\bf M}^{m_2}_{0,d}(a_2)\Big) \|u\|_{\e,s + m_1 + m_2 -r}.$$
 \end{prop}

For the composition of a Fourier multiplier and a scalar function,
we have the following proposition. The proof is rather direct.
\begin{prop}\label{commu} Given any $u\in H^s,~s>0$ and any semi-classical Fourier multiplier $\s(\e D_x)$ the symbol of which satisfies
$$\s(\xi)\in C^{1}(\R^d),~~~~\|(\s,\nabla\s)\|_{L^\infty}< +\infty,$$
we have
$$
\big\|\s(\e D_x)u\|_{\e,s}\leq C\|\s\|_{L^\infty} \|u\|_{\e,s}. $$
Given any scalar function $g(x)\in H^{s+d/2+1+\eta}$ for some $\eta>0$, we have
the estimate:
$$
\big\|[\s(\e D_x),g(x)]u\big\|_{\e,s}\leq \e
C\|g\|_{H^{s+d/2+1+\eta}}\|\nabla \s\|_{L^\infty} \|u\|_{\e,s}. $$
The constant $C$ is independent of $\s$ and $g$.

\end{prop}

\section{Bounds for the symbolic flow}\label{app:symb-bound}
Our goal in Appendix B is to prove Proposition \ref{est-flow-sym}.
We first consider the case $(i,j)=(+,+)$ and we we reproduce
below the proposition.
\begin{prop}\label{est-flow-sym-app}
For all $0 \leq \t \leq t \leq T_1 |\ln \e|^{1/2}$, all $(x,\xi)$
and all $\a \in \N^d$ with $|\a|\leq d+1+(q_0+3)/4$,  the solution  to \eqref{S} with $(i,j)=(+,+)$
satisfies
$$
 |\d_x^\a S_0^{+,+}(\t;t)| \leq
C|\ln\e|^{\a/2}\exp(C(1+|\a|)|\ln\e|^{1/2})\exp\big((t^2-\t^2)
\g^+/2\big).
 $$
\end{prop}

\subsection{Preparation}\label{sec:prepar}
By \eqref{B12-ij} and \eqref{M-symbols}, the matrix $M_{+,+}$ is
\be\label{m01}
\begin{split}M_{+,+}&=\chi_{[+,+]}^{(1)} \bp i\l_1 &0 &0&-\sqrt\e b_{12}&0&0\\
0&i\l_2&0&0&0&0\\0&0&0&0&0&0\\-\sqrt\e
b_{21}&0&0&i\mu&0&0\\0&0&0&0&-i\mu&0\\0&0&0&0&0&0 \ep\in
\C^{18\times 18}
\end{split}\ee
denoting \be\label{l1-l2}\l_1(\cdot) = \l(\cdot + 3k) - 3\o, \qquad \l_2(\cdot) =
-\l(\cdot -3k)+3\o\ee  and
$$ \begin{aligned} b_{12}(t,x,\xi) & := \chi_{[+,+]}(\xi) \varphi_1(x)
 P_+(\xi + 3k) F\big(u_{0,3}(\e^{1/4}t,x)\big) Q_+(\xi) \in \C^{3 \times 3}, \\
 b_{21}(t,x,\xi) & := 2 \chi_{[+,+]}(\xi)  \varphi_1(x) Q_+ (\xi) G(u_{0,-3}(\e^{1/4}t,x)) P_+(\xi +3k) \in \C^{3 \times 3} \end{aligned}.$$
Up to a change of order for the volumes of $S_0^{+,+}$, it is equivalent to rewrite
\be\label{m01-new}
M_{+,+}=\chi_{[+,+]}^{(1)} \bp i\l_1 &-\sqrt\e b_{12}&0 &0&0&0\\-\sqrt\e
b_{21}&i\mu&0&0&0&0\\0&0&0&0&0&0\\
0&0&0&i\l_2&0&0\\0&0&0&0&-i\mu&0\\0&0&0&0&0&0 \ep\in
\C^{18\times 18}.
\ee

 By reality of $\l$ and $\mu$, and the fact that $b_{12}$ and $b_{21}$ vanish identically outside
 $\mbox{supp}\,\varphi_1 \times \mbox{supp}\, \chi_{[+,+]},$
 and that $\chi_{[+,+]}^{(1)}(\xi)=1$ for any $\xi\in\supp\,\chi_{[+,+]}$, it suffices to prove the following estimate for $(x,\xi)\in \mbox{supp}\,\varphi_1 \times \mbox{supp}\,
 \chi_{[+,+]}$:
 \begin{equation} \label{bd:widetildeS}
 |\d_x^\a \widetilde{S}(\t;t)| \leq C|\ln\e|^{\a/2}\exp\big(C(1+|\a|)|\ln\e|^{1/2}\big)\exp\big((t^2-\t^2) \g^+/2\big),
 \end{equation} where $\widetilde{S}$ solves
\be\label{eq:widetildeS} \d_t\widetilde S+\e^{-3/4}\widetilde
M\widetilde S=0,\quad \widetilde S(\t;\t)=\Id \ee with $\widetilde
M$ defined as
 \be \widetilde M := \bp
i\l_1 &-\sqrt\e b_{12} \\-\sqrt\e b_{21}&i\mu\ep \in \C^{6 \times
6}.\nn\ee

By Taylor expansion with integral form remainder, there holds
$$u_{0,3}(\e^{1/4}t)=u_{0,3}(0)+\e^{1/4}t\d_t
u_{0,3}(0)+ \e^{1/2}t^2\int_0^{1} (\d_t^2
u_{0,3})(s\e^{1/4}t,x)(1-s)ds.
$$
Recall $u_{0,3}(t,x)=g(t,x)\vec e_{3}$ and
$u_{0,3}(0,\cdot)\equiv 0$, we have
\ba\label{dec:widetildeM}
&\widetilde M=M_0+\e M_1,\qquad M_0:=\bp i\l_1 &-\e^{3/4}t \tilde b_{12}
\\-\e^{3/4}t \tilde b_{21}&i\mu\ep,\\
&~~~~|\d_{x}^\a\d_{\xi}^\b
M_1(t,\cdot)|_{L^\infty_x}\leq C t^2|\d^2_t \d_{x}^\a g|_{L^\infty_{t,x}}
\ea with
\be\label{def:tilde b}
\begin{aligned} \tilde b_{12}(x,\xi) & := \chi_{[+,+]}(\xi)
\varphi_1(x)\d_tg(0,x)  P_+(\xi + 3k) F(\vec e_{3}) Q_+(\xi) , \\
 \tilde b_{21}(x,\xi) & :=  2\chi_{[+,+]}(\xi) \varphi_1(x) \d_t\bar g(0,x)Q_+(\xi) G(\vec e_{-3}) P_+(\xi +3k).
 \end{aligned}\ee
Then we can rewrite \eqref{eq:widetildeS} in the form
\be\label{eq:widetildeS2}\d_t\widetilde S+\e^{-3/4} M_0\widetilde
S=-\e^{1/4}M_1\widetilde S,\quad \widetilde S(\t;\t)=\Id.\ee To show
the estimate \eqref{bd:widetildeS}, we start from considering the the following
simpler equation, in which the small source term in
\eqref{eq:widetildeS2} is not included:
\be\label{eq:tildeS0}\d_t\widetilde S_0+\e^{-3/4} M_0\widetilde
S_0=0,\quad \widetilde S_0(\t;\t)=\Id.\ee
\begin{prop}\label{est-flow-widetildeS0}
For all $0 \leq \t \leq t$ and all $(x,\xi)\in {\rm supp}\,\varphi_1
\times {\rm supp}\,
 \chi_{[+,+]}$,  the solution
$\widetilde S_0$ to \eqref{eq:tildeS0} satisfies
$$
 |\widetilde S_0(\t;t)| \leq C \exp\big((t^2-\t^2) \g^+/2\big)\exp\big(Ct\big).
 $$
\end{prop}
The proof is given in the next section.
We immediately have a corollary:
\begin{cor}\label{cor-widetildeS0}
For all $0 \leq \t \leq t\leq T_1 |\ln\e|^{1/2}$ and all $(x,\xi)\in
{\rm supp}\,\varphi_1 \times {\rm supp}\,
 \chi_{[+,+]}$,  the
solution $\widetilde S_0$ to \eqref{eq:tildeS0} satisfies
$$
 |\widetilde S_0(\t;t)| \leq C \exp(C|\ln\e|^{1/2})\exp\big((t^2-\t^2) \g^+/2\big).
 $$

\end{cor}

\subsection{Proof of Proposition \ref{est-flow-widetildeS0}}\label{pf-proS0}

We prove Proposition \ref{est-flow-widetildeS0} step by step in the
following subsections. We recall that it suffices to consider $(x,\xi)\in {\rm
supp}\,\varphi_1 \times {\rm supp}\, \chi_{[+,+]}$, and  we will not
repeat this restriction in the following statements in Section \ref{pf-proS0}.

\subsubsection{A rough estimate}

\begin{lem}\label{rough-est}
There holds for all $0\leq \t\leq t<\infty$:
$$
 |\widetilde S_0(\t;t)| \leq \exp((t^2-\t^2) b^+/2),\quad b^+:=\sup_{x,\xi}\big(\frac{|\tilde b_{12}|+|\tilde
b_{21}|}{2}\big).
 $$

\end{lem}

\begin{rem}
In general, $b^+$ is strictly larger than $\g^+$. Then this estimate
is worse than the estimate in {\rm Proposition
\ref{est-flow-widetildeS0}}, so we call it a \emph{rough} estimate.
\end{rem}

To prove  Lemma \ref{rough-est}, we introduce:
\begin{lem}\label{rough-est-0} Suppose $M(t)$ a continuous matrix in $\C^{n\times n}$.  The solution $y(\t;t)$ to
$$
\d_t y + M(t) y=0, \qquad y(\t;\t)=1
$$
satisfies
$$
|y(\t;t)|\leq \exp\Big(\int_\t^t \frac{|M(t')+M(t')^*|}{2}dt' \Big).
$$

\end{lem}

\begin{proof}[Proof of Lemma {\rm \ref{rough-est-0}}] We denote
$(\cdot,\cdot)$ the inner product in $\C^N$. Then
$$
\d_t(|y|^2)=\d_t(y,y)=(\d_t y,y)+(y,\d_t
y)=-\big((M+M^*)y,y\big)\leq |M+M^*|\cdot|y|^2.
$$
Gronwall's inequality implies
$$
|y(\t;t)|^2 \leq \exp\Big(\int_\t^t |M(t')+M(t')^*|dt' \Big).
$$
This completes the proof.

\end{proof}

 Lemma \ref{rough-est} can now be proved immediately: by \eqref{eq:tildeS0} and Lemma \ref{rough-est-0},
$$
|\widetilde S_0(\t;t)| \leq \exp\Big(\int_\t^t \frac{|
M_0+M_0^*|(t')}{2\,\e^{3/4}}dt' \Big)\leq \exp\Big(\int_\t^t t' b^+
dt'\Big)=\exp((t^2-\t^2) b^+/2).
$$

\subsubsection{Spectral of $M_0$}

The eigenvalues of $M_0$ play an important role in estimating
$\widetilde S_0$. We recall the following lemma in linear algebra:
\begin{lem}\label{lem-det}

 Suppose $A,B,C,D$ are $n\times n$
matrices, if A is invertible and $AC=CA$, we have
$$
\det \bp A&B \\C&D\ep=\det(AD-CB).
$$

\end{lem}

Then the eigen-polynomial of
$M_0$ is
$$
\det\bp x-i\l_1&\e^{3/4}t \tilde b_{12}\\ \e^{3/4} t\tilde
b_{21}&x-i\mu\ep=\det\big((x-i\l_1)(x-i\mu)-\e^{3/2}t^2 \tilde
b_{21}\tilde b_{12}\big).
$$
Then $x\in {\rm sp}\,(M_0)$ if and only if $(x-i\l_1)(x-i\mu)\in{\rm
sp}\,(\e^{3/2}t^2 \tilde b_{21}\tilde b_{12})$.  By the fact ${\rm
rank}\,\tilde b_{12}\tilde b_{21}\leq 1$, the only possible nonzero
eigenvalue for $\tilde b_{21}\tilde b_{12}$ is $\tr\, (\tilde
b_{12}\tilde b_{21})$. Then the eigenvalues of $M_0$ are
\be\label{eig-m} i\l_1,~~i\mu,~~\nu_{\pm}:=\frac{i}{2} \big(\l_1 +
\mu \big) \pm \frac{1}{2}
  \big( 4 \e^{3/2}t^2 \tr\,(\tilde b_{12} \tilde b_{21}) - (\l_1 - \mu)^2
  \big)^{1/2}.
\ee By \eqref{def:gamma12} and \eqref{gamma1>0}, there holds always
$\tr\,(\tilde b_{12} \tilde b_{21})\geq 0$. We consider the
following subcases related to a small number $0<c_0<1$ to be
fixed later on.

\subsubsection{The case $\tr\,(\tilde b_{12} \tilde b_{21})< c_0$.}

\begin{lem}\label{ga-large-chi small}
If\,~$\tr\,(\tilde b_{12} \tilde b_{21})< c_0$, there holds
$$|\widetilde S_0(\t;t)|\leq
\exp\big((t^2-\t^2)C\sqrt{c_0}/2\big).$$

\end{lem}

\begin{proof}
By \eqref{def:tilde b} and \eqref{def:gamma12}, there holds
$$
\tr\,(\tilde b_{12}\tilde
b_{21})=|\varphi_1(x)\chi_{[+,+]}(\xi)\d_tg(0,x)|^2\g_1(\xi).
$$

By the lower bound of $\g_1$ in \eqref{lb-gamma}, for the case
$\tr\,(\tilde b_{12} \tilde b_{21})< c_0$, there holds \be |\varphi_1(x)\chi_0(\xi)\d_tg(0,x)|\leq
\sqrt{2c_0/\g_1(\xi_0^r)}. \nn \ee This implies $|\tilde b_{12}|+|\tilde
b_{21}|\leq C\sqrt{c_0}.$ By Lemma \ref{rough-est},  the estimate in
Lemma \ref{ga-large-chi small} holds.

\end{proof}

\subsubsection{The case $\tr\,(\tilde b_{12} \tilde b_{21}) \geq c_0$ and around the coalescence locus}

By \eqref{eig-m}, the coalescence locus $\nu_+=\nu_-$ occurs if and
only if
$$
|\l_1-\mu|=2\e^{3/4} t\sqrt{\tr\,(\tilde b_{12}\tilde b_{21})}.
$$

We consider the following subset of $\supp\,\varphi_1\times
\supp\,\chi_{[+,+]}$ near coalescence locus:
$$G_1:=\big\{(x,\xi): \tr\,(\tilde b_{12} \tilde b_{21}) \geq c_0,~ \Big||\l_1-\mu|-
2\e^{3/4} t \sqrt{\tr\,(\tilde b_{12}\tilde b_{21})}\Big| \leq
c_0\e^{3/4}t \big\}.$$
\begin{lem}\label{lem:ar coale} For all $(x,\xi)\in G_{1}$, there holds for all $0\leq \t\leq t<\infty$:
$$
|\widetilde S_0(\t;t)|\leq \frac{C}{\sqrt{c_0}} \exp\big(C
c_0^{\frac{1}{30}}(t^2-\t^2)/2\big).
$$

\end{lem}

\begin{proof} We will only consider the case when $\l_1\geq \mu $. For $(x,\xi)\in G_{1}$. The case when $\l\leq \mu$ can be treated similarly. We write the
decomposition
\begin{equation}
\begin{aligned}
&M_0=\frac{i(\l_1+\mu)}{2}+ \left(\begin{array}{cc}
\frac{i(\l_1-\mu)}{2}&
-\e^{3/4}t  \tilde b_{12} \\
-\e^{3/4}t  \tilde b_{21}& \frac{-i(\l_1-\mu)}{2}
\end{array}\right) \\
&~~~=\frac{i(\l_1+\mu)}{2}+ \e^{3/4}tN_{01}+\e^{3/4}c_1(t,x,\xi)t
N_{02},
\end{aligned}\nn
\end{equation}
where $c_1(t,x,\xi):=\frac{\l_1-\mu-2\e^{3/4}t\sqrt{\tr\,(\tilde
b_{12}\tilde b_{21})}}{\e^{3/4}t}\in [-c_0,c_0]$ and
\begin{equation}
N_{01}:=\left(\begin{array}{cc} i \sqrt{\tr\,(\tilde b_{12}\tilde
b_{21})}
 & -   \tilde b_{12} \\
-  \tilde b_{21} & -i \sqrt{\tr\,(\tilde b_{12}\tilde b_{21})}
\end{array}\right),\quad N_{02}:=\left(\begin{array}{cc} \frac{i}{2}&0 \\
0 &  \frac{-i }{2}
\end{array}\right).\nn
\end{equation}

By Lemma \ref{lem-det}, the eigen-polynomial of $N_{01}$ is
$$
x^2\Big(x-i \sqrt{\tr\,(\tilde b_{12}\tilde
b_{21})}\Big)^{2}\Big(x+i \sqrt{\tr\,(\tilde b_{12}\tilde
b_{21})}\Big)^{2}.
$$
This implies its eigenvalues are $0,~ \pm i \sqrt{\tr\,(\tilde
b_{12}\tilde b_{21})}$. We now introduce the Schur decomposition:
\begin{lem}
For any matrix $A$ of order $n\times n$, there exists a unitary $Q$
and a upper triangular $T$, such that $Q^*AQ = T$. Precisely, if we denote $T = (t_{jk})_{n\times n}$, then
$t_{jk}=0$ provided $j>k$ and $t_{11},t_{22},\cdots, t_{nn}$ are the
eigenvalues of $A$.
\end{lem}
Then there exists a unitary $Q_1$ and an upper
triangular $N_{01}^{(1)}$ such that $N_{01}=Q_1^*N_{01}^{(1)}Q_1,$
where $(N_{01}^{(1)})_{jk}=0$ for $ j>k$ and
$(N_{01}^{(1)})_{jj}\in\Big\{0,\pm i \sqrt{\tr\,(\tilde b_{12}\tilde
b_{21})}\Big\}$ are eigenvalues of $N_{01}$. Let
$$
P_{01}:=\diag\{c_0^{1/2},c_0^{1/3},\cdots,c_0^{1/6},1 \},\quad P_{01}^{-1}=\diag\{c_0^{-1/2},c_0^{-1/3},\cdots,c_0^{-1/6},1 \}.
$$
Define
$
N_{01}^{(2)}:=P_{01}N_{01}^{(1)}P_{01}^{-1}
$, then \be \label{n013}\big(N_{01}^{(2)}\big)_{jk}=
\left\{\begin{aligned}
 0, & \quad  j>k,\\
\pm i \sqrt{\tr\,(\tilde b_{12}\tilde
b_{21})}~~{\rm or}~0, & \quad  j=k,\\
c_0^{\frac{1}{j+1}-\frac{1}{k+1}}\big(N_{01}^{(1)}\big)_{jk}, & \quad  j<k<6,\\
c_0^{\frac{1}{j+1}}\big(N_{01}^{(1)}\big)_{jk}, & \quad  j<k=6.\\
\end{aligned}\right.
\ee

We change the base as we define $\widetilde S_0^{(1)}:=P_{01}Q_1
\widetilde S_0.$ Then $\widetilde S_0^{(1)}$ solves
$$
\d_t \widetilde S_0^{(1)} +\e^{-3/4} M_0^{(1)}\widetilde S_0^{(1)}
=0,\quad \widetilde S_0^{(1)}(\t;\t)=P_{01}Q_1,
$$
where
$$
M_0^{(1)}:=\frac{i(\l_1+\mu)}{2}+
\e^{3/4}tN_{01}^{(2)}+\e^{3/4}c_1(t,x,\xi)t  (P_{01}Q_1)
N_{02}(P_{01}Q_1)^{-1}.
$$
By \eqref{n013} and the reality of $\l_1,~\mu$ and
$\sqrt{\tr\,(\tilde b_{12}\tilde b_{21})}$, there holds
$$
\e^{-3/4}\big|M_0^{(1)}+\big(M_0^{(1)}\big)^*\big|\leq C\,t\,\big(
c_0^{\frac{1}{30}}+c_0^{1/2}\big).
$$
By Lemma \ref{rough-est-0}, we have
$$
|\widetilde S_0^{(1)}(\t;t)|\leq|P_{01}Q_1| \exp\big(C
c_0^{\frac{1}{30}}(t^2-\t^2)/2 \big).
$$
Then
$$
|\widetilde S_0(\t;t)|= |(P_{01}Q_1)^{-1}\widetilde
S_0^{(1)}(\t;t)|\leq \frac{C}{\sqrt{c_0}}\exp\big(C
c_0^{\frac{1}{30}}(t^2-\t^2)/2 \big).
$$

\end{proof}

\subsubsection{The case  $\tr\,(\tilde b_{12} \tilde b_{21}) \geq c_0$ and around the resonances}

Resonance happens when $\l_1=\mu$. We now consider the following
subset of $\supp\,\varphi_1\times \supp\,\chi_{[+,+]}$ around the
resonances:
$$
G_{2}:=\big\{(x,\xi):\tr\,(\tilde b_{12} \tilde b_{21}) \geq
c_0,~|\l_1-\mu|\leq c_0 t \e^{3/4} \big\}\setminus G_{1}.
$$
\begin{lem}\label{lem:ar-res}For all $(x,\xi)\in G_2$,
there holds for all $0\leq \t\leq t<\infty$:
$$
|\widetilde S_0(\t;t)|\leq
\frac{C}{c_0}\exp(b^+c_0\t)\exp\big((t^2-\t^2)\g^+/2\big)\exp\big({C}(t-\t)/{c_0^2}\big).
$$

\end{lem}

\begin{proof}

For $(x,\xi)\in G_2$, we consider the decomposition:
\be  M_{0}=i\l_1+\e^{3/4}t
M_{0}^{(1)},~~M_{0}^{(1)}:=\bp 0& -\tilde b_{12}\\-\tilde b_{21}&\ i
c_2(t,x,\xi)\ep,\nn\ee
where
 $$
 c_2(t,x,\xi):=\frac{\mu-\l_1}{t\e^{3/4}}
\in [-c_0,c_0].
$$

By Lemma \ref{lem-det},  the eigenvalues of $M_{0}^{(1)}$ are
$$
0,~ic_2(t,x,\xi),~~\kappa_{\pm}=\frac{ic_2(t,x,\xi)}{2}\pm\frac{1}{2}\sqrt{-c_2^2(t,x,\xi)+4\tr\,
(\tilde b_{12}\tilde b_{21})}.
$$

Since $\tr\, (\tilde b_{12}\tilde b_{21})\geq c_0>c_0^2$, there
holds
$$|\kappa_+-\kappa_-|\geq \sqrt3 c_0,~
|\kappa_+-0|=|\kappa_--0|=|\kappa_+-ic_2(t,x,\xi)|=|\kappa_--ic_2(t,x,\xi)|\geq
\frac{\sqrt3 }{2}c_0.$$
For the eigen-spaces $\ker (M_{0}^{(1)}-0)$
and $\ker (M_{0}^{(1)}-ic_2(t,x,\xi))$, direct calculation gives
$$
\begin{aligned}
&M_{0}^{(1)} \bp w_1\\w_2\ep=0\Longleftrightarrow \tilde
b_{21}(\xi)w_1=w_2=0,\\
&\big(M_{0}^{(1)}-ic_2(t,x,\xi)\big) \bp
w_1\\w_2\ep=0\Longleftrightarrow \tilde b_{12}(\xi)w_2=w_1=0.
\end{aligned}
$$
Here, $\tr\, (\tilde b_{12}\tilde b_{21})\geq c_0>0$, there holds
$\rank \tilde b_{12}(\xi)=\rank \tilde b_{21}(\xi)=1$. Since $\tilde
b_{12}$ and $\tilde b_{21}$ are both of order $3\times 3$, we have
$$\dim \ker (M_{0}^{(1)}-0)=\dim \ker (M_{0}^{(1)}-ic_2(t,x,\xi))=2.$$
Together with $|\kappa_+-\kappa_-|\geq \sqrt3 c_0>0$, we have that,
for matrix $ M_{0}^{(1)}$, the geometry multiplicity and  the
algebra multiplicity are equal, both of which are 6. Together with the fact that the eigenvalues are all separated with a minimum distance
$\sqrt3c_0/2$, we can always digonalize $M_{0}^{(1)}$, with the spectral
decomposition: \be\label{diag-r015-2} M_{0}^{(1)}=0
\,\Pi_{00}+ic_2(t,x,\xi)\Pi_{c_2}+\kappa_+ \Pi_{+}+\kappa_-\Pi_-,\ee
where there holds for all $(x,\xi)\in G_2$:
$$|\Pi_{00}|+|\Pi_{c_2}|+|\Pi_+|+|\Pi_-|\leq C/c_0.$$

 Spectral decomposition \eqref{diag-r015-2} and the
following equality
$$
\Pi_{00}+\Pi_{c_2}+ \Pi_{+}+\Pi_-=\Id
$$
imply \be \begin{aligned}
&\Pi_+=\frac{M_{0}^{(1)}-\kappa_-}{\kappa_+-\kappa_-}+\frac{\kappa_-\Pi_{00}}{\kappa_+-\kappa_-}-\frac{(ic_2(t,x,\xi)-\kappa_-)\Pi_{c_2}}{\kappa_+-\kappa_-},\\
&\Pi_-=\frac{M_{0}^{(1)}-\kappa_+}{\kappa_--\kappa_+}+\frac{\kappa_+\Pi_{00}}{\kappa_--\kappa_+}-\frac{(ic_2(t,x,\xi)-\kappa_+)\Pi_{c_2}}{\kappa_--\kappa_+}.\end{aligned}\nn
\ee It is easy to find that $\Pi_{00}$ and $\Pi_{c_2}$ are
independent of $t$. Then by applying $\Pi_{00}$ to
\eqref{eq:tildeS0}, we obtain
$$
\d_t (\Pi_{00}\widetilde S_0)+\e^{-3/4}i\l_1 (\Pi_{00}\widetilde
S_0)=0,~\d_t (\Pi_{c_2}\widetilde S_0)+(\e^{-3/4}i\l_1+it c_2(t,x,\xi)
(\Pi_{c_2}\widetilde S_0)=0,
$$
which implies
$$
|(\Pi_{00}\widetilde S_0)(\t;t)|=|(\Pi_{00}\widetilde S_0)(\t;\t)|,\quad |(\Pi_{c_2}\widetilde S_0)(\t;t)|=|(\Pi_{c_2}\widetilde
S_0)(\t;\t)|.
$$

Applying $\Pi_{+}(t)$  and  $\Pi_{-}(t)$ to \eqref{eq:tildeS0} gives
\be\label{eq:piS0}\begin{split}&\d_t (\Pi_{+}\widetilde
S_0)+(\e^{-3/4}i\l_1+t\kappa_+ (\Pi_{+}\widetilde
S_0)=(\d_t\Pi_+)\widetilde S_0,\\&\d_t (\Pi_{-}\widetilde
S_0)+(\e^{-3/4}i\l_1+t\kappa_- (\Pi_{-}\widetilde
S_0)=(\d_t\Pi_-)\widetilde S_0. \end{split}\ee From
\eqref{diag-r015-2}, direct calculation gives \be
|\d_t\Pi_+|+|\d_t\Pi_-|\leq \frac{C}{c_0 }\big(1+\frac{1}{t}\big).
\nn \ee This means that $(\d_t\Pi_+)$ and $(\d_t\Pi_-)$ are
unbounded for $t$ near $0$. To show the upper bound, we first consider for large $t>c_0>0$; for small $t$,
we use the rough estimate in Lemma \ref{rough-est}. To be precise, we consider $\widetilde S_1(\t;t):=\widetilde S_0(\t;t+c_0).$ Then
$$\d_t\widetilde S_1(\t;t)+ \e^{-3/4}M_0(t+c_0)\widetilde S_1(\t;t)=0,\quad \widetilde S_1(\t;\t)= \widetilde S_0(\t;\t+c_0).$$
By applying the eigenprojectors, we have \be\label{est:pi00S1}
|(\Pi_{00}\widetilde S_1)(\t;t)|=|(\Pi_{00}\widetilde
S_1)(\t;\t)|,\quad|(\Pi_{c_2}\widetilde
S_1)(\t;t)|=|(\Pi_{c_2}\widetilde S_1)(\t;\t)|, \ee and
\be\label{eq:piS1}\begin{split}&\d_t (\Pi_{+}(t+c_0)\widetilde
S_1)+(\e^{-3/4}i\l_1+t\kappa_+ )(\Pi_{+}(t+c_0)\widetilde
S_1)=(\d_t\Pi_+(t+c_0))\widetilde S_1,\\&\d_t
(\Pi_{-}(t+c_0)\widetilde S_1)+(\e^{-3/4}i\l_1+t\kappa_-)
(\Pi_{-}(t+c_0)\widetilde S_1)=(\d_t\Pi_-(t+c_0))\widetilde S_1,
\end{split}\ee
where there holds
$$
|\d_t\Pi_+(t+c_0)|+|\d_t\Pi_-(t+c_0)|\leq
\frac{C}{c_0}\big(1+\frac{1}{t+c_0}\big)\leq \frac{C}{c_0^2}.
$$
Since the real parts $\Re e(\e^{-3/4}i\l_1+t\kappa_+ )\leq t\g^+,~
\Re e (\e^{-3/4}i\l_1+t\kappa_- )\leq 0$, together with equation
\eqref{eq:piS1}, we have
\be\label{est:PiS1-0}\begin{aligned}&|\Pi_{+}(t+c_0)\widetilde
S_1(\t;t)|\leq
\exp\big((t^2-\t^2)\g^+/2\big)|\Pi_{+}(\t+c_0)\widetilde
S_1(\t;\t)|\\&~~~~~~~~~~~~~~~~~~~~~~~~~~+\frac{C}{c_0^2}\int_\t^t
\exp\big((t'^2-\t^2)\g^+/2\big)|\widetilde
S_1(\t;t')|d\,t',\\
&|\Pi_{-}(t+c_0)\widetilde S_1(\t;t)|\leq |\Pi_{-}(\t+c_0)\widetilde
S_1(\t;\t)|+\frac{C}{c_0^2}\int_\t^t |\widetilde S_1(\t;t')|d\,t'.
\end{aligned}\ee
By Lemma \ref{rough-est},
$$|\widetilde
S_{1}(\t;\t)|=|\widetilde S_0(\t;\t+c_0)|\leq
\exp\big(((\t+c_0)^2-\t^2)b^+/2\big)\leq
\exp\big(b^+(c_0\t+c_0^2/2)\big).
$$
We choose $c_0$ small such that \be\label{cond:c0-1}\exp(b^+c_0^2/2)\leq 2.\ee Then
$|\widetilde S_{1}(\t;\t)|\leq 2\exp\big(b^+c_0\t\big).$ Together
with \eqref{est:pi00S1} and \eqref{est:PiS1-0}, we deduce
\be \begin{aligned}&|\widetilde S_1(\t;t)|\leq
\frac{C}{c_0}\exp\big((t^2-\t^2)\g^+/2\big)\exp(b^+c_0\t)\\&~~~~~~~~~~~~~~~+\frac{C}{c_0^2}\int_\t^t
\exp\big((t^2-t'^2)\g^+/2\big)|\widetilde S_1(\t;t')|d\,t'.
\end{aligned}\nn\ee
Gronwall's inequality gives
$$
|\widetilde S_1(\t;t)|\leq
\frac{C}{c_0}\exp(b^+c_0\t)\exp\big((t^2-\t^2)\g^+/2\big)\exp\big({C}(t-\t)/{c_0^2}\big).
$$
Back to $\widetilde S_0,$ for $t\geq \t+c_0$, we have
$$
|\widetilde S_0(\t;t)|=|\widetilde S_1(\t;t-c_0)|\leq
\frac{C}{c_0}\exp(b^+c_0\t)\exp\big((t^2-\t^2)\g^+/2\big)\exp\big({C}(t-\t)/{c_0^2}\big).
$$
For $0\leq \t\leq t\leq \t+c_0$,  Lemma \ref{rough-est} gives
$
|\widetilde S_0(\t;t)|\leq 2\exp(b^+c_0\t).
$
Then we get the estimate in Lemma \eqref{lem:ar-res}.

\end{proof}

\subsubsection{The case $\tr\,(\tilde b_{12}\tilde
b_{21})\geq  c_0$ and  away both from the coalescence locus and the
resonance}

 We consider the following subset of $\supp\,\varphi_1\times
\supp\,\chi_{[+,+]}$:
$$\begin{aligned}
&G_3:=\Big\{(x,\xi):\tr\,(\tilde b_{12}\tilde b_{21})\geq  c_0,~
|\l_1-\mu|\geq c_0\e^{3/4}t,\\
&~~~~~~~~~~~~~~~~~~~~~~\Big||\l_1-\mu|-2 \e^{3/4}t\sqrt{\tr\,(\tilde
b_{12}\tilde b_{21})}\Big| \geq c_0 \e^{3/4}t\Big\}.
\end{aligned}
$$
In this case we can always diagonalize $M_0$. A similar argument as
in the previous section gives:
\begin{lem}\label{lem:away}For all $(x,\xi)\in G_3$,
there holds for all $0\leq \t\leq t<\infty$:
$$
|\widetilde S_0(\t;t)|\leq
\frac{C}{c_0}\exp(b^+c_0\t)\exp\big((t^2-\t^2)(\g^+-c_0)/2\big)\exp\big({C}(t-\t)/{c_0^2}\big).
$$
\end{lem}

\subsubsection{Proof of Proposition
\ref{est-flow-widetildeS0}---Summary}
 We choose and fix $c_0$ small such that
$$
\exp(c_0^2/2)\leq 2, \quad\mbox{for \eqref{cond:c0-1}};\qquad C
c_0^{\frac{1}{30}}\leq \g^+,\quad\mbox{for Lemma \ref{lem:ar coale}
and Lemma \ref{ga-large-chi small}}.
$$
Then by Lemma \ref{ga-large-chi small}, Lemma \ref{lem:ar coale},
Lemma \ref{lem:ar-res} and Lemma \ref{lem:away}, for all
$(x,\xi)\in\supp\,\varphi_1\times \supp\,\chi_{[+,+]}$, there holds
\be |\widetilde S_0(\t;t)|\leq
\frac{C}{c_0}\exp(b^+c_0\t)\exp\big((t^2-\t^2)\g^+/2\big)\exp\big({C}(t-\t)/{c_0^2}\big).\nn\ee
With the new constant $C=C/c_0^2$, we obtain the estimate in
Proposition \ref{est-flow-widetildeS0}.

\subsection{Proof of Proposition
\ref{est-flow-sym-app}}\label{sec:end-s++}

To prove Proposition \ref{est-flow-sym-app}, it is sufficient to
prove \eqref{bd:widetildeS} where $\widetilde S$ is solution of
\eqref{eq:widetildeS2}. By \eqref{eq:widetildeS2} and
\eqref{eq:tildeS0}, we have
$$
\widetilde S(\t;t)=\widetilde S_0(\t;t)-\e^{1/4}\int_\t^t \widetilde
S_0(t',t) M_1 \widetilde S(\t;t')d\,t'.
$$
By Corollary \ref{cor-widetildeS0}, there holds for any $0\leq
\t\leq t\leq T_1|\ln\e|^{1/2}$: \be
\begin{aligned}|\widetilde S(\t;t)|&\leq|\widetilde
S_0(\t;t)|+\e^{1/4}\int_\t^t |\widetilde
S_0(t',t)|| M_1|| \widetilde S(\t;t')|d\,t'\\
&\leq C\exp(C|\ln\e|^{1/2})\exp\big((t^2-\t^2)
\g^+/2\big)\\&~~~+C\exp(C|\ln\e|^{1/2})\e^{1/4}\int_\t^t
\exp\big((t^2-t'^2) \g^+/2\big) |\widetilde S(\t;t')|d\,t'.
\end{aligned}\nn\ee
Let
$$
S^b(\t;t):=\frac{|\widetilde S (\t;t)|}{\exp\big((t^2-\t^2)
\g^+/2\big)}.
$$
Then \be \begin{aligned} S^b(\t;t)\leq
C\exp(C|\ln\e|^{1/2})+C\e^{1/4}\exp(C|\ln\e|^{1/2})\int_\t^t
S^b(\t;t')d\,t'.
\end{aligned}\nn \ee
For $\e$ small, there always holds\be
\e^{1/4}\exp(C|\ln\e|^{1/2})=
\e^{1/5}\exp\big(|\ln\e|^{1/2}(C-|\ln\e|^{1/2}/20)\big)\leq
\e^{1/5}. \nn\ee
Then by Gronwall's inequality,
$$
S^b(\t;t)\leq C \exp(C|\ln\e|^{1/2}) \exp\big(C\e^{1/5}\big)\leq
2C\exp(C|\ln\e|^{1/2}).
$$
Back to $\widetilde S$, \be\label{est:tildeS-0} |\widetilde S(\t;t)|\leq
2C\exp(C|\ln\e|^{1/2})\exp\big((t^2-\t^2) \g^+/2\big), \ee which is
\eqref{bd:widetildeS} with $\a=0$. To show the higher-order
estimates, we apply $\d_{x_j}$ to \eqref{eq:widetildeS2} and obtain
$$
\d_t\d_{x_j}\widetilde S+\e^{-3/4}M_0\d_{x_j}\widetilde S
=-\e^{-3/4}(\d_{x_j}M_0)\widetilde S
-\e^{1/4}(\d_{x_j}M_1)\widetilde S-\e^{1/4}M_1\d_{x_j}\widetilde S.
$$
By the definitions of $M_0$ and $M_1$ in \eqref{dec:widetildeM} and
Proposition \ref{est-flow-widetildeS0}, observing that
$\d_{x_j}\widetilde S(\t;\t)=0$, we have $$\begin{aligned}
|\d_{x_j}\widetilde S(\t;t)|&\leq
C(1+\e^{1/4})\exp(C|\ln\e|^{1/2})\int_\t^t \exp\big((t^2-t'^2)
\g^+/2\big)|\widetilde S(\t;t')|d\,t'\\
&~~~+C\e^{1/4}\exp(C|\ln\e|^{1/2})\int_\t^t \exp\big((t^2-t'^2)
\g^+/2\big)| \d_{x_j}\widetilde S(\t;t')|d\,t'.
\end{aligned}$$
By \eqref{est:tildeS-0}, the above equation implies
$$\begin{aligned} |\d_{x_j}\widetilde S(\t;t)|&\leq
C(t-\t)\exp(2C|\ln\e|^{1/2})\exp\big((t^2-\t^2)
\g^+/2\big)\\
&~~~+C\e^{1/4}\exp(C|\ln\e|^{1/2})\int_\t^t \exp\big((t^2-t'^2)
\g^+/2\big)| \d_{x_j}\widetilde S(\t;t')|d\,t'.
\end{aligned}$$
Then Gonwall's inequality gives
$$
|\d_{x_j}\widetilde S(\t;t)|\leq
C|\ln\e|^{1/2}\exp(2C|\ln\e|^{1/2})\exp\big((t^2-\t^2) \g^+/2\big).
$$

Since $\d_x^\a M_0$ and $\d_x^\a M_1$ are uniformly bounded for all $|\a|\leq d+1+(q_0+3)/4$, then
by induction, we have for any $\a\in\N^d$ with $|\a|\leq d+1+(q_0+3)/4$:
$$
|\d_{x}^\a\widetilde S(\t;t)|\leq
C|\ln\e|^{\a/2}\exp(C(1+|\a|)|\ln\e|^{1/2})\exp\big((t^2-\t^2)
\g^+/2\big).
$$

We complete the proof of Proposition \ref{est-flow-sym-app}.

\subsection{Upper bound for $S^{-,-}$}

By a similar argument as the proof of Proposition
\ref{est-flow-sym-app} and \eqref{def:g+}, we have
\begin{prop}\label{est-flow-sym-app-2}
For all $0 \leq \t \leq t \leq T_1 |\ln \e|^{1/2}$, all $(x,\xi)$
and all $\a \in \N^d$ with $|\a|\leq d+1+(q_0+3)/4$, the solution  to \eqref{S} with $(i,j)=(-,-)$
satisfies the bound
$$
 |\d_x^\a S_0^{-,-}(\t;t)| \leq
C|\ln\e|^{\a/2}\exp(C(1+|\a|)|\ln\e|^{1/2})\exp\big((t^2-\t^2)
\g^+/2\big).
 $$

\end{prop}

\subsection{Upper bound for $S^{+,0}$}
\begin{prop}\label{est-flow-sym-app-3}
For any $0<\tilde c_0<1$, all $0 \leq \t \leq t \leq T_1 |\ln
\e|^{1/2}$, all $(x,\xi)$ and all $\a \in \N^d$ with $|\a|\leq d+1+(q_0+3)/4$, the solution  to
\eqref{S} with $(i,j)=(+,0)$ satisfies the bound
$$
 |\d_x^\a S_0^{+,0}(\t;t)| \leq
\frac{1}{ \tilde c_0}|\ln\e|^{\a/2}\exp\big(\tilde b_{[+,0]}^+
\tilde c_0(t^2-\t^2)/2\big),
 $$
 where
 $$
 \tilde b_{[+,0]}^+:=\sup_{x,\xi}\big|\chi_{[+,0]}(\xi) \varphi_1(x)\d_t g(0,x)
 P_+(\xi + 3k) F(e_3) Q_0(\xi)\big|.
$$
\end{prop}

\begin{proof}
 By \eqref{B12-ij} and \eqref{M-symbols}, the matrix $M_{+,0}$ is:
\be
\begin{aligned}M_{+,0}&=\bp i\chi_{[+,0]}^{(1)} \l_1 &0 &0&0&0&-\sqrt\e b\\
0&i\chi_{[+,0]}^{(1)}\l_2&0&0&0&0\\0&0&0&0&0&0\\0&0&0&i\chi^{(1)}_{[+,0]}\mu&0&0\\0&0&0&0&-i\chi^{(1)}_{[+,0]}\mu&0\\0&0&0&0&0&0
\ep
\end{aligned}\nn\ee
denoting
$$ b(x,\xi)  := \chi_{[+,0]}(\xi) \varphi_1(x)\d_tg(0,x)
 P_+(\xi + 3k) F\big(\vec e_{3}\big) Q_0(\xi) \in \R^{3 \times 3}.$$
 By the argument in Section \ref{sec:prepar} and Section \ref{sec:end-s++}, to prove Proposition \ref{est-flow-sym-app-3} it suffices to
 prove the following lemma:
\begin{lem}\label{est-flow-s+0}For the solution to
\be\d_t\widetilde S_1+\e^{-3/4} \widetilde M_1\widetilde S_1=0,\quad
\widetilde S_1(\t;\t)=\Id,\nonumber\ee where
$$
\widetilde M_1:=\bp i\l_1&-\e^{3/4}t \tilde b\\0&0\ep,\quad \tilde
b:=\chi_{[+,0]}(\xi) \varphi_1(x)\d_tg(0,x)
 P_+(\xi + 3k) F(\vec e_{3}) Q_0(\xi),
$$
 we have the following
estimate for any $0<\tilde c_0<1$ and all $0 \leq \t \leq t<\infty$:
$$
 |\widetilde S_1(\t;t)| \leq \frac{1}{\tilde c_0} \exp\big(\tilde b_{[+,0]}^+\tilde c_0(t^2-\t^2)/2\big).
 $$

\end{lem}

\textbf{Proof of Lemma \ref{est-flow-s+0}}: Let $ P_{\tilde c_0}:=\bp \tilde c_0&0
\\ 0& 1\ep,\quad \widetilde S_2:=P_{\tilde c_0} \widetilde S_1. $ Then
$\widetilde S_2$ solves $$\d_t\widetilde S_2+\e^{-3/4} \widetilde
M_2\widetilde S_2=0,\quad \widetilde S_1(\t;\t)=\bp \tilde c_0&0
\\ 0& 1\ep,\quad \widetilde M_2:=\bp i\l_1&-\e^{3/4}\tilde c_0 t \tilde b\\0&0\ep.$$
By Lemma \ref{rough-est-0}, we have here
$$
|\widetilde S_2(\t;t)|\leq \exp\big(\tilde
b_{[+,0]}^+\tilde c_0(t^2-\t^2)/2\big).
$$
Then $$|\widetilde S_1(\t;t)|= |P_{\tilde c_0}^{-1}\widetilde
S_2(\t;t)|\leq \frac{1}{\tilde c_0}\exp\big(\tilde
b^+_{[+,0]}\tilde c_0(t^2-\t^2)/2\big).$$

\end{proof}

\subsection{Upper bound for $S^{-,0}$}

The same argument as in Section B.5 gives

\begin{prop}\label{est-flow-sym-app-4}
For any $0<\tilde c_0<1$, for all $0 \leq \t \leq t \leq T_1 |\ln
\e|^{1/2}$, all $(x,\xi)$ and all $\a \in \N^d$ with $|\a|\leq d+1+(q_0+3)/4$, the solution  to
\eqref{S} with $(i,j)=(-,0)$ satisfies the bound
$$
 |\d_x^\a S_0^{+,0}(\t;t)| \leq
\frac{1}{ \tilde c_0}|\ln\e|^{\a/2}\exp\big(\tilde b_{[-,0]}^+
\tilde c_0(t^2-\t^2)/2\big),
 $$
 where
 $$
 \tilde b_{[-,0]}^+:=\sup_{x,\xi}\big|\chi_{[-,0]}(\xi)
 \varphi_1(x)\d_t \bar g(0,x)
 P_-(\xi - 3k) F(\vec e_{-3}) Q_0(\xi)|.
$$
\end{prop}

\subsection{Proof of Proposition \ref{est-flow-sym}}
Proposition \ref{est-flow-sym} is concluded by Proposition \ref{est-flow-sym-app}, Proposition
\ref{est-flow-sym-app-2}, Proposition \ref{est-flow-sym-app-3} and
Proposition \ref{est-flow-sym-app-4}, with the choice $\tilde c_0$ in Proposition \ref{est-flow-sym-app-3} and Proposition \ref{est-flow-sym-app-4} small so that $\tilde c_0\,\tilde b^+_{[\pm,0]}\leq \g^+$.

\section{An integral representation formula} \label{app:duh}

 We adapt to the present context an integral representation formula introduced in \cite{dua}. Consider the initial value problem
 \begin{equation} \label{buff01}
  \begin{aligned} \d_t u + \frac{1}{\e^{3/4}} \ope(M) u  = f, \qquad u(0) = u_0,\end{aligned}
  \end{equation}
 where
 $u_0 \in L^2(\R^d),$ $f \in L^\infty([0,T_1 |\ln \e|^{1/2}],L^2(\R^d)),$ for some $T_1 > 0$ independent of $\e$.
We assume that $M = M(\e,t,x,\xi)$ is a matrix-valued, time-dependent symbol that satisfies the following assumption:

 \begin{ass} \label{ass:B} For some fixed $R_1,R_2 > 0,$ for all $t$ and $\e > 0,$
 $M$ satisfies
$$
M=0,~\mbox{for $|\xi|\geq R_1$};\quad M=\diag\,
\{i\l_1(\xi),\cdots,i\l_N(\xi)\},~\mbox{for $|x|\geq R_2$},
$$
where $\l_1,\cdots,\l_N$ are smooth real  valued scalar functions
dependent only on $\xi$. We also suppose that $M$ satisfies the
bound
$$ \sup_{0 \leq t \leq T_1 |\ln \e|} \sup_{(x,\xi) \in \R^{2d}} \big| \d_x^\a \d_\xi^\b M(\e,t,x,\xi) \big| \leq C_{\a,\b}, \quad \a, \b \in \N^d,~|\a|\leq d/2+d+1+(q_0+3)/4.$$
\end{ass}

For the flow $S_0$ of $\e^{-3/4} M,$ defined for $0 \leq \t \leq t
\leq T_1 |\ln \e|^{1/2}$ by
 \begin{equation} \label{resolvent0}\d_t S_0(\t;t) + \frac{1}{\e^{3/4}} M S_0(\t;t) = 0, \qquad S_0(\t;\t) = \mbox{Id},\end{equation}
we assume an exponential growth in time:
 \begin{ass}\label{ass:BS}
There holds for some $\g^+ > 0,$ all $0 \leq \t \leq t \leq T_1
|\ln \e|^{1/2},$ and all $\a\in \N^d$ with $|\a|\leq d+1+(q_0+3)/4$: \be |\d_x^\a S_0(\t;t)| \leq
C|\ln\e|^{\a/2} \exp(C(1+|\a|)|\ln\e|^{1/2})\exp\big((t^2-\t^2)
\g^+/2\big).\nn\ee
\end{ass}

We introduce correctors $\{ S_q\}_{1 \leq q \leq q_0}$, with $q_0$
large determined by \eqref{def:zeta}, defined
  as the solutions of
 \begin{equation} \label{resolventk}
  \d_t S_q + \frac{1}{\e^{3/4}} M S_q + \sum_{1 \leq |\a| \leq [(q+3)/4]} \frac{(- i)^{|\a|}}{|\a|!}  \d_\xi^\a M \d_x^\a S_{q+3-4|\a|} = 0, \quad S_{q}(\t;\t) = 0.
  \end{equation}
Then we have the following lemma:
 \begin{lem} \label{lem:bd-S} There holds, for all
 $q \in [0,q_0],$ all $\a\in \N^d$ with $|\a|\leq d+1$, all $\b \in \N^d,$ all $0 \leq \t \leq t \leq T_1 |\ln \e|^{1/2},$ the bounds
\be\label{boud-corr}  \begin{aligned} |\d_x^\a\d_\xi^\b  S_q(\t;t) |
&\leq
C\e^{-3|\b|/4}|\ln\e|^{|\a+\b+q|/2}\\
&\quad\times
\exp\big(C(1+|\a+\b|+q)|\ln\e|^{1/2}\big)\exp\big((t^2-\t^2)
\g^+/2\big).\end{aligned}\ee
 \end{lem}

 \begin{proof}   By \eqref{resolvent0} and \eqref{resolventk}, there holds for $q \geq 1$
 \begin{equation} \label{sq} S_q(\t;t) =
 \sum_{1 \leq |\a| \leq [(q+3)/4]} \frac{(- i)^{|\a|}}{|\a|!}  \int_{\t}^t S_0(\t';t)  \d_\xi^\a M(t') \d_x^\a S_{q+3-4|\a|}(\t;t') \, dt'.
\end{equation}
From here, we see that the bound \eqref{boud-corr} which holds true
for $\b = 0$ by Assumption \ref{ass:BS}, propagates from $q$ to
$q+1.$ By induction in $\b$, we obtain the desired result.
\end{proof}

By Assumption \ref{ass:B} and \eqref{resolventk}, for $x\geq R_2$,
we have
 \begin{lem} \label{lem:x>R2} If $x\geq R_2$, there holds
\be
S_q(\t;t,x,\xi)=\de_q(\t;t,\xi):=\exp\left(\frac{i(t-\t)}{\e^{3/4}}\begin{pmatrix}
\l_1(\xi)& 0 & \cdots & 0\\
0 & \l_2(\xi) & & 0\\
\vdots & &\ddots & \vdots\\
0 & 0 & \cdots & \l_N(\xi)
\end{pmatrix}\right).
\nn\ee

 \end{lem}

Then we have the operator norm
 \begin{lem} \label{lem:bd-actionS} There holds, for all
 $q \in [0,q_0],$ all $0 \leq \t \leq t \leq T_1 |\ln \e|^{1/2},$ and all $u \in L^2,$ the bound
   \begin{equation} \begin{aligned} |\op_\e^\psi(S_q(\t;t)) u |_{L^2} &\leq
C|\ln\e|^{(d+1+q)/2}\\
&\times\exp\big(C(2+d+q)|\ln\e|^{1/2}\big)\exp\big((t^2-\t^2)
\g^+/2\big)\| u \|_{L^2}.\end{aligned}\nn
 \end{equation}
 \end{lem}

\begin{proof} From Lemma \ref{lem:bd-S}, the compactness of the support of $M$ on $\xi$ and Lemma \ref{lem:x>R2}, we deduce the bound
$$ \begin{aligned}\sum_{|\a| \leq d+1} \sup_{\xi \in \R^d} \big| \d_x^\a (S_q - \delta_{q}) \big|_{L^1_x}& \leq
C|\ln\e|^{(d+1+q)/2}
\\&\times\exp\big(C(d+1+q)|\ln\e|^{1/2}\big)
\exp\big((t^2-\t^2) \g^+/2\big),\end{aligned}$$ where $\delta_{q}$ is given in Lemma \ref{lem:x>R2}, and $\op_\e(\de_q)$ is a unitary Fourier multiplier.
We then conclude by Proposition \ref{prop:action}.
\end{proof}

Let $S := \displaystyle{\sum_{0 \leq q \leq q_0} \e^{q/4} S_q.}$
The following Lemma expresses the fact that $\ope(S)$ is an
approximate solution operator:
 \begin{lem} \label{lem:duh-remainder} Under Assumptions {\rm \ref{ass:B}} and {\rm \ref{ass:BS}}, there holds
  \begin{equation} \label{tildeS1-duh} \op_\e^\psi(\d_t S) +\frac{1}{\e^{3/4}} \op_\e^\psi(M) \op_\e^\psi(S) = \rho_0,\end{equation}
  where for $0 \leq \t \leq t \leq T_1 |\ln \e|^{1/2},$
  \begin{equation} \label{tilde-remainder} \| \rho_0\|_{L^2 \to L^2} \lesssim \e^{q_0/16 - 3(d+1)/2} \exp\big( (t^2 - \t^2) \g^+/2\big),\end{equation}
  where the notation $a\lesssim b$ means \be\label{def:lesssim}a\leq C
|\ln\e|^{C} \exp(C|\ln\e|^{1/2}),\quad \mbox{for some constant
$C>0$.}\ee
  \end{lem}

\begin{rem} If $a\lesssim \e^{\zeta}b$ for some $\zeta >0$, there
holds for small $\e$:
$$
a\leq \e^{\zeta/2} b.
$$
Indeed,  $a\leq C \e^{\zeta}|\ln\e|^{C}
\exp(C|\ln\e|^{1/2})$  can be rewritten as
$$
a\leq C \e^{\zeta/2}
\big(\e^{\zeta/4}|\ln\e|^{C}\big)\exp\big(|\ln\e|^{1/2}(-\zeta
|\ln\e|^{1/2} /4+C)\big) b.
$$
This implies   $a\leq \e^{\zeta/2}b$ for $\e$ small.

\end{rem}
  \begin{proof}
  By definition of $S$ and \eqref{resolventk},
  \begin{equation} \label{c1}
   - \d_t \ope(S) = {\rm I} + {\rm II},
   \end{equation}
  with the notations
 $$ \begin{aligned}
 {\rm I} & := \sum_{0 \leq q \le q_0} \e^{(q/4-3/4)} \ope(M S_q), \\
 {\rm II} & :=  \sum_{\begin{smallmatrix} 1 \leq q \leq q_0 \\ 1 \leq |\a| \leq [(q+3)/4] \end{smallmatrix}}
 \e^{q/4} \frac{(- i)^{|\a|}}{|\a|!} \ope\Big(    \d_\xi^\a M \d_x^\a
 S_{q+3-4|\a|}\Big).
 \end{aligned}
$$
By Proposition \ref{prop:composition},
$$
 \begin{aligned} \ope(M S_q)  &= \ope(M) \ope(S_q) - \sum_{1 \leq |\a| \leq [(q_0 - q+3)/4]}
 \e^{|\a|} \frac{(-i)^{|\a|}}{|\a|!} \ope\Big(\d_\xi^\a M \d_x^\a S_q\Big) \\ &\quad - \e^{1 + [(q_0 - q+3)/4]} R^\psi_{1 + [(q_0 - q+3)/4]}(M,S_q),
\end{aligned}$$
so that
$$
{\rm I}  = \e^{-3/4} \ope(M) \ope(S) - \sum_{\begin{smallmatrix} 0
\leq q \leq q_0-1 \\1 \leq |\a| \leq [(q_0 - q+3)/4]
\end{smallmatrix}} \e^{(q/4-3/4 + |\a|)} \frac{(-i)^{|\a|}}{|\a|!}
\ope\Big( \d_\xi^\a M \d_x^\a S_q\Big) + \rho_0
$$
with \be\label{def:rho0} \rho_0 := \sum_{0 \leq q \leq
q_0-1}\e^{(1+q)/4 + [(q_0 - q+3)/4]} R^\psi_{1 + [(q_0 -
q+3)/4]}(M,S_q).\ee Changing variables in the double sum, we find
$$ {\rm I} = \e^{-3/4} \ope(M) \ope(S) - \sum_{\begin{smallmatrix}
1 \leq q' \leq  q_0  \\ 1 \leq |\a| \leq [( q'+3)/4] \end{smallmatrix}}
 \e^{q'/4} \frac{(-i)^{|\a|}}{|\a|!} \ope\Big(\d_\xi^\a M \d_x^\a S_{q' +3 - 4 |\a|}\Big) + \rho_0,$$
hence
\begin{equation} \label{c2} {\rm I} + {\rm II} = \e^{-1/4} \ope(M) \ope(S) +
\rho_0.
\end{equation}
Identities \eqref{c1} and \eqref{c2} prove \eqref{tildeS1-duh}. From
Proposition \ref{prop:composition} and Assumption \ref{ass:B}, we
deduce
 $$ \| R^\psi_{1+[(q_0-q+3)/4]}(M,S_q) \|_{L^2 \to L^2} \leq C {\bf M}^0_{1+[(q_0-q+3)/4],2d+2+[(q_0-q+3)/4] }(S_q),$$
 and with Lemma \ref{lem:bd-S} this implies
 $$ \| R^\psi_{1+[(q_0-q+3)/4]}(M,S_q) u \|_{L^2 \to L^2} \lesssim \e^{-3(2d+2+[(q_0-q+3)/4])/4} \exp\big( (t^2 - \t^2) \g^+/2\big).$$
 Then by \eqref{def:rho0}, we have
 \begin{equation}
 \| \rho_0 \|_{L^2 \to L^2} \lesssim \e^{q_0/16 - 3(d+1)/2} \exp\big( (t^2 - \t^2) \g^+/2\big).\nonumber
 \end{equation}
 \end{proof}
 We finally give the representation formula:
  \begin{theo} \label{th:duh} Under Assumptions {\rm \ref{ass:B}} and {\rm \ref{ass:BS}}, the Cauchy problem \eqref{buff01} with source
  $f \in L^\infty([0,T_1 |\ln \e|^{1/2}], L^2)$ and datum $u_0 \in L^2$  has a unique solution
  $u \in L^\infty([0,T_1 |\ln \e|^{1/2}], L^2)$ given by
  \begin{equation} \label{buff0.01} u = \op_\e^\psi(S(0;t)) u_0 + \int_0^t \op_\e^\psi(S(t';t))(\Id + \e^\zeta F_1(t')) \big( f(t') + \e^{\zeta} F_2(t') u_0\big)\, dt',
  \end{equation}
  for some $\zeta > 0,$ where for some $N(\zeta) >0,$ for all $0 \leq t \leq T_1 |\ln
  \e|^{1/2},$ there holds
  \begin{equation} \label{bd-R121}
  \| F_1(t) \|_{L^2 \to L^2} + \| F_2(t) \|_{L^2 \to L^2} \leq |\ln \e|^{N(\zeta)} \exp(N(\zeta)|\ln \e|^{1/2}).
  \end{equation}
 \end{theo}

\begin{proof}  By Lemma \ref{lem:duh-remainder} and direct calculation,  the map
  \begin{equation}  u := \op_\e^\psi( S(0;t)) u_0 + \int_0^t \op_\e^\psi( S(t';t)) g(t') \, dt'\nn
  \end{equation}
   solves \eqref{buff01} in time interval $[0, T_1 |\ln \e|^{1/2}]$ if and only if there holds for such time:
\begin{equation} \label{cond-g1}
  (\Id + r(t)) g  = f(t) - \rho_0(0;t) u_0,
   \end{equation}
 where $r$ is the linear integral operator
  $$r(t): \qquad v \to \int_0^t  \rho_0(\t;t) v(\t) \, d\t,$$
 with $\rho_0$ the remainder in Lemma \ref{lem:duh-remainder}, satisfying bound \eqref{tilde-remainder}.
 We choose the expansion index $q_0$ large enough so that
 \begin{equation} \label{def:zeta} \zeta := \frac{q_0}{16} - \frac{3d+3}{2} - \frac{T_1^2 \g^+}{2} > 0.
 \end{equation}
 Then there holds for all $0 \leq t \leq T_1 |\ln \e|^{1/2}$:
  \begin{equation} \begin{aligned} \label{bd-rf1} \sup_{0 \leq t \leq T_1|\ln\e|^{1/2}} \| r(t) v(t) \|_{L^2}
   & \lesssim \e^{\zeta} \sup_{0 \leq t \leq T_1 |\ln\e|^{1/2}} \| v(t) \|_{L^2}. \end{aligned}
  \end{equation}
By \eqref{bd-rf1}, for $\e$ small, $\Id + r(t)$ is invertible in $\mathcal{L}\big(L^\infty([0,T_1 |\ln \e|^{1/2}], L^2)\big)$ which denotes the vector space of linear operators form $L^\infty([0,T_1 |\ln \e|^{1/2}], L^2)$ to itself. Again by \eqref{bd-rf1}, $(\Id + r(t))^{-1}$ is also bounded in $\mathcal{L}\big(L^\infty([0,T_1 |\ln \e|^{1/2}], L^2)\big)$, uniformly in $\e.$ Hence, \eqref{cond-g1} can be solved in $L^\infty([0,T_1 |\ln \e|^{1/2}], L^2),$
 and we obtain \eqref{buff0.01} with
 $$ \e^\zeta F_1(t) :=  (\Id + r(t))^{-1} - \Id, \qquad \e^\zeta F_2(t) := - \rho_0(0;t).$$
Bound \eqref{bd-R121} follows from \eqref{bd-rf1} and $\| \rho_0
\|_{L^2 \to L^2} \lesssim \e^\zeta.$ Uniqueness is a direct
consequence of the linearity of equation \eqref{buff01} and
boundedness of $\ope(M)$ in ${\cal L}(L^2)$.
\end{proof}

\end{appendix}

\section*{Acknowledgements}
\thispagestyle{empty}

The author thanks Benjamin Texier for fruitful discussions about this problem and for valued and detailed comments about the manuscript.  The author acknowledges the support of the project LL1202 in the programme ERC-CZ funded by the Ministry of Education, Youth and Sports of the Czech Republic.


\end{document}